\documentclass[final]{siamart0516}

\usepackage{lipsum}
\usepackage{amsfonts}
\usepackage{graphicx}
\usepackage{algorithmic}
\usepackage{epstopdf}

\usepackage{array,makecell}
\usepackage{amssymb}
\usepackage{bm}
\usepackage{float}
\usepackage{amsmath}
\usepackage{multicol}
\usepackage{colortbl} 
\usepackage{caption}
\usepackage{hyperref}
\usepackage{multirow}
\usepackage{hhline}

\usepackage{cases}

\ifpdf
  \DeclareGraphicsExtensions{.eps,.pdf,.png,.jpg}
\else
  \DeclareGraphicsExtensions{.eps}
\fi

\makeatletter
\newsavebox\myboxA
\newsavebox\myboxB
\newlength\mylenA

\newcommand*\xoverline[2][0.75]{%
    \sbox{\myboxA}{$\m@th#2$}%
    \setbox\myboxB\null
    \ht\myboxB=\ht\myboxA%
    \dp\myboxB=\dp\myboxA%
    \wd\myboxB=#1\wd\myboxA
    \sbox\myboxB{$\m@th\overline{\copy\myboxB}$}
    \setlength\mylenA{\the\wd\myboxA}
    \addtolength\mylenA{-\the\wd\myboxB}%
    \ifdim\wd\myboxB<\wd\myboxA%
       \rlap{\hskip 0.5\mylenA\usebox\myboxB}{\usebox\myboxA}%
    \else
        \hskip -0.5\mylenA\rlap{\usebox\myboxA}{\hskip 0.5\mylenA\usebox\myboxB}%
    \fi}
\makeatother

\usepackage{tikz}

\usetikzlibrary{shapes,arrows,positioning}
\definecolor{dark_blue}{RGB}{46,87,144}
\definecolor{blue_pers}{RGB}{54,104,171}
\definecolor{grey_pers}{RGB}{245,245,245}
\definecolor{red_pers}{RGB}{213,78,33}
\definecolor{green_pers}{RGB}{194, 239, 194}

\usepackage{xcolor}

\tikzstyle{block0} = [rectangle, draw, fill=grey_pers!30, 
    text width=14em, text centered, rounded corners, minimum height=2em]
\tikzstyle{block} = [rectangle, draw, fill=grey_pers!30, 
    text width=8em, text centered, rounded corners, minimum height=2em]
\tikzstyle{blocka} = [rectangle, draw, fill=grey_pers!30, 
    text width=0.2em, text centered, rounded corners, minimum height=2em]

\tikzstyle{type} = [rectangle, draw, fill=grey_pers, 
    text width=8em, text centered, minimum height=2em]    
\tikzstyle{objection} = [rectangle, draw, 
    text width=8em, text centered, rounded corners, minimum height=2em]
\tikzstyle{ref} = [rectangle, draw, fill=grey_pers, 
    text width=3em, text centered, minimum height=2em]    
\tikzstyle{ref1} = [rectangle, draw, 
    text width=3em, text centered, rounded corners, minimum height=2em]

\tikzstyle{line} = [draw, -latex']
\tikzstyle{line_1} = [draw]
\tikzstyle{cloud} = [draw, ellipse,fill=grey_pers, node distance=3cm,
    minimum height=2.4em]
\tikzstyle{circle_blue} = [draw,circle,fill=blue_pers!30, node distance=2cm,
    minimum size=2em]
    \tikzstyle{circle_green} = [draw,circle,fill=green_pers,
     node distance=2cm, minimum size=2em]
\tikzstyle{circle_grey} = [draw,circle,fill=grey_pers, node distance=2cm,
    minimum size=2.2em] 
\tikzstyle{circle_red} = [draw,circle,fill=red_pers!30, node distance=2cm,
    minimum size=2em] 
\tikzstyle{back} = [rectangle, draw, 
    text width=5em, text centered, rounded corners, minimum height=1.5em]

\tikzstyle{back_b} = [rectangle, draw, fill=blue_pers!30, 
    text width=5em, text centered, rounded corners, minimum height=1.5em]

\tikzstyle{back_g} = [rectangle, draw, fill=grey_pers!30, 
    text width=5em, text centered, rounded corners, minimum height=1.5em]

\newcommand{\ba}{{\bf a}}
\newcommand{\boldbeta}{{\boldsymbol{\beta}}}
\newcommand{\bSigma}{{\boldsymbol{\Sigma}}}

\newcommand{\bb}{{\bf b}}

\newcommand{\br}{{\bf r}}
\newcommand{\bn}{{\bf n}}

\newcommand{\bx}{{\bf x}}
\newcommand{\by}{{\bf y}}
\newcommand{\bv}{{\bf v}}

\newcommand{\bee}{{\bf e}}

\newcommand{\bxi}{\boldsymbol{\xi}}

\newcommand{\bB}{{\bf B}}

\newcommand{\bM}{{\bf M}}

\newcommand{\bA}{{\bf A}}

\newcommand{\U}{{\textup{U}}}
\newcommand{\F}{{\textup{F}}}

\newcommand{\IN}{\mathbb{N}}
\newcommand{\IR}{\mathbb{R}}
\newcommand{\IP}{\mathbb{P}}
\newcommand{\IE}{\mathbb{E}}

\newcommand{\IV}{\mathbb{V}}
\newcommand{\IS}{\mathbb{S}}

\newcommand{\card}{\operatorname{card}}
\newcommand{\atantwo}{\operatorname{atan2}}
\newcommand{\divg}{\operatorname{div}}

\newcommand{\be}{\begin{equation}}
\newcommand{\ee}{\end{equation}}
\newcommand{\norm}[2]{\left\|#1\right\|_{#2}}

\newcommand{\Hloc}{H_{\textup{loc}}}

\newcommand{\eps}{\varepsilon}

\newcommand{\DL}{\mathsf{DL}}
\newcommand{\SL}{\mathsf{SL}}

\newcommand{\opA}{\mathsf{A}}
\newcommand{\opV}{\mathsf{V}}
\newcommand{\opW}{\mathsf{W}}
\newcommand{\SRC}{\textup{SRC}}
\newcommand{\opK}{\mathsf{K}}
\newcommand{\opI}{\mathsf{Id}}

\newcommand{\opB}{\mathsf{B}}

\newcommand{\mA}{\mathcal{A}}
\newcommand{\mD}{\mathcal{D}}
\newcommand{\mH}{\mathcal{H}}

\newcommand{\mL}{\mathcal{L}}
\newcommand{\mM}{\mathcal{M}}
\newcommand{\mO}{\mathcal{O}}
\newcommand{\mP}{\mathsf{P}}

\newcommand{\mU}{\mathcal{U}}

\newcommand{\fA}{\mathsf{A}}

\newcommand{\fB}{\mathsf{B}}

\newcommand{\fZ}{\mathsf{Z}}

\newcommand{\fC}{\mathsf{C}}

\newcommand{\fI}{\mathsf{I}}

\newcommand{\fK}{\mathsf{K}}

\newcommand{\fP}{\mathsf{P}}

\newcommand{\fR}{\mathsf{R}}

\usepackage{comment}

\usepackage[normalem]{ulem}
\definecolor{Gray}{gray}{0.2}
\definecolor{Greeen}{RGB}{14,118,17}
\definecolor{BlueGreeen}{RGB}{31,191,190}
\definecolor{Pink}{RGB}{189, 27, 189}
\definecolor{Gray1}{gray}{0.9}
\definecolor{dark_green}{RGB}{0,100,0}
\definecolor{dark_blue}{RGB}{46,87,144}

\newcommand\npe[1]{\iffalse #1 \fi}

\newcommand{\TheTitle}{
Helmholtz scattering by random domains: first-order sparse boundary element approximation}
\newcommand{\TheAuthors}{Paul Escapil-Inchausp\'e and Carlos Jerez-Hanckes}

\headers{\TheTitle}{\TheAuthors}

\title{{\TheTitle}\thanks{Submitted to the editors \today.
\funding{This work was supported in part by Fondecyt Regular 1171491.}}
}
\author{
  Paul Escapil-Inchausp\'{e}\thanks{School of Engineering, Pontificia Universidad Cat\'olica de Chile, Santiago, Chile (\email{pescapil@uc.cl}).}
  \and
 Carlos Jerez-Hanckes\thanks{Faculty of Engineering and Sciences, Universidad Adolfo Iba\~nez, Santiago, Chile (\email{carlos.jerez@uai.cl}).}}
 
\usepackage{amsopn}

\DeclareMathOperator{\diag}{diag}

\usepackage{amsfonts} 

\newsiamremark{remark}{Remark} 
\newsiamthm{problem}{Generic Problem}
\newsiamthm{subproblem}{Problem}

\ifpdf
\hypersetup{
  pdftitle={\TheTitle},
  pdfauthor={\TheAuthors}
}

\begin{document}
\date{\today}

\maketitle
\begin{abstract}
We consider the numerical solution of time-harmonic acoustic scattering by obstacles with uncertain geometries for Dirichlet, Neumann, impedance and transmission boundary conditions. In particular, we aim to quantify diffracted fields originated by small stochastic perturbations of a given relatively smooth nominal shape. Using first-order shape Taylor expansions, we derive tensor deterministic first kind boundary integral equations for the statistical moments of the scattering problems considered. These are then approximated by sparse tensor Galerkin discretizations via the combination technique (Griebel et al. \cite{griebel1990combination,combi2}). We supply extensive numerical experiments confirming the predicted error convergence rates with poly-logarithmic growth in the number of degrees of freedom and accuracy in approximation of the moments. Moreover, we discuss implementation details such as preconditioning to finally point out further research avenues.

\end{abstract}

\begin{keywords} 
Helmholtz equation, shape calculus, uncertainty quantification, boundary element method, combination technique
\end{keywords}

\begin{AMS}
47A40, 35J25, 49Q12.
\end{AMS}
\pagestyle{myheadings}
\thispagestyle{plain}
\markboth{ESCAPIL-INCHAUSP\'{E} AND JEREZ-HANCKES}{HELMHOLTZ SCATTERING BY RANDOM DOMAINS}

\section{Introduction}
\label{sec:intro}
Modeling wave scattering is key in numerous fields ranging from aeronautics to bioengineering or astrophysics. As applications become more complex, the ability to efficiently quantify the effects of random perturbations originated by actual manufacturing or operation conditions becomes ever more relevant for robust design. Under this setting, we consider standard time-harmonic wave scattering models with only aleatoric uncertainty, i.e.~randomness in the shapes. More specifically, we aim at providing an accurate and fast uncertainty quantification (UQ) method for computing statistical moments of wave scattering solutions assuming small random perturbations or deviations from a nominal deterministic shape. 

The model problems here considered involve solving Helmholtz equations in unbounded domains with constant coefficients supplemented by one or more different boundary conditions (BCs), namely, Dirichlet, Neumann, impedance and transmission ones. Under reasonable decay conditions at infinity, deterministic versions of such problems can be shown to be uniquely solvable even for Lipschitz scatterers \cite{Nedelec,sauter}. Considering Lipschitz parametrized transformations, the small perturbation assumption leads to diffeomorphisms between nominal and perturbed domains. This, in turn, gives rise to suitable shape Taylor expansions for the scattered fields, for which the corresponding shape derivatives (SDs) must be computed. Restricting ourselves to sufficiently smoother nominal domain, these SDs are solutions of homogeneous boundary value problems (BVPs) with boundary data depending on the normal component of the velocity field, allowed by the Hadamard structure theorem (see Theorem 2.27 in \cite{Sokolowski}).

We will then approximate fields in the perturbed domains by quantities defined solely on the nominal shape. Indeed, for the cases considered --constant coefficients and unbounded domains--, one can conveniently reduce the volume problems associated to the scattered fields as well as to their SDs, onto the scatterers' boundaries by means of the integral representation formula \cite{sauter}. This involves solving boundary integral equations (BIEs) shown to be well posed.

The above described first order approximation (FOA) can be extended from the deterministic case to now random (but small) perturbations \cite{chernov2013first}, giving birth to equations with deterministic operators with stochastic right-hand sides. Assuming separability of the underlying functional spaces as well as Bochner integrability, application of statistical moments on the linearized equation yields tensorized versions of the operator equations, thus parting from the multiple solves required by Monte Carlo (MC) methods. Yet, direct numerical approximation of these tensor systems gives rise to the infamous curse of dimensionality. This can be, in turn, remedied by applying the general sparse tensor approximation theory originally developed by von Petersdorff and Schwab \cite{vonPetersdorff2006}, and which has multiple applications ranging from diffraction by gratings \cite{SAJ18} to neutron diffusion \cite{FJM19} problems. In our case, numerically, we will employ the Galerkin boundary element method (BEM) to solve the arising first kind BIEs. As both nominal solutions and SDs will be derived over the same surface, the FOA-BEM allows for substantial computational savings by employing the same matrix computations.

Depending on the regularity of solutions, statistical moments resulting from the FOA-BEM can be computed by sparse tensor approximations robustly. Harbrecht, Schneider and Schwab \cite{sparse3} studied the interior Laplace problem with Dirichlet BC whereas the Laplace transmission problem was analyzed in \cite{chernov}. Jerez-Hanckes and Schwab \cite{JS15_613} provide the numerical analysis of the method in the case of Maxwell scattering. Computationally, further acceleration can be achieved by employing the combination technique (CT), introduced by Griebel and co-workers \cite{griebel1990combination,combi2}. Specifically, the method allows for simple and parallel implementation, which we will further detail in the manuscript. Throughout, we apply the FOA-BEM-CT method --referred to as first-order sparse BEM (FOSB) method to alleviate notations-- to the Helmholtz problem. To our knowledge, the case of the FOSB method for the Helmholtz-UQ remains untackled.

The manuscript is structured in the following way. First, we introduce the mathematical tools used throughout in \Cref{sec:mathematical_tools}. Generic scattering problems formulations as well as the description of the BVPs solved by the SDs are given in \Cref{sec:helm_scat}. We then restrict ourselves to the associated BIEs in \Cref{sec:boundary_reduction} and analyze their Galerkin solutions in \Cref{sec:galerkin_CT}. Implementation aspects of the FOSB method are given in \Cref{sec:further} whereas numerical results are provided in \Cref{sec:Numres}. Finally, further research avenues are highlighted in \Cref{sec:conclusion}. 
\section{Mathematical tools}
\label{sec:mathematical_tools}
We start by setting basic definitions as well as the functional space framework adopted for our analysis. As a reference, \Cref{tab:Acronyms} beneath \Cref{sec:galerkin_CT} provides a non-exhaustive list of the acronyms used throughout this work.
\subsection{General notation}
\label{subs:prelim}
Throughout, vectors and matrices are expressed using bold symbols, $(\ba \cdot \bb)$ denotes the classical Euclidean inner product, $\norm{\cdot}{2}:=\sqrt{\ba \cdot \ba}$ refers to the Euclidean norm, $C$ is a generic positive constant and $o$, $\mO$ are respectively the usual little-$o$ and big-$\mO$ notations. Also, we set $\imath^2=-1$, $\IS^1$ and $\IS^2$ are the unit circle and sphere, respectively.

Let $D \subseteq  \IR^d$, with $d=2,3$, be an open set. For a natural number $k$, we set $\IN_k:=\{k,k+1,\ldots\}$. For $p \in \IN_0=\{0,1,\ldots\}$, we denote by $C^{p}(D)$ the space of $p$-times differentiable functions over $D$, by $C^{p,\alpha}(D)$ the space of H\"older continuous functions with exponent $\alpha$, where $0<\alpha\leq1$. Also, let $L^p(D)$ be the standard class of functions with bounded $L^p$-norm over $D$. For a Banach space $X$ and an open set $T \subset \IR$, we introduce the usual Bochner space $C^p(T;X)$. Given $s \in \IR$, $q\geq 0$, $p\in [1,\infty]$, we refer to \cite[Chapter 2]{sauter} for the definitions of function spaces $W^{s,p}(D)$, $H^s(D)$, $\Hloc^q(D)$ and $\Hloc^q(\Delta,D)$. Norms are denoted by $\norm{\cdot}{}$, with subscripts indicating the associated functional space. Similarly, relative norms are denoted by brackets e.g., $[a-b]=\|a-b\| /\|b\|$ for $a$ an approximation of a reference $b$.

For $k \in \IN_1$, and $\bx_i \in \IR^d$, $i=1,\cdots,k$, we set $\underline{\bx}:=(\bx_1,\cdots,\bx_k)$. Besides, $k$-fold tensors quantities are denoted with parenthesized subscripts, e.g., $f^{(k)}:=f \otimes \cdots \otimes f$. This notation applies indifferently to functions, domains and function spaces. The diagonal terms of a $k$-fold tensor $\Sigma^k$ at $\underline{\bx}$ are denoted by $\diag\Sigma^k(\underline{\bx}):=\Sigma^k|_{\bx_1=\cdots=\bx_k}$.
Following \cite[Section 4.1]{JS15_613}, for $X,Y$ separable Hilbert spaces, we set $\fB \in \mL(X,Y)$ the space of linear continuous mapping from $X$ to $Y$ and define the unique continuous tensor product operator:
\vspace{-0.2cm}
$$
\fB^{(k)} := \smash[b]{\! \underbrace{\fB \otimes \cdots \otimes \fB}_\textup{$k$-times}} \in \mL(X^{(k)},Y^{(k)}).\vspace{0.5em}
$$
\subsection{Traces and surface operators}
Let $D  \subset \IR^d$ with $d=2,3$ be open bounded with Lipschitz boundary $\Gamma:=\partial D$ and complement exterior domain $D^c:=\mathbb{R}^d \backslash \overline{D}$. Equivalently, we will write $D^0 \equiv D^c$ and $D^1\equiv D$ to refer to exterior and interior domains, respectively. Accordingly, when defining scalar fields in $D^c\cup D$, we use notation $\U=(\U^0,\U^1)$. For $i=0,1$, we introduce the continuous and surjective trace mappings \cite[Sections 2.6 and 2.7]{sauter}:
\begin{equation*}
\begin{array}{cl}
\textup{(Dirichlet trace)} & \gamma_0 : \Hloc^1(D^i) \to  H^{\frac{1}{2}}(\Gamma),\\
\noalign{\vspace{2pt}}
\textup{(Neumann trace)}   & \gamma_1 : \Hloc(\Delta,D^i) \to  H^{-\frac{1}{2}}(\Gamma).
\end{array}
\end{equation*}
For a suitable scalar field $\U^i$, $i=0,1$, we refer to a pair of traces $\bxi^i$ as Cauchy data if
\be
\label{eq:cauchy_trace}
\bxi^i \equiv \begin{pmatrix} \lambda_i \\ \sigma_i \end{pmatrix}: = \begin{pmatrix} \gamma_0\U^i \\ \gamma_1\U^i \end{pmatrix}.
\ee
Likewise, we introduce the second-order trace operator $\gamma_2\U^i := (\nabla^2\U_i |_\Gamma) \bn \cdot \bn  = \frac{\partial^2\U^i}{\partial \bn^2}\big|_\Gamma$ along with the tangential gradient $\nabla_\Gamma$ and tangential divergence $\divg_\Gamma$ \cite[Section 2.5.6]{Nedelec}. 

\subsection{Random domains}
\label{subs:domain_transformation}
Throughout, we consider an open bounded Lipschitz --nominal-- domain $D \subset \IR^d$, $d=2,3$, of class $C^{2,1}$ \cite[Definition 3.28]{MacLean}, with boundary $\Gamma:=\partial D$ and exterior unit normal field $\bn\in W^{2,\infty}(\Gamma)$ pointing by convention towards the exterior domain. Those domains are commonly referred to as domains with \textit{Lyapunov boundary}. The mean curvature $\mathfrak{H}:=\divg \bn$ belongs to $W^{1,\infty}(\Gamma)$. 

Let $(\Omega,\mA,\IP)$ be a suitable probability space and $X$ a separable Hilbert space. For an index $k\in \IN_1$ and $\underline{\bx} := (\bx_1,\cdots,\bx_k)$, and for $\U:\Omega \to X$ a random field in the Bochner space $L^k(\Omega,\IP ; X)$ \cite[Section 4.1]{JS15_613}, we introduce the statistical moments:
\begin{align}
\label{eq:statmoment}
\mM^k[\U(\omega)]&: = \int_\Omega \U(\bx_1,\omega)\cdots \U(\bx_k,\omega)d\IP(\omega),\\
\IV^k[\U(\omega)]&=: \diag \mM^k[\U(\omega)] - \IE[\U(\omega)]^k,
\end{align}
with $\mM^1\equiv \IE$ being the \textit{expectation} and $\IV^2$ the \textit{pseudo-variance} \cite{SAJ18}.

In a nutshell, the aim of the present work is as follows: given a random domain with realization $\mD(\omega)$ specified later on, consider $\U(\bx,\omega)$ defined over $\mD(\omega)$ as the solution of a Helmholtz scattering problem (see \Cref{sec:helm_scat}). We seek at quantifying:
\be
\label{eq:initial_problem}
\IE [ \U(\omega)] \textup{ and } \mM^k[\U(\omega)-\IE[\U(\omega)]]~\textup{for}\ k\in \IN_2.
\ee
\begin{remark}[Complex statistical moments]\label{remark:complex_moments}Statistical moments for complex random fields induce several quantities of interest. As introduced in \cite[Section V-A]{eriksson2010essential}, for $k\in \IN_2$ and $\U \in L^k(\Omega,\IP;X)$, the $k$th statistical moments are defined as
\be
\alpha_{p;q}\equiv\alpha_{p;q}[\U(\omega)]:= \IE [ \U^p \overline{\U}^q]\textup{, for } p,q \in \IN_0 \textup{, such that } p+q = k.
\ee
Notice that symmetric moments are redundant, i.e.~$\alpha_{p;q}=\overline{\alpha_{q;p}}$. Also, for $k=2$, complex moments are the \textit{pseudo-covariance} $\alpha_{2;0}$ and  covariance $\alpha_{1;1}=\IE[\U \overline{\U}]$. In this manuscript, we focus on 
$$
\mM^k[\U] = \alpha_{k;0} = \overline{\alpha_{0;k}} .
$$
However, some applications involve other choices for $p,q$. Still, our analysis applies verbatim to $\alpha_{p;k}$ up to conjugation of terms in the tensor deterministic formulation in \Cref{subs:tensor_bie} (see, for instance, \cite{SAJ18} for $k=2$).
\end{remark}
We consider a centered random \textit{velocity field} field $\bv \in L^k(\Omega, \IP ; W^{2,\infty}(\Gamma;\IR^d))$, i.e.~such that $\IE[\bv(\cdot,\omega)]=\bf{0}$. Also, assume $\|\bv(\cdot,\omega)\|_{W^{2,\infty}(\Gamma)}\lesssim 1$ uniformly for all $\omega \in \Omega$ and introduce a family of random surfaces \{$\Gamma_t\}_{t}$ via the mapping
\be\label{eq:transformed_velocity}
\Omega \ni \omega \mapsto \Gamma_t (\omega) = \{\bx + t \bv(\omega) ,~ \bx \in \Gamma\} =: T_t(\Gamma)(\omega).
\ee
Following \cite{JS15_613}, we deduce that there exists $\eps >0$ such that, for each $|t|<\eps$ and $\IP$-a.s.~$\omega$, the collection $\{\Gamma_t(\omega)\}$ generates bi-Lipschitz diffeomorphisms and induces connected Lipschitz domain $D_t(\omega)$ by continuity of $\bv(\omega)$ $\IP$-a.s.~on the compact surface $\Gamma$. Besides, we define $\mathfrak{D}_t(\omega)$ corresponding to either $D_t^c(\omega)$ or $D_t^c(\omega) \cup D_t(\omega)$ according to the problem considered. Finally, we notice that $(\bv (\omega) \cdot \bn) \in W^{2,\infty}(\Gamma)$.

\subsection{First-order approximations}
\label{subs:shape_derivative}
With the domain transformation and velocity field defined in \Cref{subs:domain_transformation}, we are ready to introduce the concept of random SD.
\begin{definition}[Random SD]
\label{def:domain_derivative}For $\omega \in \Omega$, consider a random shape dependent scalar field $\U_t(\omega)$ defined in $\mathfrak{D}_t(\omega)$ and denote $\U\equiv\U_0$ for the nominal domain solution. Then, $\U_t(\omega)$ admits a SD $\U'(\omega)$ in $\mathfrak{D}$ along $\bv(\omega)$ if the following limit exists
\be
\label{eq:ddev}
\U'(\omega):= \lim_{t \to 0}\frac{\U_t(\omega)-\U}{t}.
\ee
Assuming SD belongs to $H^1_\textup{loc}(\mD)$ and a Lipschitz condition, then the following Taylor expansion holds for $|t|<\eps$:
\be
\label{eq:Taylor}
\U_t(\bx,\omega) = \U(\bx) + t \U'(\bx,\omega)+\mO(t^2) ~\textup{ in } H^1(Q(\omega))\textup{, }  Q(\omega) \Subset \mathfrak{D} 
\cap \mathfrak{D}_t(\omega).
\ee
\end{definition}
Finally, following \cite[Section 2.1]{dolz2018hierarchical}, we introduce $K$ such that 
\be
K\Subset \mathfrak{D}_t^{\cap\Omega},~ D_t^{\cap \Omega}:= \bigcap_{\omega \in \Omega} D_t(\omega).
\ee
Consequently, according to \Cref{eq:Taylor} and using the embedding arguments of \cite[Lemma 5.9]{dolz2018hierarchical} for the variance, the quantities of interest can be accurately approximated for $k\geq 2$ by
\begin{equation}
\label{eq:approx_moments}
\begin{array}{rll}
\IE [ \U_t(\omega) ] &= \U  +  \mO(t^2),& \textup{ in } H^1(K),\\
\noalign{\vspace{3pt}}
\mM^k [\U_t(\omega) -\U]& = t^k \mM^k[\U'(\omega)] + \mO(t^{k+1}),& \textup{ in } H^1(K)^{(k)}\textup{, and}\\
\noalign{\vspace{3pt}}
\IV^k [\U_t(\omega)]& = t^k \diag \mM^k[\U'(\omega)]+ \mO(t^{k+1}),& \textup{ in } L^2(K).
\end{array}
\end{equation}
Hence, for a random class of parametrized perturbations (\Cref{eq:transformed_velocity}), the statistical moments (\Cref{eq:initial_problem}) can be approximated accurately through $\U$, $\mM^k[\U'(\omega)]$, defined in $\mathfrak{D}$ and $\mathfrak{D}^{(k)}$: the FOA amounts to computing $\U$ and $\mM^k[\U'(\omega)]$. Before proceeding, we decide to sum up the main points of the FOSB method in \Cref{picture:steps}. It describes the path followed throughout and, for each step, details the related section and the quantity of interest considered. The technique is sequential from top to bottom, and between each step we use arrows specify whether an approximation is done or an equivalent formulation is used. Notice that the two ``equivalent'' steps enclose the operations realized on the boundary of the nominal scatterer.
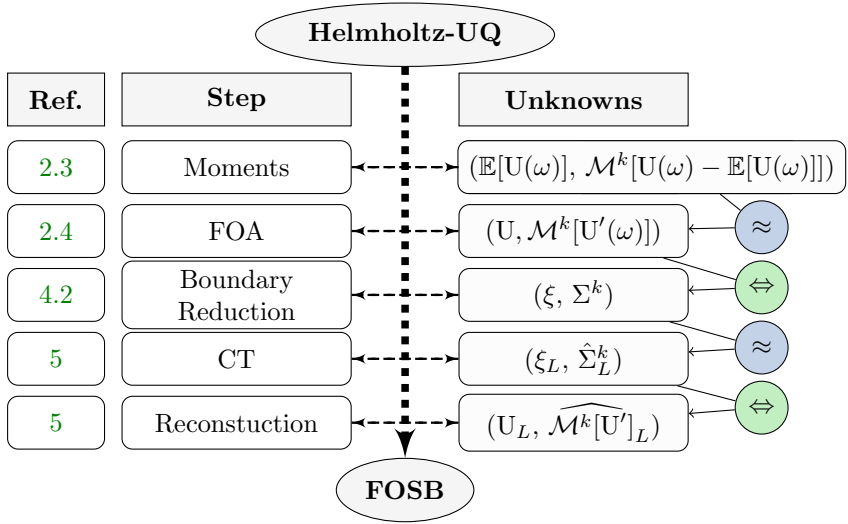
\begin{figure}[t]
 \centering
  \begin{tikzpicture}[paths/.style={->, thick, >=stealth'}, node distance = 2cm, auto]
   \node [cloud, font = \bf] (init) {Helmholtz-UQ};
   \node [type, below left = 0.2cm and -0.7cm of init,font=\bf] (00) {Step};

   \node [objection, below of=00,node distance=.9cm] (10) {Moments};

   \node [type, below right = 0.2cm and -0.7cm of init, font=\bf] (01) {Unknowns};
   \node [block, below of=01,node distance=.9cm] (11) {$a$};
   \node [block, right of=11,node distance=1.5cm] (a1) {};
   
   \node [block0, right of=11,node distance=1.04cm] (a) {$(\IE[\U(\omega)]$, $\mM^k[\U(\omega) - \IE[\U(\omega)]])$};
   \path [line,dashed,line width = 0.3mm] (11) -- (10);
   \path [line,dashed,line width = 0.3mm] (10) -- (11);
    
   \node [objection, below of=10,node distance=0.85cm] (20) {FOA};
   \node [block, below of=11,node distance=0.85cm] (21) {$(\U,\mM^k[\U'(\omega)])$};
   \path [line,dashed,line width = 0.3mm] (21) -- (20);
   \path [line,dashed,line width = 0.3mm] (20) --  (21);

   \node [objection, below of=20,node distance=0.85cm] (30) {Boundary Reduction};
   \node [block, below of = 21,node distance=0.85cm] (31) {$(\xi$, $\Sigma^k)$};
   \path [line,dashed,line width = 0.3mm] (30) --  (31);
   \path [line,dashed,line width = 0.3mm] (31) --  (30);


   \node [objection, below of=30,node distance=0.85cm] (40) {CT};
   \node [block, below of = 31,node distance=0.85cm] (41) {$(\xi_L$, $\hat{\Sigma}_L^k)$};
   \path [line,dashed,line width = 0.3mm] (40) --  (41);
   \path [line,dashed,line width = 0.3mm] (41) --  (40);

   \node [objection, below of=40,node distance=0.85cm] (50) {Reconstuction};
   \node [block, below of = 41,node distance=0.85cm] (51) {$(\U_L$, $\widehat{\mM^k[\U']}_{L})$};
   \path [line,dashed,line width = 0.3mm] (50) --  (51);
   \path [line,dashed,line width = 0.3mm] (51) --  (50);

   \node [cloud, below = 5.2cm of init, font=\bf] (end) {FOSB};
    \path [line,dashed, line width = 1mm] (init) --  (end);

    \node [ref, left of=00,node distance=2.4cm] (12) {{\bf Ref.}};
   \node [ref1, left of=10,node distance=2.4cm] (12) {\ref{subs:domain_transformation}};
   \node [ref1, left of=20,node distance=2.4cm] (22) {\ref{subs:shape_derivative}};
   \node [ref1, left of=30,node distance=2.4cm] (32) {\ref{subs:tensor_bie}};
   \node [ref1, left of=40,node distance=2.4cm] (42) {\ref{sec:galerkin_CT}};
   \node [ref1, left of=50,node distance=2.4cm] (52) {\ref{sec:galerkin_CT}};

   \node [right of=11,node distance=2.5cm] (03) {};
   \node [circle_blue,below of=03,node distance=0.8cm] (13) {$\approx$};
   \node [circle_green, below of=13,node distance=0.8cm] (23) {$\Leftrightarrow$};
   \node [circle_blue, below of=23,node distance=0.8cm] (33) {$\approx$};
   \node [circle_green, below of=33,node distance=0.8cm] (43) {$\Leftrightarrow$};
   \draw[->] (a1) -- (13) -- (21) ;
   \draw[->] (21) -- (23) -- (31) ;
   \draw[->] (31) -- (33) -- (41) ;
   \draw[->] (41) -- (43) -- (51) ;
\end{tikzpicture}
\caption{Sequential description of the FOSB. For each step, we detail the related section and unknwowns and precise whether it consists in an approximation (blue) or an equivalent (green) one.}
\label{picture:steps}
\vspace{-0.5cm}
\end{figure}

\section{Deterministic Helmholtz scattering problems}
\label{sec:helm_scat}
Let us now describe the Helmholtz problems considered in two and three dimensions. We characterize physical domains by a positive bounded wave speed $c$ and a material density constant $\mu$ --representing, for instance, the permeability in electromagnetics. For time-harmonic excitations of angular frequency $\omega>0$, set the wavenumber $\kappa:=\omega/c$ and define the Helmholtz operator:
$$L_\kappa:\U \mapsto -\Delta \U - \kappa^2 \U.$$
The Sommerfeld radiation condition (SRC) \cite[Section 2.2]{Nedelec} for $\U$ defined over $D^c$ and $\kappa$ reads
\be
\SRC(\U,\kappa) \iff \Big|\frac{\partial}{\partial r} \U-\imath \kappa \U\Big|  = o\left( r^{\frac{1-d}{2}}\right) \textup{ for } r:=\norm{\bx}{2} \to \infty,
\ee
for $d=2,3$. This condition will guarantee uniqueness of solutions \Cref{subsec:prob_helm_scat} and the definition of $\F \in C^\infty(\IS^{d-1})$ the far-field \cite[Lemma 2.5]{chandler2012numerical} such that
\be
\label{eq:far_field}
\Big|\U  -  \exp(\imath \kappa r) r^\frac{1-d}{2} \F(\bx/r) \Big| = \mO\Big(r^{-\frac{1+d}{2}}\Big) \textup{ for } r : = \|\bx\|_2 \to \infty.
\ee
\begin{remark}[Far-field Taylor expansions]\label{remark:FF_taylor}As the far-field does not depend on domain transformations, the Taylor expansion for $\U_t,\U$ and $\U'$ in \Cref{eq:approx_moments} transfers straightly to $\F_t,\F$ and $\F'$, respectively. 
\end{remark}
\subsection{Problem Formulations}
\label{subsec:prob_helm_scat}
We introduce the following BVPs corresponding to Dirichlet, Neumann, impedance and transmission BCs.  By linearity, we can write the total wave as a sum of scattered and incident ones, i.e.~$\U=\U^\textup{sc}+\U^\textup{inc}$. Notation (P$_\boldbeta$) with bold $\boldbeta$ will refer to any of the problems considered. For the sake of clarity, we summarize these notations and illustrate domain perturbations in \Cref{fig:fig}.
\begin{subproblem}[P$_\beta$) ($\beta=0,1,2$]Given $\kappa>0$ and $\U^\textup{inc} \in H_{\textup{loc}}^1(D^c)$ with $L_\kappa \U^\textup{inc}=0$ in $D^c$, we seek $\U\in H_{\textup{loc}}^1(D^c)$ such that
$$
\left\{
\begin{array}{lll}
\Delta\U+ \kappa^2\U  = 0&\textup{ in } D^c,\\
\noalign{\vspace{3pt}}
  \gamma_{\beta} \U = 0&\textup{ on } \Gamma, &\textup{ if } \beta \in \{0,1\}\textup{, or}\\
  \noalign{\vspace{3pt}}
  \gamma_{1} \U  + \imath \eta \gamma_0 \U = 0,~\eta >0 &\textup{ on } \Gamma, & \textup{ if } \beta = 2,\\
  \noalign{\vspace{3pt}}
    \SRC(\U^\textup{sc},\kappa).
\end{array}
\right.
$$
\end{subproblem}
\begin{subproblem}[P$_\textup{3}$]Let $\kappa_i,\mu_i>0$, $i=0,1$, with either $\kappa_0\neq\kappa_1$ or $\mu_0\neq\mu_1$, and $\U^\textup{inc} \in H_{\textup{loc}}^1(D^c)$ with $L_{\kappa_0} \U^\textup{inc}=0$ in $D^c$. We seek $(\U^0,\U^1) \in  H_{\textup{loc}}^1(D^c) \cup H^1(D)$ such that
$$
\left\{
\begin{array}{lll}
\Delta \U^i+\kappa^2_i\U^i &= 0&\textup{ in } D^i \textup{, for } i=0,1,\\
\noalign{\vspace{3pt}}
  [\gamma_{0} \U]_\Gamma &= 0 &\textup{ on } \Gamma,\\
  \noalign{\vspace{3pt}}
  [\mu^{-1}\gamma_{1} \U]_\Gamma &= 0&\textup{ on } \Gamma,\\
  \noalign{\vspace{3pt}}
\SRC(\U^\textup{sc},\kappa_0).
\end{array}
\right.
$$
\end{subproblem}
Exterior problems (P$_\beta$), $\beta=0,1$, represent the {sound-soft} and {-hard} acoustic wave scattering while (P$_2$) and (P$_3$) describe the \textit{exterior impedance} and \textit{transmission} problems, respectively. Notice that (P$_\boldbeta$) is known to be well posed \cite[Chapter 4]{MacLean}. 
\begin{figure}[t]
\begin{minipage}{\linewidth}
\begin{minipage}{0.7\linewidth}
\begin{table}[H]
\renewcommand\arraystretch{1.2}
\begin{center}
\footnotesize
\resizebox{9.1cm}{!} {
\begin{tabular}{|c|c|c|c|} \hline   
 $\beta$ &  Problem  & $\mathfrak{D}$  & BCs\\ \hline\hline
 0 &  Sound-soft & $D^c$ & $\gamma_0 \U= 0$ \\ \hline
 1 &  Sound-hard & $D^c$ & $\gamma_1 \U= 0$ \\ \hline
 2 &  Impedance & $D^c$ & $\gamma_1 \U + \imath \eta \gamma_0 \U = 0$ \\ \hline
 3 &  Transmission & $D^c\cup D $ & $[\gamma_0\U]_\Gamma  = [\mu^{-1}\gamma_1\U]_\Gamma =0$ \\  \hline
\end{tabular}}
\end{center} 
\label{tab:Notations}
\end{table}  
\end{minipage}
\begin{minipage}{0.29\linewidth}
\begin{figure}[H]
\center
\includegraphics[width=.78\textwidth]{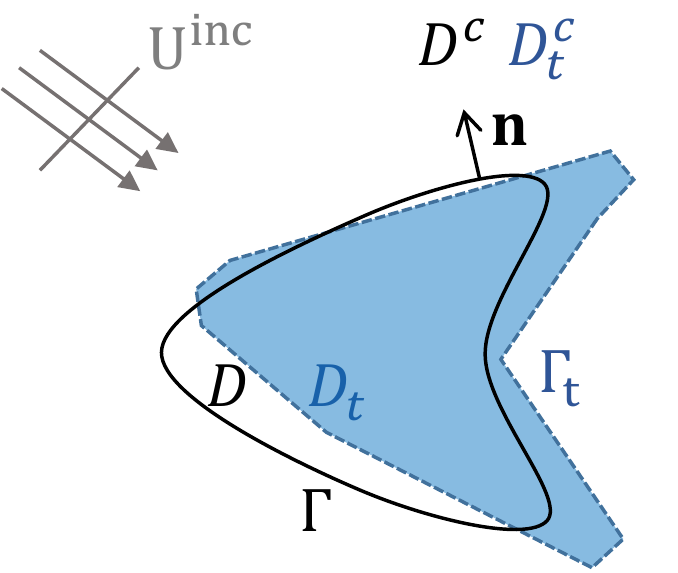}
\label{fig:fig2}       
\end{figure}
\end{minipage}
\end{minipage}
\caption{Overview of (P$_\boldbeta$) (left) and representation of domain transformations (right).}
\label{fig:fig}
\end{figure}

\subsection{Shape derivatives for Helmholtz scattering problems}
\label{subs:helmholtz_scattering_shape}
We summarize the BVPs, denoted by (SP$_\boldbeta$), satisfied by the SD for each BC, as detailed in \cite[Table 5.6]{hiptmair2017shape}.
\begin{subproblem}[SP$_\beta$) ($\beta=0,1,2$]We seek $\U'\in \Hloc^1(D^c)$ solution of
$$
\left\{
\label{eq:S_EP}
\begin{array}{lll}
\Delta\U'+\kappa^2\U' = 0&\textup{ in } D^c,\\
\noalign{\vspace{3pt}}
  \gamma_{\beta} \U' = g_\beta&\textup{ on } \Gamma, & \textup{ if } \beta \in \{0,1\} \textup{, or} \\
  \noalign{\vspace{3pt}}
  \gamma_{1} \U'  + \imath \eta \gamma_0 \U'=  g_2,~\eta >0 &\textup{ on } \Gamma, & \textup{ if } \beta = 2,\\
  \noalign{\vspace{3pt}}
  \SRC(\U',\kappa),
\end{array}
\right.
$$
wherein, for $\U$ being the respective solution of (P$_\beta$), we have
\begin{align*}
g_0 &: =  - \gamma_1\U (\bv \cdot \bn),\\
g_1 &: = \divg_{\Gamma} \left((\bv \cdot \bn) \nabla_{\Gamma} \U\right) + \kappa^2 \gamma_0\U ( \bv \cdot \bn ),\\
g_2 & : =  \divg_{\Gamma} \left((\bv \cdot \bn) \nabla_{\Gamma} \U\right) + \kappa^2 \gamma_0\U ( \bv \cdot \bn ) + \imath \eta  (\bv \cdot \bn) (-\gamma_1 \U- \mathfrak{H} \gamma_0 \U).
\end{align*}
\end{subproblem}
\begin{subproblem}[SP$_\textup{3}$]We seek $\U'=(\U'^0,\U'^1)\in H_{\textup{loc}}^1(D^c) \times H^1(D)$ solution of
$$
\left\{
\label{eq:S_EP_3}
\begin{array}{lll}
\Delta \U'^i + \kappa^2_i \U'^i &= 0&\textup{ in } D^i\textup{, for } i=0,1,\\
\noalign{\vspace{3pt}}
  [ \gamma_{0}  \U' ]_\Gamma&= h_0 &\textup{ on } \Gamma,\\
  \noalign{\vspace{3pt}}
  [ \frac{1}\mu \gamma_{1}  \U' ]_\Gamma &= h_1 &\textup{ on } \Gamma,\\
  \noalign{\vspace{3pt}}
  \SRC(\U'^0,\kappa_0),
\end{array}
\right.
$$
with boundary data built using $\U$ solution of (P$_3$), as follows
\begin{align*}
h_0 & :=  -  [ \gamma_1\U ]_\Gamma (\bv \cdot \bn), \\
h_1 & := \Big[ \frac{1}{\mu} \Big]_\Gamma\divg_{\Gamma} \left((\bv \cdot \bn) \nabla_{\Gamma} \U\right) +  [ \kappa^2  ]_\Gamma\gamma_0\U ( \bv \cdot \bn ).
\end{align*}
\end{subproblem}
In the proposed setting, (SP$_\boldbeta$) is known to be well posed (\textit{cf.}~\cite[Section 3.2]{hiptmair2017shape}).

Having described the deterministic problems, we now consider the random domains described \Cref{subs:domain_transformation} and analyze, for each realization $\U_t(\omega)$, solutions of (P$_\boldbeta$). The prior choice of random domains ensures wellposedness of the perturbed solution $\U_t(\omega)$ and of its shape derivative $\U'(\omega)$ for each realization. Therefore, we apply the FOA framework of \Cref{subs:shape_derivative} to $\U_t(\omega)$, allowing to obtain an accurate approximation of the statistical moments of $\U_t(\omega)$ through:
$$ \U \textup{ and } \mM^k[\U'(\omega)],$$
defined over $\mD$ and $\mD^{(k)}$, respectively --check step 2 in \Cref{picture:steps}. In the same spirit as in \cite{dolz2018hierarchical}, the domain and perturbations considered allow for a bounded shape Hessian in $H_\textup{loc}^1(\mD)$, hence the Lipschitz condition for the SD. As these domains are unbounded, we reduce the problem to the boundary $\Gamma$ via BIEs. Notice that the randomness in $\U'(\omega)$ appears only through $(\bv\cdot \bn)(\omega)$, which appears solely in BCs.
\section{Boundary Reduction}
\label{sec:boundary_reduction}
In this section, we explain how to reduce the Helm-holtz boundary value problems described before as well as their SDs onto the boundary via the integral representation formula. Then, we consider the small random domain counterparts and show how the SDs are equivalently reduced to BIEs comprising deterministic operators with stochastic right-hand side. As mentioned initially, this will fit the general framework described in \cite{vonPetersdorff2006} to compute statistical moments.
\subsection{Boundary integral operators in scattering theory}
\label{subs:bio}
First, we define the duality product between $\bxi_1= (\lambda_1 , \sigma_1)$ and $\bxi_2=(\lambda_2 , \sigma_2)$ both in the Cartesian product space $H^{1/2}(\Gamma)\times H^{-1/2}(\Gamma)$:
$$
\langle \bxi_1,\bxi_2\rangle_\Gamma := \langle \lambda_1 , \sigma_1 \rangle_\Gamma + \langle \lambda_2 , \sigma_2 \rangle_\Gamma.$$
Recall the fundamental solution $G_\kappa(\bx,\by)$ of the Helmholtz equation for $\kappa>0$:
\be
G_\kappa(\bx,\by) :=\left\{
\begin{array}{ll}
\dfrac{\imath}{4}H_0^{(1)}(\kappa \|\bx -\by\|_2)& \textup{ for } d=2, \\\noalign{\vspace{4pt}}
\dfrac{\imath}{4\pi}\dfrac{ \exp (\imath \kappa \|\bx -\by\|_2)}{\|\bx-\by\|_2} & \textup{ for } d=3,
\end{array}
\right.
\ee
where $H_0^{(1)}$ is the zeroth-order Hankel function of the first kind. With this, we introduce the single- and double-layer potentials for $\phi \in L^1(\Gamma)$:
\begin{align*}
\SL_\kappa (\phi) (\bx) &:= \int_\Gamma G_\kappa (\bx-\by) \phi(\by) d\Gamma(\by)  & \bx \in \IR^{d}\backslash \Gamma, \\
\DL_\kappa (\phi) (\bx) &:= \int_\Gamma \frac{\partial}{\partial \bn_\by}  G_\kappa (\bx-\by) \phi(\by) d\Gamma(\by)  & \bx \in \IR^{d}\backslash \Gamma.
\end{align*}
With this at hand, we introduce the block Green's potential:
$$
\fR_\kappa := (\DL_\kappa ,-\SL_\kappa),
$$
and for the sake of convenience, we identify implicitly the Cauchy data (\Cref{eq:cauchy_trace}) with the domain index for $\beta=2$, i.e.:
\be
\label{eq:compact_potential}
\fR_\kappa(\bxi)(\bx) \equiv \fR_{\kappa_i}(\bxi^i)(\bx) \textup{, if } \bx \in D^i,~i=0,1.
\ee
The identity operator $\opI$ and five continuous boundary integral operators in Lipschitz domains for $\kappa >0$, $\eta > 0$ and $|s|\leq 1$ \cite[Theorems 2.25 and 2.26]{chandler2008wave}:
\be
\begin{array}{lll}
\opV_\kappa &: H^{s-1/2}(\Gamma) \to H^{s+1/2}(\Gamma), & \opV_\kappa := \{\!\!\{\gamma_0 \}\!\!\}_\Gamma \circ \SL_\kappa,\\
\opK_\kappa &: H^{s+1/2}(\Gamma) \to H^{s+1/2}(\Gamma), & \opK_\kappa := \{\!\!\{ \gamma_0 \}\!\!\}_\Gamma \circ \DL_\kappa,\\
\opK'_\kappa &: H^{s-1/2}(\Gamma) \to H^{s-1/2}(\Gamma), & \opK'_\kappa := \{ \!\!\{\gamma_1 \}\!\!\}_\Gamma \circ \SL_\kappa,\\
\opW_\kappa &: H^{s+1/2}(\Gamma) \to H^{s-1/2}(\Gamma), & \opW_\kappa := -\{ \!\!\{\gamma_1\}\!\!\}_\Gamma \circ \DL_\kappa,\\
\opB'_{\kappa,\eta} &:H^{s+1/2}(\Gamma) \to H^{s-1/2}(\Gamma),& \opB'_{\kappa,\eta}:=\opW_\kappa -\imath \eta \left(\frac{1}{2} \opI + \opK'_\kappa\right).
\end{array}
\ee
Also, we introduce the following operator:
$$
\opA_\kappa : = \begin{bmatrix}
       -\opK_\kappa & \opV_\kappa         \\[0.3em]
       \opW_\kappa          & \opK'_\kappa
     \end{bmatrix},
$$
along with 
$$
\widehat{\opA}_{\kappa,\mu}: = \begin{bmatrix}
       1 & 0        \\[0.3em]
       0 & 1 / \mu
     \end{bmatrix}   \begin{bmatrix}
       -\opK_\kappa & \opV_\kappa          \\[0.3em]
       \opW_\kappa          & \opK'_\kappa
     \end{bmatrix} \begin{bmatrix}  1& 0        \\[0.3em]
       0 & \mu
     \end{bmatrix} =    \begin{bmatrix}
       -\opK_\kappa &  \mu \opV_\kappa          \\[0.3em]
      1/\mu \opW_\kappa          & \opK'_\kappa
     \end{bmatrix}.
$$

Next, we consider a \textit{radiating solution} $\U$, i.e.~$\U=(\U^0,\U^1)$~such that $L_\kappa \U^i =0, i=0,1$, and $\U^0$ with Sommerfeld radiation conditions \cite[Section 3.6]{sauter}. Therefore, the following representation formula holds:
\be
\label{eq:representation_formula}
\U= \DL_{\kappa}([\gamma_0 \U]_\Gamma) - \SL_{\kappa}([\gamma_1 \U]_\Gamma) = \fR_\kappa ([\bxi]_\Gamma)\textup{ in } D^0\cup D^1.
\ee
Its Cauchy data $\bxi^i=(\gamma_0\U^i,\gamma_1\U^i)=:(\lambda^i,\sigma^i)\in H^{1/2}(\Gamma)\times H^{-1/2}(\Gamma)$ satisfy
\be
\renewcommand\arraystretch{1.3}
\label{eq:cald}\begin{array}{cl}
\textup{(Interior)} &\bxi^1 =  \left(\frac{1}{2}\opI+ \opA_{\kappa}\right) \bxi^1= : \fP_\kappa^1 \bxi^1,\\
\textup{(Exterior)}   &\bxi^0 =  \left(\frac{1}{2}\opI -\opA_{\kappa} \right)\bxi^0=: \fP_\kappa^0 \bxi^0 = (\opI - \fP_\kappa^1)\bxi^0.
\end{array}
\ee
Notice that the above identities are also valid for $\widehat{\opA}_{\kappa_0,\mu}$ and $\widehat{\opA}_{\kappa_1,\mu}$. Operators $\fP_\kappa^0,\fP_\kappa^1$ are dubbed exterior and interior \textit{Calder\'on} projectors. They share the interesting property that for $i=0,1$, $(\fP_\kappa^i)^2 = \fI$, allowing for Calder\'on-based operator preconditioning (see \Cref{subs:prec}). 

Lastly, we introduce $S_\textup{Dir}(D)\equiv S_0(D)$ and $S_\textup{Neum}(D)\equiv S_1(D)$ the countable set accumulating only at infinity of strictly positive eigenvalues of Helmholtz problem with homogeneous Dirichlet and Neumann boundary conditions \cite[Section 3.9.2]{sauter}. 

\subsection{Tensor BIEs}
\label{subs:tensor_bie}
We now show that the FOA analysis for (P$_\boldbeta$) can be reduced to  (B$_\boldbeta$) defined further in \Cref{pb:B_Generic}, consisting in two deterministic first kind BIEs including a tensor one (refer to \Cref{picture:steps}). To begin with, we consider both deterministic problems (P$_\boldbeta$) and (SP$_\boldbeta$), and show that they can be reduced to two wellposed BIEs of the form: 
\begin{problem}[Deterministic BIEs]\label{pb:BIE_Generic}Provided $\fZ \in \mL(X,Y)$ and $\fB\in \mL(Y,Y)$ for separable Hilbert spaces $X,Y$, $f\in Y$ and $g\in Y$, we seek $\xi,\xi' \in X$ such that:
\be
\left\{
\begin{array}{lll}
\fZ \xi &= f  &\textup{ on } \Gamma,\\
\noalign{\vspace{2pt}}
\fZ \xi' & = \fB g &\textup{ on } \Gamma.
\end{array}
\right.
\ee
\end{problem}
Equivalence between problems couple ((P$_\boldbeta$),(SP$_\boldbeta$)) and \Cref{pb:BIE_Generic} is derived through the following steps:
\begin{enumerate}
  \item[(i)] Using \Cref{subs:bio}, we perform the boundary reduction for (P$_\boldbeta$) and (SP$_\boldbeta$), leading to \Cref{pb:BIE_Generic} (see \Cref{appendix:A});
  \item[(ii)]We prove that \Cref{pb:BIE_Generic} is well-posed in adapted Sobolev spaces using Fredholm theory. Notice that we remove the spurious eigenvalues to guarantee injective boundary integral operators.
\end{enumerate}
Step (ii) is extensively surveyed for $\beta=0,1,2$ in \cite{chandler2008wave}; see Table 2.1 and Theorem 2.25 for $\beta=0,1$, and refer to Section 2.6 for $\beta=2$. The transmission problem ($\beta=2$) is analyzed in \cite[Section 3]{claeys2012multi}.

After reducing the deterministic problem to the boundary, we retake \Cref{subs:shape_derivative} and consider the random counterparts of (P$_\boldbeta$) and (SP$_\boldbeta$), leading to $\U$ and $\U'(\omega)$, and perform correspondingly the boundary reduction. The generic random BIEs for the SD read:
\begin{problem}[Random BIEs]Provided $\fZ \in \mL(X,Y)$ and $\fB\in \mL(Y,Y)$ for separable Hilbert spaces $X,Y$, $g\in L^k(\Omega,\IP;Y)$, for $k\in \IN_2$, we seek $\xi'\in L^k(\Omega,\IP; X)$ such that 
\be
\fZ \xi = \fB g \textup{ on } \Gamma.
\ee
\end{problem}
Applying Theorem 6.1 in \cite{vonPetersdorff2006}, we deduce that the tensor operator equation admits a unique solution $\xi'\in L^k(\Omega,\IP;Y)$ and that $\mM^k[\xi'(\omega)]\in Y$. Therefore, we arrive at a tensor BIE with stochastic right-hand sides, providing the final form of the wellposed deterministic tensor operator BIEs (B$_\boldbeta$):

\begin{problem}[B$_\boldbeta$) (Formulation for the BIEs]\label{pb:B_Generic}Given $\fZ\in \mL(X,Y)$, $\fB \in \mL(Y,Y)$ for separable Hilbert spaces $X,Y$, $k\in \IN_2$, $f\in Y$ and $\mM^k[g]\in Y^{(k)}$, seek $\xi \in X,\Sigma^k \in X^{(k)}$ such that:
\be
\left\{
\begin{array}{lll}
\fZ \xi &= f  &\textup{ on } \Gamma,\\
\noalign{\vspace{2pt}}
\fZ^{(k)} \Sigma^k & = \fB^{(k)}\mM^k[g] &\textup{ on } \Gamma^{(k)}.
\end{array}
\right.
\ee
\end{problem}
We now detail the resulting sets of BIEs for each problem (B$_\boldbeta)$ as well as for. As in \cite[Section 6.2]{vonPetersdorff2006}, notice that the statistical moments and the layer potentials commute by Fubini's theorem.
\begin{subproblem}[B$_\textup{0}$]If $\kappa^2 \notin S_\textup{Dir}(D)$, $\gamma_0\U^\textup{inc} \in H^{1/2}(\Gamma)$ and $\mM^k[g_0]\in H^{1/2}(\Gamma)^{(k)}$, $k\in \IN_2$, we seek $\gamma_1 \U\in H^{-1/2}(\Gamma)$ and $\mM^k[\gamma_1\U']\in H^{-1/2}(\Gamma)^{(k)}$ such that:
\be
\left\{
\begin{array}{llll}
\opV_\kappa \gamma_1 \U&= \gamma_0\U^\textup{inc}  &\textup{ on } \Gamma,\\
\noalign{\vspace{2pt}}
\opV_\kappa^{(k)} \mM^k[\gamma_1 \U']& = \left(  -\frac{1}{2}\opI +\opK_\kappa \right)^{(k)} \mM^k[ g_0]&\textup{ on } \Gamma^{(k)}.
\end{array}
\right.
\ee
Then, 
\begin{align*} 
\U  &= \U^\textup{inc} - \SL_\kappa \gamma_1 \U \textup{ in } D^c,\\
\mM^k[\U']&=  \mM^k[-  \SL_\kappa \gamma_1 \U' + \DL_\kappa g_0] \textup{ in } (D^c)^{(k)}\\
&=  \fR_\kappa^{(k)} \mM^k[(g_0,\gamma_1\U')] \textup{ in } (D^c)^{(k)}.
\end{align*}
\end{subproblem}
\begin{subproblem}[B$_\textup{1}$]If $\kappa^2 \notin S_\textup{Neum}(D)$, $\gamma_1\U^\textup{inc} \in H^{-1/2}(\Gamma)$ and $\mM^k[g_1]\in H^{-1/2}(\Gamma)^{(k)}$, $k\in \IN_2$, we seek $\gamma_0 \U\in H^{1/2}(\Gamma)$ and  $\mM^k[\gamma_0\U'] \in H^{1/2}(\Gamma)^{(k)}$ such that:
\be
\left\{
\begin{array}{llll}
\opW_\kappa \gamma_0 \U&= \gamma_1\U^\textup{inc}  &\textup{ on } \Gamma, \\
\noalign{\vspace{2pt}}
\opW_\kappa^{(k)} \mM^k[\gamma_0 \U']& = \left(- \left(\frac{1}{2} \opI +  \opK'_\kappa  \right)\right)^{(k)} \mM^k[g_1] &\textup{ on } \Gamma^{(k)}.\\
\end{array}
\right.
\ee
Also, 
\begin{align*} 
\U &= \U^\textup{inc} + \DL_\kappa \gamma_0 \U \textup{ in } D^c,\\
\mM^k[\U'] &= \mM^k[-  \SL_\kappa \gamma_1 \U' + \DL_\kappa g_0]\textup{ in } (D^c)^{(k)}\\
&=  \fR_\kappa^{(k)} \mM^k[(\gamma_0\U',g_1)] \textup{ in } (D^c)^{(k)}.
\end{align*}
\end{subproblem}
\begin{subproblem}[B$_\textup{2}$]If $\kappa^2 \notin S_\textup{Neum}(D)$, $\gamma_1\U^\textup{inc}\in H^{-1/2}(\Gamma)$ and $\mM^k[g_2] \in H^{-1/2}(\Gamma)^{(k)}$, $k\in \IN_2$, we seek $\gamma_0 \U \in H^{1/2}(\Gamma)$ and $\mM^k[\gamma_0\U']\in H^{1/2}(\Gamma)^{(k)}$ such that:
\be
\left\{
\begin{array}{llll}
\opB'_{\kappa,\eta}\gamma_0 \U &= \gamma_1 \U^\textup{inc} &\textup{ on } \Gamma,\\
\noalign{\vspace{2pt}}
(\opB'_{\kappa,\eta})^{(k)}\mM^k[\gamma_0 \U']& =  \left(\frac{1}{2}\opI + \opK'_\kappa \right)^{(k)} \mM^k[g_2] &\textup{ on } \Gamma^{(k)}.\\
\end{array}
\right.
\ee
Moreover, 
\begin{align*} 
\U &= \U^\textup{inc} + (\imath\eta  \SL_\kappa + \DL_\kappa)  \gamma_0 \U \textup{ in } D^c,\\
\mM^k[\U'] &= \mM^k[(\imath\eta  \SL_\kappa + \DL_\kappa) \gamma_0\U' -\SL_\kappa g_2  ]\textup{ in } (D^c)^{(k)}\\
&=  \fR_\kappa^{(k)} \mM^k[(\gamma_0\U' , g_2 - \imath \eta \gamma_0\U')] \textup{ in } (D^c)^{(k)}.
\end{align*}
\end{subproblem}
\begin{subproblem}[B$_\textup{3}$]For $\bxi^\textup{inc} := (\gamma_0 \U^\textup{inc}, \gamma_1 \U^\textup{inc})  \in\left[H^{1/2}(\Gamma) \times H^{-1/2}(\Gamma)\right]^{(k)}$ and $\mM^kh=\mM^k(h_0,h_1)\in  (H^{1/2}(\Gamma) \times H^{-1/2}(\Gamma))^{(k)}$, for $k\in \IN_2$, we seek $\bxi^0 := (\gamma_0 \U^0, \gamma_1 \U^0 ) \in H^{1/2}(\Gamma) \times H^{-1/2}(\Gamma)$ and $\mM^k[\bxi'] \equiv \mM^k[{\bxi'}^0] \in (H^{1/2}(\Gamma) \times H^{-1/2}(\Gamma))^{(k)}$ such that:
\be
\left\{
\label{eq:B_TP}
\begin{array}{lll}
\left(\widehat{\opA}_{\kappa_0,\mu_0} + \widehat{\opA}_{\kappa_1,\mu_1}\right) \bxi^0&= \bxi^\textup{inc} &\textup{ on } \Gamma,\\
\noalign{\vspace{3pt}}
\left(\widehat{\opA}_{\kappa_0,\mu_0} + \widehat{\opA}_{\kappa_1,\mu_1}\right)^{(k)} \mM^k[\bxi'] &= \left(\frac{1}{2}\opI + \widehat{\opA}_{\kappa_1,\mu_1} \right)^{(k)} \mM^k[h] &\textup{ on } \Gamma^{(k)}.\\
\end{array}
\right.
\ee
Also,
\begin{align*} 
\U(\bx) &= \U^\textup{inc}(\bx) - \SL_{\kappa_0} \gamma_1 \U^0 +  \DL_{\kappa_0} \gamma_0 \U^0, ~ \bx \in D^c,\\
\U(\bx) &=  - \SL_{\kappa_1} \gamma_1 \U^1 +  \DL_{\kappa_1} \gamma_0 \U^1, ~ \bx \in D,\\
\mM^k[\U'] & = \mM^k [\fR_\kappa(\bxi)] = \fR_\kappa^{(k)} \mM^k [\bxi] \textup{ in } \mD^{(k)}.
\end{align*}
\end{subproblem}
Ultimately, we sum up the functional spaces and BIEs for (PB$_\boldbeta$) in \Cref{tab:NotationsBIEs}. Also, corresponding Sobolev spaces of higher regularity will be denoted $X^s$, $Y^s$, for $s\geq0$, with $Y^0\equiv Y$ and $X^0\equiv X$.  Operator $\fC$ refers to the left-preconditioner that is used for operator preconditioning purposes, as detailed later on in \Cref{subs:prec}. 
\begin{table}[t]
\renewcommand\arraystretch{1.7}
\begin{center}
\footnotesize
\resizebox{13cm}{!} {
\begin{tabular}{|c|c|c|c|c||c|} \hline   
 $\beta$ &  Problem  & $X^s$ & $Y^s$  & $\fZ$  & $\fC$\\ \hline\hline
 0 &  Soft & $H^{-1/2+s}(\Gamma)$ &$H^{1/2+s}(\Gamma)$ & $\opV_\kappa$  & $\opW_\kappa$\\ \hline
 1 &  Hard & $H^{1/2+s}(\Gamma)$ & $H^{-1/2+s}(\Gamma)$ & $\opW_\kappa$& $\opV_\kappa$ \\ \hline 
 2 &  Impedance &$H^{1/2+s}(\Gamma)$& $H^{-1/2+s}(\Gamma)$  & $\opW_\kappa  - \imath \eta \left(\frac{1}{2}\opI + \opK'_\kappa \right)$  &  $\opV_\kappa$\\ \hline
 3 &  Transmission & $H^{1/2+s}(\Gamma)\times H^{-1/2+s}(\Gamma)$& $H^{1/2+s}(\Gamma)\times H^{-1/2+s}(\Gamma)$ & $(\widehat{\opA}_{\kappa_0,\mu_0} + \widehat{\opA}_{\kappa_1,\mu_1})$  & $(\widehat{\opA}_{\kappa_0,\mu_0} + \widehat{\opA}_{\kappa_1,\mu_1})$ \\  \hline
\end{tabular}}
\end{center} 
\caption{Overview of the BIEs for (B$_\boldbeta$) and associated operator preconditioner employed in \Cref{subs:prec}.}
\label{tab:NotationsBIEs}
\end{table}
\section{Galerkin Method and Sparse Tensor Elements}
\label{sec:galerkin_CT}
We now aim to solve numerically the variational forms arising from the BIEs described in \Cref{pb:B_Generic}. Let us introduce a nested shape-regular and quasi-uniform family $\{\mM_l\}_{l\in \IN_0}$ of surface  triangulations consisting of triangles or quadrilaterals, with each level $l$ associated to a meshwidth $h_l>0$. For $\beta =0,1$, we define the associated boundary element spaces $V^\beta_0\subset V^\beta_1 \subset \cdots V^\beta_l \subset H^{\frac{1}{2}-\beta}(\Gamma)$:
\begin{align*}
V_l^0 &: = \{ \lambda \in C(\Gamma) : \lambda|_K \in \fP_p(K) , ~\forall K \in \mM_l, ~p \in \IN_1\},\\
V_l^1 &: = \{ \sigma \in L^2(\Gamma) : \sigma|_K \in \fP_{p}(K) , ~\forall K \in \mM_l,~p \in \IN_0\}.
\end{align*}
where $\mP_p(K)$ stands for the space of polynomials of degree $\leq p $, $p \in \IN_0$ on the cell $K$. Notice that under regular enough Neumann data, i.e.~Neumann traces belong to $H^{(d-1)/2+\delta}(\Gamma)$ \cite[Theorem 2.5.4]{sauter} for any $\delta >0$, we can also use piecewise continuous functions e.g., piecewise linear functions $\IP^1$, as in \Cref{sec:Numres}. Afterwards, we introduce usual best approximation estimates for the $h$-version of boundary elements \cite[Chapter 9]{sauter}.
\begin{lemma}[Interpolation error for Dirichlet traces]
\label{lemma:interpV}
For $0 \leq t \leq s \leq p + 1$ and all $\lambda \in H^s(\Gamma)$, there holds
\be
\inf_{v_l \in V^0_l} \|\lambda - v_l \|_{H^t(\Gamma)} \leq C h^{s-t} \|\lambda\|_{H^s(\Gamma)},
\ee
where $C>0$ is independent of $h$ and $\lambda$.
\end{lemma}
\begin{lemma}[Interpolation error for Neumann traces]
\label{lemma:interpW}
For $0 \leq t \leq s \leq p + 1$ and all $\sigma \in H^s(\Gamma)$, there holds
\be
\inf_{v_l \in V^1_l} \|\sigma - v_l \|_{H^{-t}(\Gamma)} \leq C h^{s+t} \|\sigma\|_{H^s(\Gamma)},
\ee
where $C>0$ is independent of $h$ and $\lambda$.
\end{lemma}
\subsection{First-order statistical moments}
Adopting the notation in \Cref{tab:NotationsBIEs}, we define $X_L\subset X$ from $V_L^0$ and $V_L^1$, and arrive at the Galerkin formulation for $\xi$:
\begin{problem}[Galerkin formulation]\label{pb:var_1st}Seek $\xi_L \in X_L\subset X$ such that:
\be
\langle \fZ\xi_L , \phi_L \rangle_\Gamma = \langle f , \phi_L \rangle_\Gamma, ~ \forall \phi_L \in X_L.
\ee
\end{problem}
We define $N_L := \card(X_L)$. Classical results for coercive operators \cite{sauter} ensure that there exists a minimum resolution $L_0$ such that the discrete solution is well defined and converges quasi-optimally in $X$. Thus, provided that $\xi \in X^s$ for any $0\leq s \leq p+1$, by \Cref{lemma:interpV,lemma:interpW} it holds
\be
\|\xi- \xi_L \|_{X} \leq C h^{s} \|\xi\|_{X^s}.
\ee
\subsection{Higher-order statistical moments and CT}
Having introduced the tensor $L^2$-product $\langle\cdot,\cdot \rangle_{\Gamma^{(k)}}$ \cite{combi2}, we state the tensor deterministic variational forms of the BIEs:

\begin{problem}[Tensor Galerkin]\label{pb:full}Given $k\in \IN_2$, seek $\Sigma_L^k  \in X_L^{(k)}$ such that 
\be
\langle \fZ^{(k)}\Sigma_L^k  , \Theta_L^k \rangle_{\Gamma^{(k)}} = \langle \fB^{(k)}\mM^k[g] , \Theta_L^k \rangle_{\Gamma^{(k)}}, ~ \forall \Theta_L^k\in X_L^{(k)}.
\ee
\end{problem}
As shown in \cite[Section 3.5]{vonPetersdorff2006}, there is a $L_0(\kappa)\in\IN_0$ for which, for all $L\geq L_0(\kappa)$, the tensorized problem admits a discrete inf-sup, and has a unique solution converging quasi-optimally in $X^{(k)}$. From here, we deduce the following error estimates
\be
\|\Sigma^k- \Sigma_L^k\|_{X^{(k)}} \leq C h^{s} \|\Sigma^k\|_{(X^s)^{(k)}}.
\ee
provided that $\Sigma^k\in (X^s)^{(k)}$, for any $0 \leq s \leq p+1$. Now, we introduce the complement spaces:
$$
W_0 := X_0 ,~ W_l := X_l \backslash X_{l-1},~ l>0,
$$
and consider the sparse tensor product space:
\be
\widehat{X}_L^{(k)}(L_0)=  \bigoplus_{\|\underline{l}\|_1 \leq L + (k-1)L_0} W_{l_1} \otimes\cdots \otimes W_{l_k}
\ee
Then, we can state the following stability condition.
\begin{lemma}[{\cite[Theorem 5.2]{vonPetersdorff2006}}]
For $k\in\IN_2$, there exists $L_0(k)$ and $\hat{c}_S$ such that for all $L\geq L_0$, it holds
\be
\inf_{0\neq \hat{\Sigma} \in \widehat{X}_L^{(k)}} \sup_{0\neq \hat{\Theta} \in \widehat{X}_L^{(k)}} \frac{ \langle \fZ^{(k)} \hat{\Sigma}, \hat{\Theta} \rangle_{\Gamma^{(k)}}}{\|\hat{\Sigma}\|_{X^{(k)}} \|\hat{\Theta}\|_{X^{(k)}}} \geq \frac{1}{\hat{c}_S} >0.
\ee
\end{lemma}
Therefore, we deduce that the problem is well posed and we deduce the following convergence error in sparse tensor spaces:
\begin{lemma}[{\cite[Theorem 5.3]{vonPetersdorff2006}}]
Provided that $\Sigma^k \in (X^s)^{(k)}$ for any $0\leq s\leq p+1$, the following error bound holds for $L \geq L_0(k)$:
$$\|\Sigma^k - \hat{\Sigma}^k_L\|_{X^{(k)}} \leq C h^s |\log h|^{(k-1)/2} \|\Sigma^k\|_{(X^{s})^{(k)}}.
$$
\end{lemma}
We solve the Galerkin system in the sparse tensor space applying the CT \cite{griebel1990combination}. It consists in solving the full systems for $\underline{l}$ specified in \cite[Theorem 13]{combi2} and for associated spaces $X^k_{\underline{l}}$ as described below.\\
\begin{problem}[Tensor Galerkin - Subblocks]\label{pb:sub}Given $k\in \IN_2$, seek $\Sigma_{\underline{l}}^k  \in X_{\underline{l}}^{(k)}$ such that
\be
\langle( \fZ_{l_1} \otimes \cdots \otimes \fZ_{l_k})\Sigma^k_{{\underline{l}}}  , \Theta^k_{{\underline{l}}} \rangle_{\Gamma^{(k)}} = \langle \fB^{(k)} \mM^k[g] , \Theta_{{\underline{l}}}^k\rangle_{\Gamma^{(k)}}, ~ \forall \Theta^k_{{\underline{l}}} \in X_{\underline{l}}.
\ee
\end{problem}
Thus, following \cite[Lemma 12 and Theorem 13]{combi2}, the Galerkin orthogonality allows to rearrange the solution in the sparse tensor space as
\be
\label{eq:lvl_CT}
\hat{\Sigma}_{L}^k(L_0) = \sum_{i=0}^{k-1} (-1)^i \dbinom{k-1}{i}\sum_{\|\underline{l}\|_1 = L + (k-1)L_0 - i}\Sigma^k_{\underline{l}}.
\ee
The total number of degrees of freedom (dofs) is of order $dofs = \mO(N_L \log^{k-1} N_L)$. 

Finally, we plug the unknowns $\xi_L,\hat{\Sigma}^k_L$ into the volume reconstruction formulas presented in \Cref{subs:tensor_bie} and obtain the couple 
\be
\U_L(\bx) \textup{, and } \widehat{\mM^k[\U']}_L(\underline{\bx}) \textup{, for } \bx\in \mD, \underline{\bx} \in \mD^{(k)},
\ee
being the final approximate delivered by the method. 

\begin{remark}[Affine meshes]Meshing by planar surface elements induces a geometrical error, which typically limits the order of convergence of Galerkin BEM to $\mO(h^2)$. Following \cite[Chapter 8]{sauter}, we present in \Cref{tab:ConvergenceRates} the conjectured convergence rates for $\IP^1$ discretization with affine meshes for the mean field and two-point covariance for both Neumann and Dirichlet trace counterparts for (B$_\boldbeta$).\end{remark}

\begin{figure}[t]
\begin{minipage}{0.49\linewidth}
\centering
\vspace{-0.2cm}
\begin{table}[H]
\renewcommand\arraystretch{1.3}
\begin{center}
\footnotesize
\begin{tabular}{|c|c|c|} \hline   
  &  Dirichlet traces  &  Neumann traces  \\ \hline
 Norm &  $\|\cdot\|_{H^{1/2}(\Gamma)}$  & $\|\cdot\|_{H^{-1/2}(\Gamma)}$  \\ \hline\hline
 $\xi_L$ &  $h^{3/2}$ & $h^{2}$  \\ \hline
 $\Sigma^k_L$ & $h^{3/2}$ & $h^{2}$ \\ \hline 
 $\hat{\Sigma}^k_L$ & $h^{3/2}|\log h|^{(k-1)/2}$ & $h^2|\log h|^{(k-1)/2}$  \\ \hline 
\end{tabular}
\vspace{-0.15cm}
\caption{}
\label{tab:ConvergenceRates}
\end{center} 
\end{table}  
\end{minipage}
\begin{minipage}{0.49\linewidth}
\vspace{-0.15cm}
\begin{table}[H]
\renewcommand\arraystretch{1.3}
\begin{center}
\footnotesize
\resizebox{5cm}{!} {
\begin{tabular}{|c|c|c|} \hline   
 FOA & First-Order Approximation\\ \hline 
SD & Shape Derivative\\ \hline
FOSB & First-Order Sparse Boundary\\ \hline
BIE & Boundary Integral Equation \\ \hline
MC & Monte-Carlo \\ \hline
CT & Combination Technique \\ \hline
\end{tabular}}
\end{center} 
\vspace{-0.2cm}
\caption{}
\label{tab:Acronyms}
\end{table}  
\end{minipage}
\vspace{-0.6cm}
\caption*{\Cref{tab:ConvergenceRates} (left): Expected convergence rates for the quantities of interest for $k\in\IN_2$ with $\IP^1$ discretization and affine meshes. \Cref{tab:Acronyms} (right): Non-exhaustive list of acronyms.}
\end{figure}
\section{Implementation considerations}
\label{sec:further}
In what follows, we aim at understanding several technical aspects related to the implementation of the FOSB scheme. 
\subsection{Symmetric covariance kernels}
\label{subs:cov}
Consider the case $k=2$ for a solution $\Sigma^2 \equiv \Sigma$. In most applications, the right-hand side is a  symmetric pseudo-covariance kernel, which entails a symmetric solution $\Sigma(\bx_1,\bx_2)=\Sigma(\bx_2,\bx_1)$. Therefore, the sparse tensor approximation or the CT allow for a two-fold reduction of the dofs for a given accuracy, since for any $l_1,l_2 \in \IN_0$, the matrix representation of unknowns reads $\bSigma_{l_1,l_2}=\bSigma_{l_2,l_1}^T$, its transpose. We express the latter in \Cref{tab:CombSym}, for $(L_0,L)=(0,5)$ and $(2,5)$ and for the test case that we detail further in \Cref{subsec:TP}, and giving $N_L^2=595,984$ in the full tensor space $V_L^{(2)}$, evidencing the efficiency of the CT and the benefits due to symmetry of the solution.  Indeed, for $L,L_0 \in \IN_0$, $L_0 \geq L$, the CT yields:
\be
\hat{\Sigma}_{L}(L_0) = \sum_{l_1+l_2 = L + L_0}\Sigma_{l_1,l_2}-\sum_{l_1+l_2 = L + L_0-1}\Sigma_{l_1,l_2},
\ee
while its symmetric counterpart uses a reduced number of subblock indices:
\begin{align*}
\hat{\Sigma}_{L}(L_0)& = \sum_{\substack{l_1+l_2 = L + L_0 \\ l_2<l_1}}(\Sigma_{l_1,l_2}+\Sigma_{l_2,l_1})-\sum_{\substack{l_1+l_2 = L + L_0-1\\ l_2 < l_1}}(\Sigma_{l_1,l_2}+\Sigma_{l_2,l_1})\\
&+ \Sigma_{(L+L_0)/2} -\Sigma_{(L+L_0)/2-1}\textup{, if } L+L_0 \textup{ is odd}.
\end{align*}
The remark applies identically to complex Hermitian covariance matrices, as $\bSigma_{l_1,l_2}=\overline{\bSigma_{l_2,l_1}}^T=\bSigma_{l_2,l_1}^H$ (see \Cref{remark:complex_moments}) and can be directly generalized for higher moments.

\begin{table}[t]
\renewcommand\arraystretch{1.4}
\begin{center}
\footnotesize
\begin{tabular}{
    |>{\centering\arraybackslash}m{2.8cm}
    >{\centering\arraybackslash}m{2.8cm}
    |>{\centering\arraybackslash}m{2.8cm}
    >{\centering\arraybackslash}m{2.8cm}|
    }\hline
  \multicolumn{2}{|c|}{$L=5$ and $L_0=0$} & \multicolumn{2}{c|}{$L=5$ and $L_0=2$}\\\hline
  \multicolumn{2}{|c|}{\includegraphics[width=0.4\textwidth]{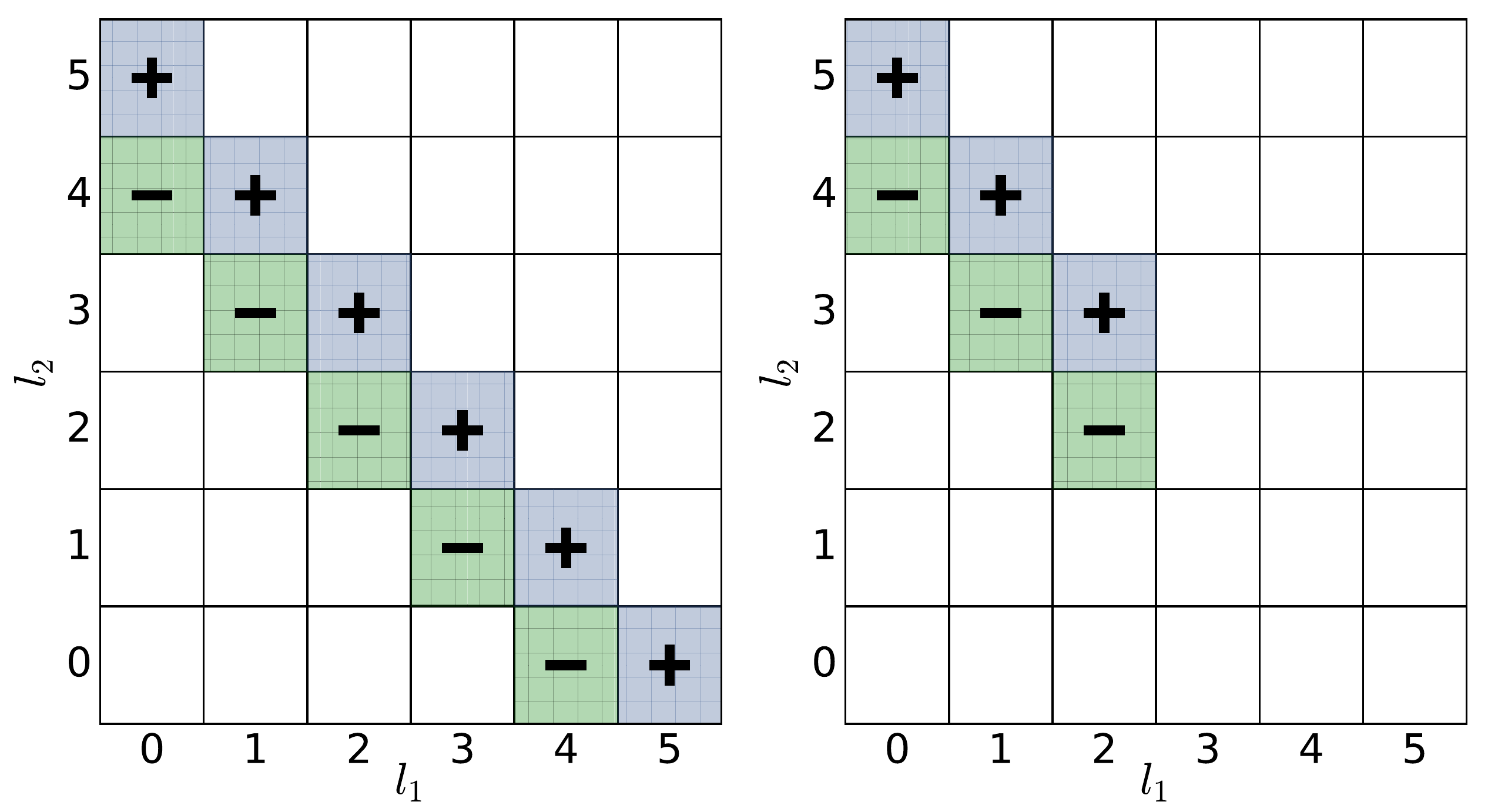}} & \multicolumn{2}{c|}{\includegraphics[width=0.4\textwidth]{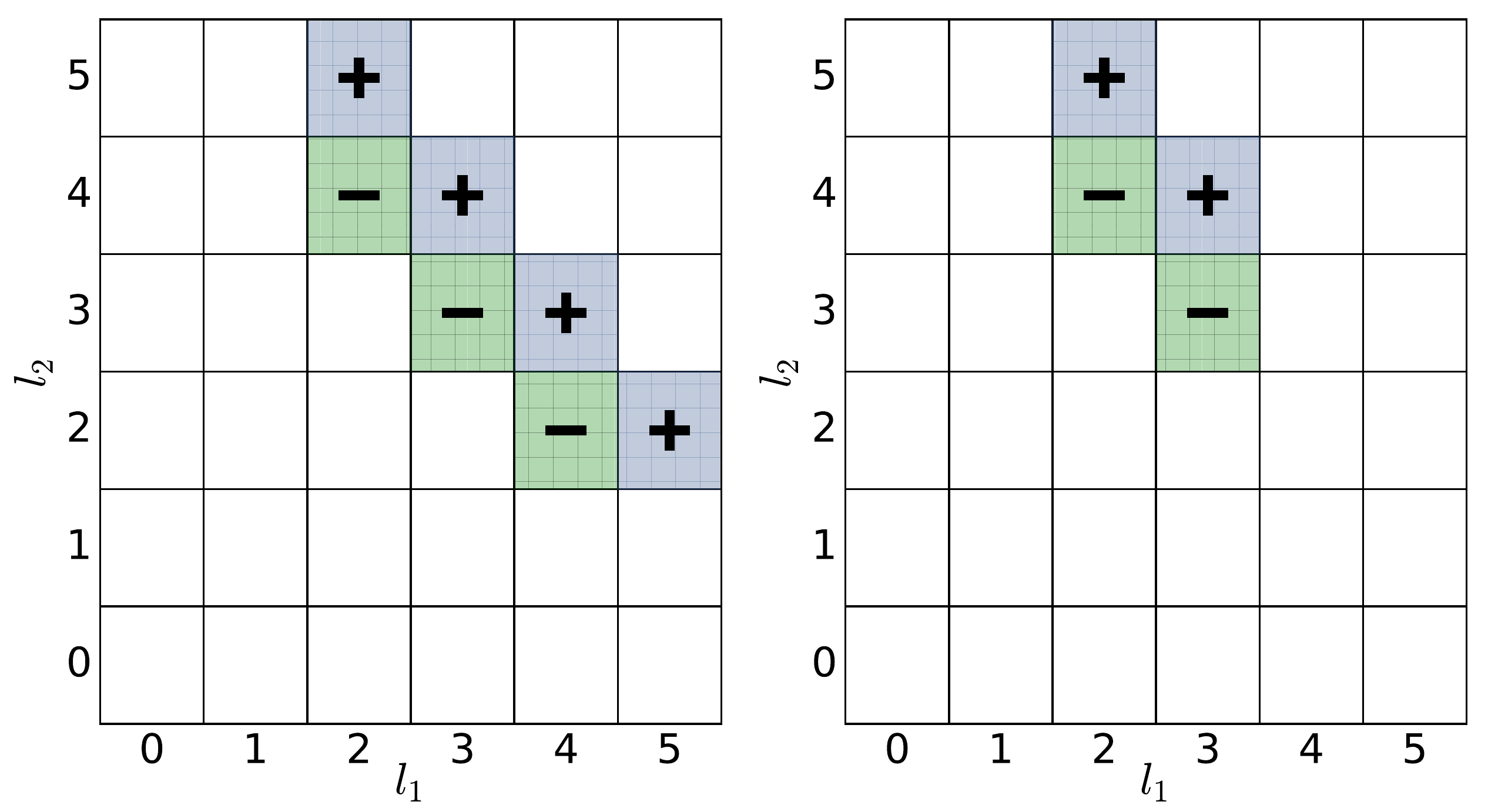}}\\ \hline
  ~~~~~~~~~$165,740$ & $87,672$& ~~~~~~~~~$413,980$&$225,808$\\\hline
 \end{tabular}
\caption{Subspaces used for the classical CT (left of each cell) and the symmetric CT (right of each cell) for $k=2$. In last row, we detail dofs of each scheme. Notice that $N_L^2=595,984$.}
\label{tab:CombSym}
\end{center} 
\vspace{-.2cm}
\end{table}  
\subsection{Preconditioning}
\label{subs:prec}
The CT allows to solve smaller subsystems by gathering the operators assembled over distinct levels --on the indices stated in \Cref{eq:lvl_CT}. It is known that the condition number of tensor operators grows with the dimension (\textit{cf.}~\cite[Section 3]{griebel2009optimized} and \cite{FJM19}). Hence, the need to precondition with an adapted framework such that the linear systems remains at least mesh independent. We opt to apply operator-based preconditioners such as Calder\'on preconditioning \cite{hiptmair2006operator,IEEE_EJH} and assume for $\beta=0,1,2$ that $\kappa^2 \notin \{S_\textup{Dir},S_\textup{Neum}\}$. On each level, we apply the preconditioner $\fC$ proposed in \Cref{tab:NotationsBIEs}. We propose the following result for the induced linear system in \Cref{lemma:mesh_ind}.
\begin{lemma}[Mesh independence result]
\label{lemma:mesh_ind}
For $k \in \IN_1$ and for each ${\underline{l}} =(l_k)_k $ with $l_k \in \IN_0$, $l_k \geq L_0(k)$ such as defined in \Cref{sec:galerkin_CT}, the discretized Galerkin system issued from operator $(\fC \fZ)^{(k)}$ has a spectral condition number $\kappa_2$ independent of the mesh size, i.e.~remains bounded as $\|{\underline{l}}\|_1 \to \infty$. 
\end{lemma}
\begin{proof}
The result is proved for $k=1$ in \cite{hiptmair2006operator} and applies straightforwardly to $k\geq 2$ as the condition number of tensor operators is multiplicative \cite{FJM19}.
\end{proof}
This last result shows mesh independence of the numerical scheme, i.e. it guarantees the $h$-stable (linear) convergence of GMRES (refer to \cite[Section 4]{galkowski2016wavenumber}). Also, as the domains have a Lyapunov boundary, $\fK_\kappa$ is compact in both $H^{1/2}(\Gamma)$ and $L^2(\Gamma)$ (see the discussion after Theorem 2.49 in \cite{chandler2008wave}). Consequently, the induced operators are second-kind Fredholm operators of the form $\fI + \fK :X\to X $ \cite{antoine2016integral} with $X$ a separable Hilbert space, which entails fast asymptotic convergence (i.e.~super-linear) of iterative solvers \cite{van1993superlinear}. Additionally, $L^2$-compactness is advantageous as it is naturally suited to the euclidean norm-based GMRES \cite{campbell1996convergence}. Still,  second-kind Fredholmness is not transferred to $(\fI + \fK )^{(k)}$, as the cross-terms are not compact, e.g., for $k=2$, $\fI \otimes \fK$ and $\fK \otimes \fI$ are not compact at a continuous level. 
\begin{remark}[Non-compactness of cross-terms]As stated in \cite[Corollary 1]{zanni2015note} for $X$ a Hilbert complex space and $\fA,\fB\in \mL(X)$:
\be
\fB \otimes \fA \textup{ is nonzero and compact } \iff \fA \textup{ and } \fB \textup{ are both nonzero and compact.}
\ee
Suppose that $\fK$ is nonzero and $\fI \otimes \fK$. Therefore, $\fI$ is compact, which is a contradiction, proving that $\fI \otimes \fK$ is not compact. Similarly, we ensure that $\fK \otimes \fK$ is compact.
\end{remark}
Despite the above, super-linear convergence for the higher moments is likely: clustering properties of $\fA$ are surprisingly transferred to the tensor operator as hinted by the next result.
\begin{theorem}[Clustering properties of the tensor matrix equation]
\label{theorem:cluster}
Consider that $\fA = \fI + \fK:X\to X$ with $X$ a separable Hilbert space and $\fK$ a compact operator. Therefore, for $k\in \IN_2$, the discretized system of $\fA^{(k)}$ has a cluster at $1$.
\end{theorem}
\begin{proof}
As $\fK$ is compact, its singular values $\sigma_j(\fK)$, $j=1,\cdots$ with $\sigma_j(\fK)\to 0$ as $j\to \infty$. Therefore, the singular values of $(\fI \otimes \fK)$ give $\sigma_{j,l}(\fI \otimes \fK )= \sigma_j(\fK) \to 0$. Therefore, $\fI \otimes \fK$ has a cluster at $1$. The same proof applies to any cross-term. Finally, as the constants in the asymptotic bounds are independent of the mesh side, the clustering property transfers at discrete level.
\end{proof}
To quantify a possible super-linear behavior at iteration $m \in \IN_1$, we introduce $\br_m$ the GMRES residual and the convergence factor given by the following $m$-th root:
\be
Q_m := \left(\frac{\|\br_m\|_2}{\|\br_0\|_2} \right)^{1/m}.
\ee
Notice that the super-linear behavior shows up in the final phase of convergence of Krylov solvers and ``is often not seen unless one iterates to very small relative errors and the condition number is large'' (\cite[Section 13.5]{axelsson1996iterative}).
\subsection{Wavenumber analysis}
\label{subs:wavenumber}
In this manuscript, we focus on first-kind BIEs preconditioned via Calder\'on identities. Still, the proposed technique does not give results concerning the wavenumber dependence in the constants of: (i) the FOA; (ii) the quasi-optimality constant of the sparse tensor approximation; and, (iii) the condition number. Despite being out of the scope of this manuscript, we aim at giving a few remarks about the analysis for high wavenumbers. The smoothness of the domains considered here hints at using non-resonant $L^2$-combined field formulations \cite[Section 3.9.4]{sauter}, but would require a more complex analysis to prove enough regularity for the shape derivative, namely to prove that Cauchy data $(\lambda_i,\sigma_i)$ for the SD $i=0,1$ belong to $H^1(\Gamma)$ and $L^2(\Gamma)$ respectively. Furthermore, extensive results were proved for (P$_\boldbeta$), and the analysis could be carried on, under additional restrictive requirements on the domain such as star-shapedness (refer e.g., to \cite{spence2014wavenumber,galkowski2018optimal,galkowski2016wavenumber}. Those surveys can lead to elliptic formulations, allowing for application of C\'ea's lemma \cite{cea1964approximation}, with a simple characterization of the $\kappa$-dependence of the constants of (ii) and (iii). Concerning item (i), we expect the constant to be specified with the help of the BVP for the shape Hessian, provided sufficient regularity of both domain and transformations. Furthermore, star-shapedness is a classical assumption for the UQ by random domains, as it allows to handle the domain transformations more easily.

\section{Numerical Results}
\label{sec:Numres}
We now apply the proposed technique to realistic applications. In order to investigate the accuracy of the first-order shape approximation, in \Cref{subsec:SSSH} we analyze with the shape sensitivity analysis of sound-soft and -hard problems for a kite-shaped object. Thus, the transmission problem is set over the unit sphere and focus is set on the CT for the two-point covariance field i.e.~$k=2$. The error convergence rates for the CT relying on the Mie series are analyzed in \Cref{subsec:TP}. Finally, the behavior of GMRES is discussed in \Cref{subsec:TP_Iterative} and the FOSB is compared to MC simulation for a complex case in \Cref{subsec:Non_Smooth}. Domains are excited by a plane wave polarized along the $x$-direction, i.e.~$\U^\textup{inc}(\bx) = e^{\imath \kappa x}$, with $\bx=(x,y,z)\in \IR^3$. 

All simulations are performed via the open-source Galerkin boundary element library Bempp 3.2 \cite{Bempp}.\footnote{https://bempp.com/download/} The induced linear systems are preconditioned by strong-form multiplicative Calder\'on preconditioning (\textit{cf.}~\cite{kleanthous2018calderon}). Tests are executed on a 32 core, 4 GB RAM per core, 64-bit Linux server using Python 2.7.6. Default parameters used throughout are the following: linear systems are solved with restarted GMRES(200) \cite{gmres_book} with a tolerance of $10^{-4}$. Simulations are accelerated with Hierarchical Matrices ($\mH$-mat) \cite[Chapter 2]{bebendorf2007hierarchical} combined with the Adaptive Cross Approximation algorithm (ACA) \cite[Section 3.4]{bebendorf2007hierarchical}. The relative tolerance for ACA is set to $10^{-5}$. Meshes and simulations are fully reproducible using pioneering Bempp-UQ platform, a documented Python/Bempp-based plug-in including Python Notebooks.\footnote{https://github.com/pescap/Bempp-UQ}

In our simulations, we shall represent the polar radar cross section (RCS) over the unit circle $\IS^1$ and in decibels (dB) defined by:
$$
\textup{RCS}_t(\theta) :=  10 \log_{10}\left(4 \pi \frac{ |\F_t^\textup{scat}|^2}{|\F^\textup{inc}|^2}\right) ,~\theta:=\atantwo(y,x) \in [0,2\pi].
$$
\subsection{Kite-shaped object: FOA analysis}
\label{subsec:SSSH}
First, we aim at evidencing FOA's accuracy. For this, we introduce a kite-shaped object perturbed according to $\Gamma_t := \{  \bx + t \bv,~ \bx \in \Gamma \}$, with $\bv := [(z^2-1) (\cos(\theta)-1),0.25 \sin(\theta) (1-z^2),0]$ in Cartesian axes.
In \Cref{figs:transformedBoundariesK}, we represent the family of transformed boundaries considered here, corresponding to $t=\{0.01,0.1,0.25,0.5,1.0\}$. 
 \begin{figure}[t]
 \centering
\includegraphics[width=.8\linewidth]{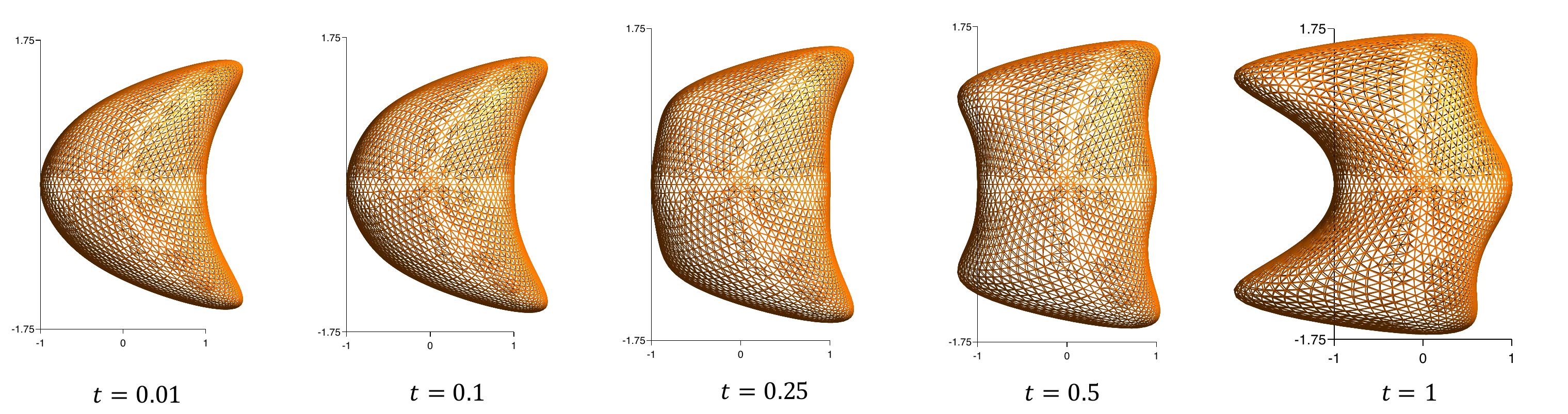}
\vspace{-.2cm}
\caption{Transformed boundaries function to $t$, meshed with $3,249$ vertices.}
\label{figs:transformedBoundariesK}
\end{figure}
For wavenumbers $\kappa=1$ and $\kappa=8$, we illuminate the object for $\beta=0,1$ and solve (P$_\beta$) for all the values of $t$. We compare the far-field and RCS of $\U_t$ (in black) to:
$$
\begin{array}{lc}
\U, & \textup{ the zeroth order approximation (ZOA, in red), and}\\
(\U+ t \U'), &\textup{ the }\textup{FOA (in blue).}\\
\end{array}
$$
Notice that the zeroth order approximation has no role in the FOA scheme, but will be used throughout as an additional reference for comparison purposes. The Galerkin discretization for $\kappa=1$ (resp.~$\kappa=8$) was realized with a precision of $30$ (resp.~$20$) triangular elements per wavelength and led to a Galerkin matrix of size $N=3249$ (resp.~$N=9820$).

\Cref{tab:ResDB} presents RCSs for $\beta = 0,1$ and $\kappa=1,8$. For different values of $t$, we plot RCS in dB of $\U_t$ on the left along with the one of the FOA $(\U + t\U')$. The $x$-axis represents the translated angle $(\theta + \pi)$ in radians. We remark that (i) as expected, the FOA gives a proper approximation for small values of $t$, (ii) the approximation seems less accurate for the shadow region, due to the oscillatory behavior of the latter and (iii) there is an evident dependence of the quality of the approximation function to the wavenumber. As an effect, we see that the FOA is accurate in a wider range of values of $t$ for $\kappa=1$ than for $\kappa =8$. 

To corroborate these remarks, we plot in \Cref{tab:ResErr} on the left-side of each cell: the error $[\cdot]_{L^2(\IS^1)}$ for $\F$ and $(\F + t \F')$ the FOA for $$t\in \{0.025, 0.05, 0.075, 0.1, 0.125, 0.25, 0.5, 0.75, 1.00\}.$$ 
These figures evidence the predicted linear and quadratic errors of both zero and first order approximations (see \Cref{remark:FF_taylor}). Besides, we observe that the FOA is indeed more accurate that $\F$ for small values of $t$. Still, the accuracy range of the FOA decreases strongly with $\kappa$. For instance, for $\beta=0,1$ and $\kappa=1$, the FOA gives a $15\%$ error for $t\leq 0.5$. Dissimilarly for $\kappa=8$, the latter remains true only for $t\leq 0.1$ for $\beta=0$ and even gives an error of $20\%$ for $\beta=1$. 

The right-side of each cell in \cref{tab:ResErr} presents the RCS pattern  of $\U_t,\U$ and $(\U+t\U')$ for $(\kappa,t)=(1,0.5)$ and $(\kappa,t)=(8,0.1)$. Let us focus on $\kappa=8$ and for $t \geq 0.25$: the FOA is clearly out of its admissible range. Next, we detail further the relative errors: in \Cref{tab:ResErrlog}, we represent $[\cdot]_{L^2(\IS^1)}$ in a log-log scale function to $t$ and verify that for $\kappa=1$, the error rate are as expected. For $\kappa=8$, the FOA presents slightly reduced convergence rates for small values of $t$ due to discretization error.
\begin{table}[t]
\renewcommand\arraystretch{1.4}
\begin{center}
\footnotesize
\begin{tabular}{
    >{\centering\arraybackslash}m{0.8cm}
    |>{\centering\arraybackslash}m{5.4cm}
    |>{\centering\arraybackslash}m{5.4cm}
    }
\vspace{0.1cm}
\includegraphics[width=1\linewidth]{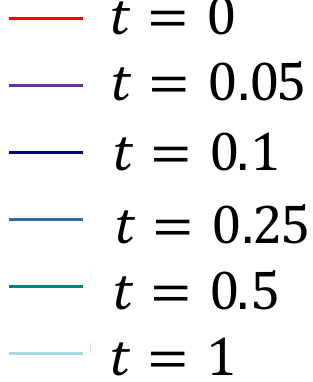}&  Sound-soft problem & Sound-hard problem \\ \hline
 $\kappa=1$&  \includegraphics[width=1\linewidth]{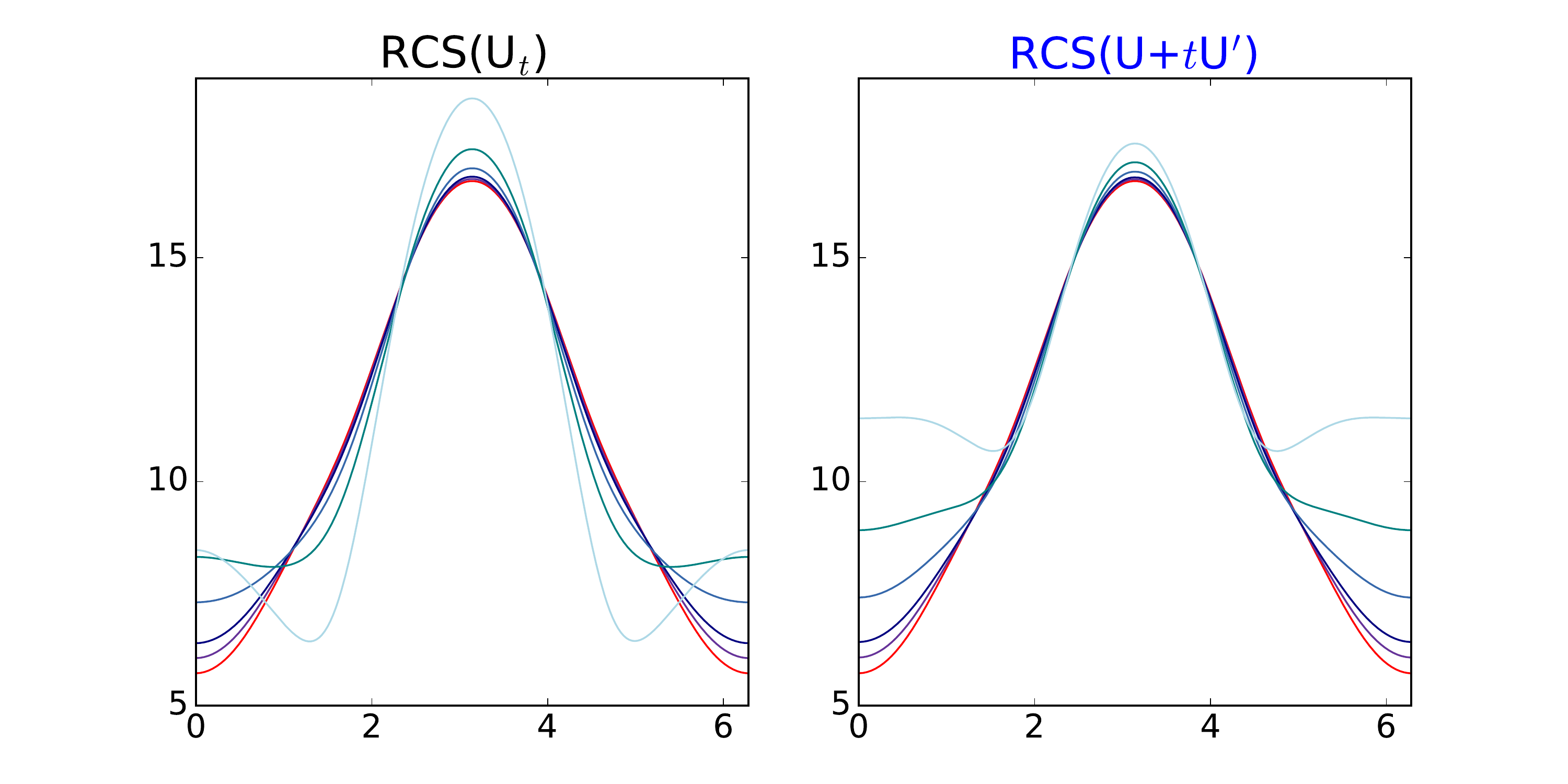} &  \includegraphics[width=1\linewidth]{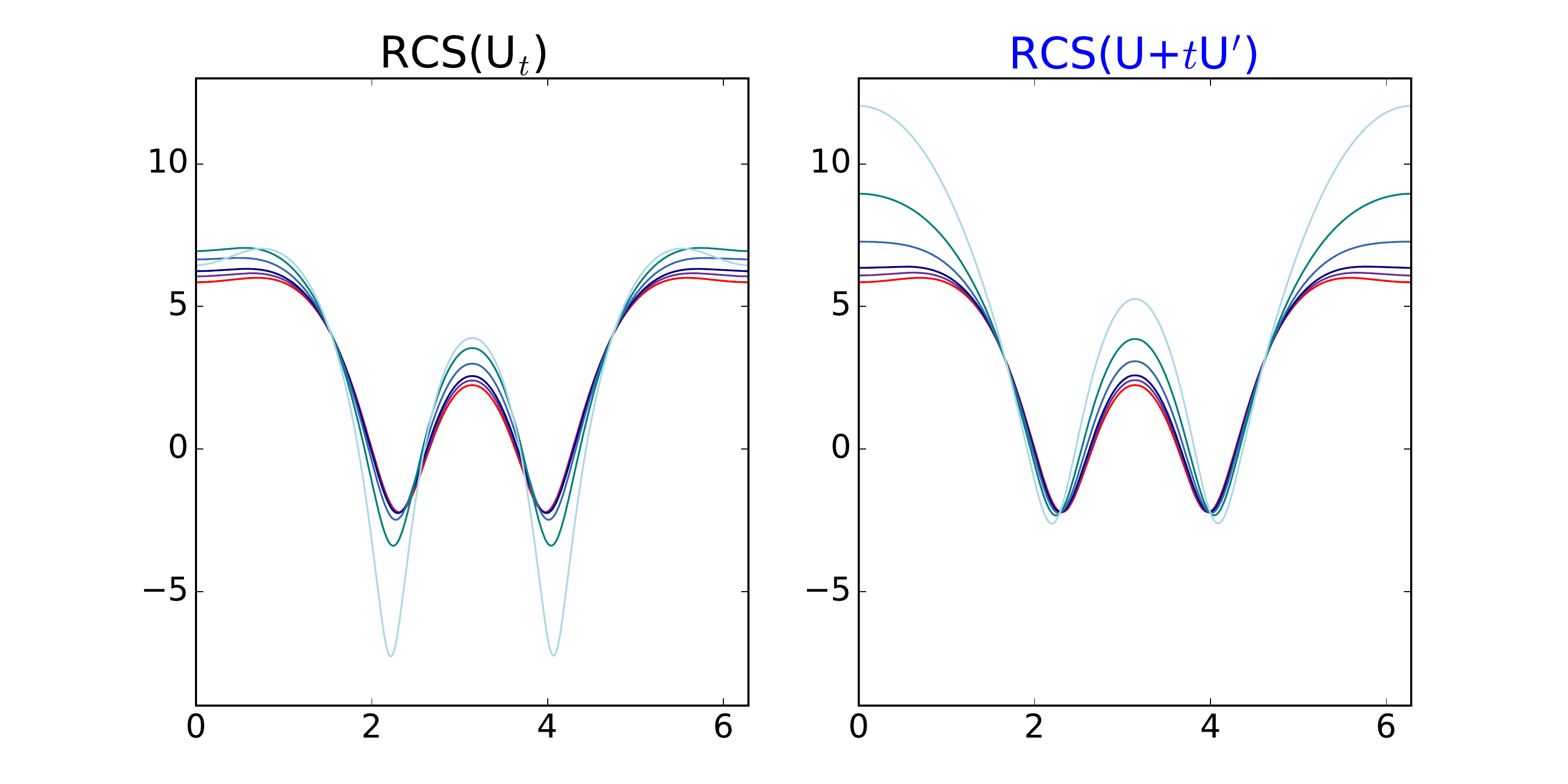}\\ \hline
 $\kappa=8$&  \includegraphics[width=1\linewidth]{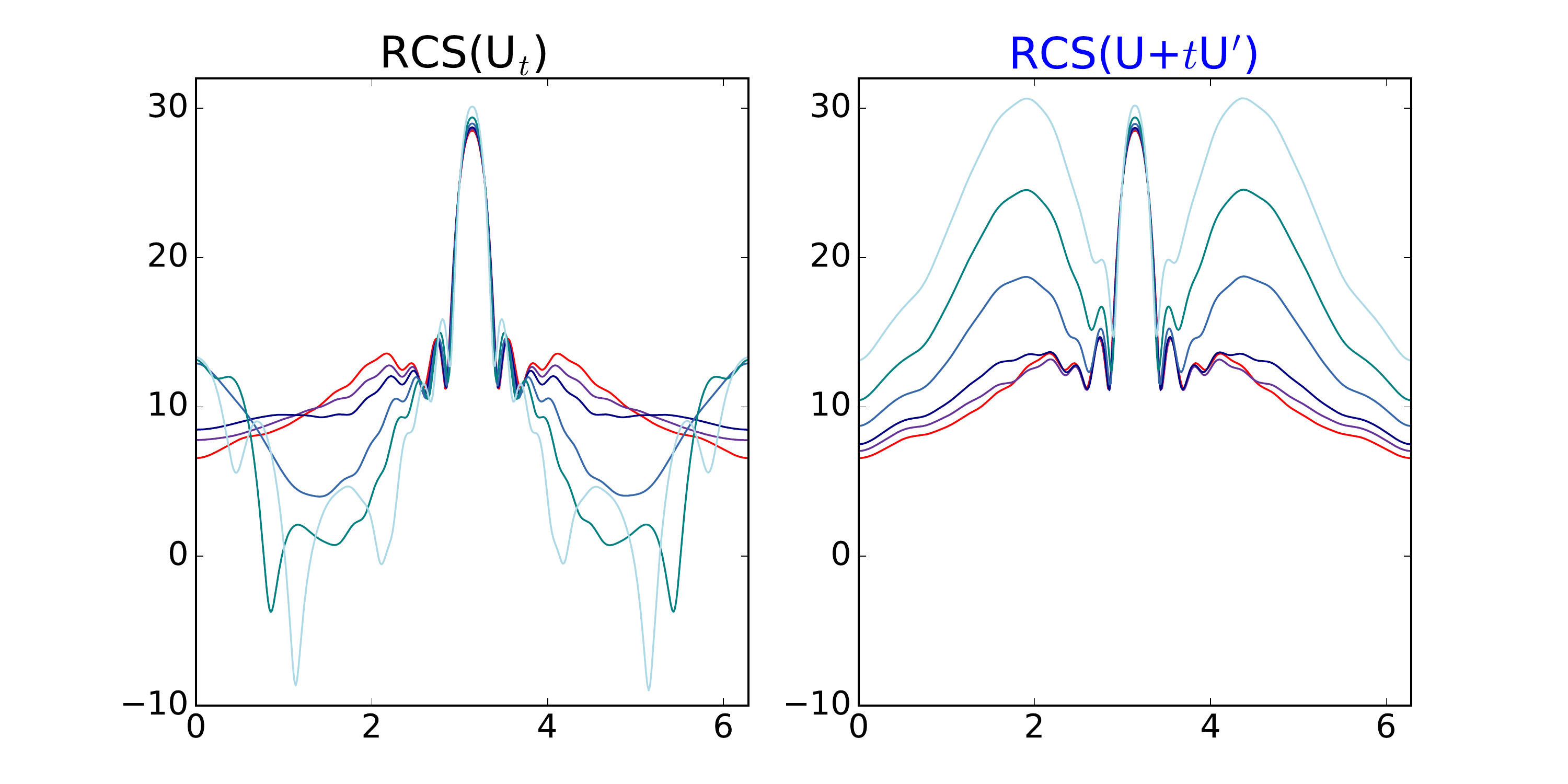} &  \includegraphics[width=1\linewidth]{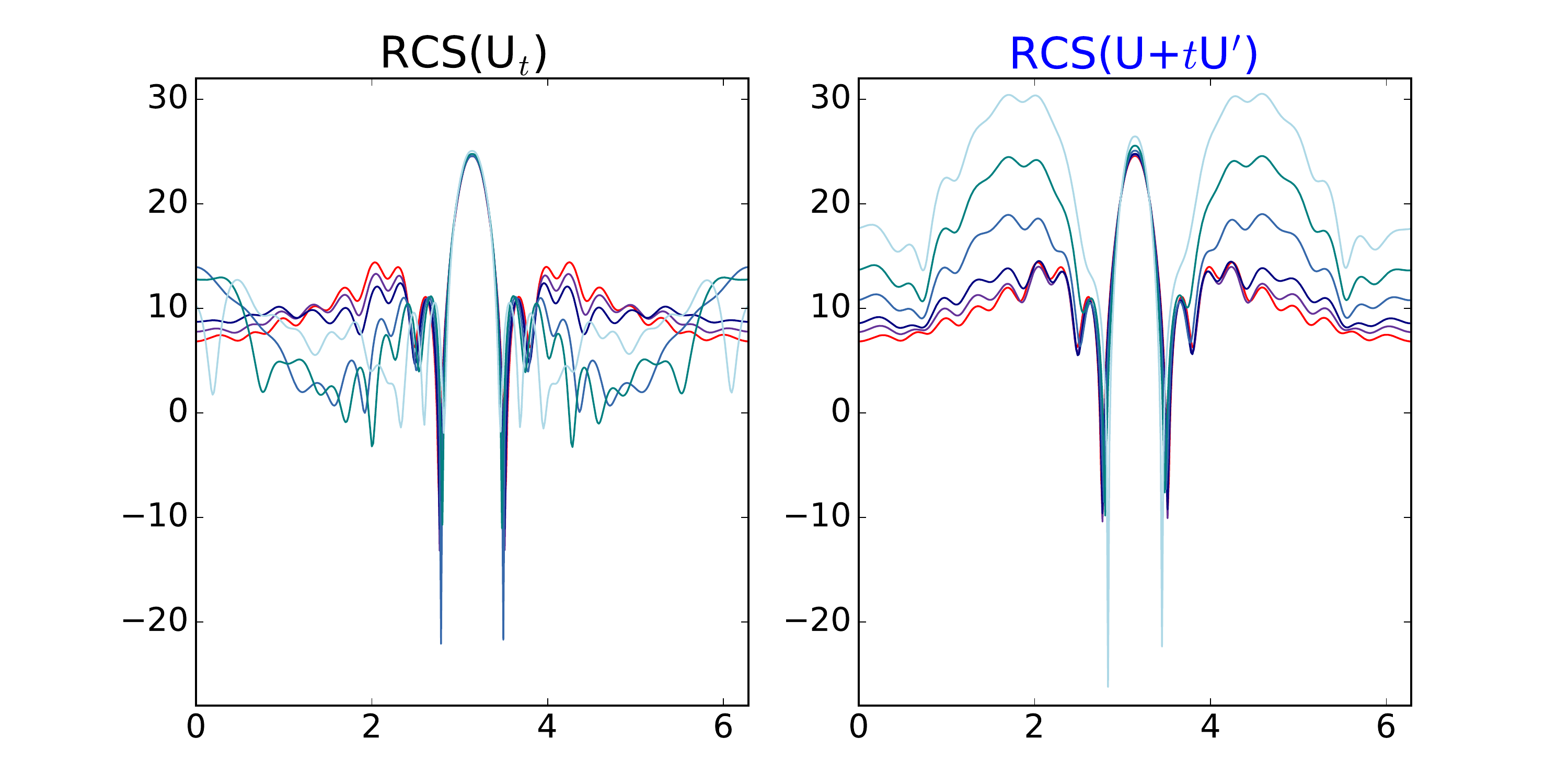}\\\hline
 \end{tabular}
\caption{RCS patterns (in dB) versus the angle $(\theta+\pi)$ in radians.}
\label{tab:ResDB}
\end{center} 
\end{table}  

\begin{table}[t]
\renewcommand\arraystretch{1.7}
\begin{center}
\footnotesize
\begin{tabular}{
    >{\centering\arraybackslash}m{0.8cm}
    |>{\centering\arraybackslash}m{5.4cm}
    |>{\centering\arraybackslash}m{5.4cm}
    }
\vspace{0.1cm}
&  Sound-soft problem & Sound-hard problem \\ \hline
 $\kappa=1$&  \includegraphics[width=1\linewidth]{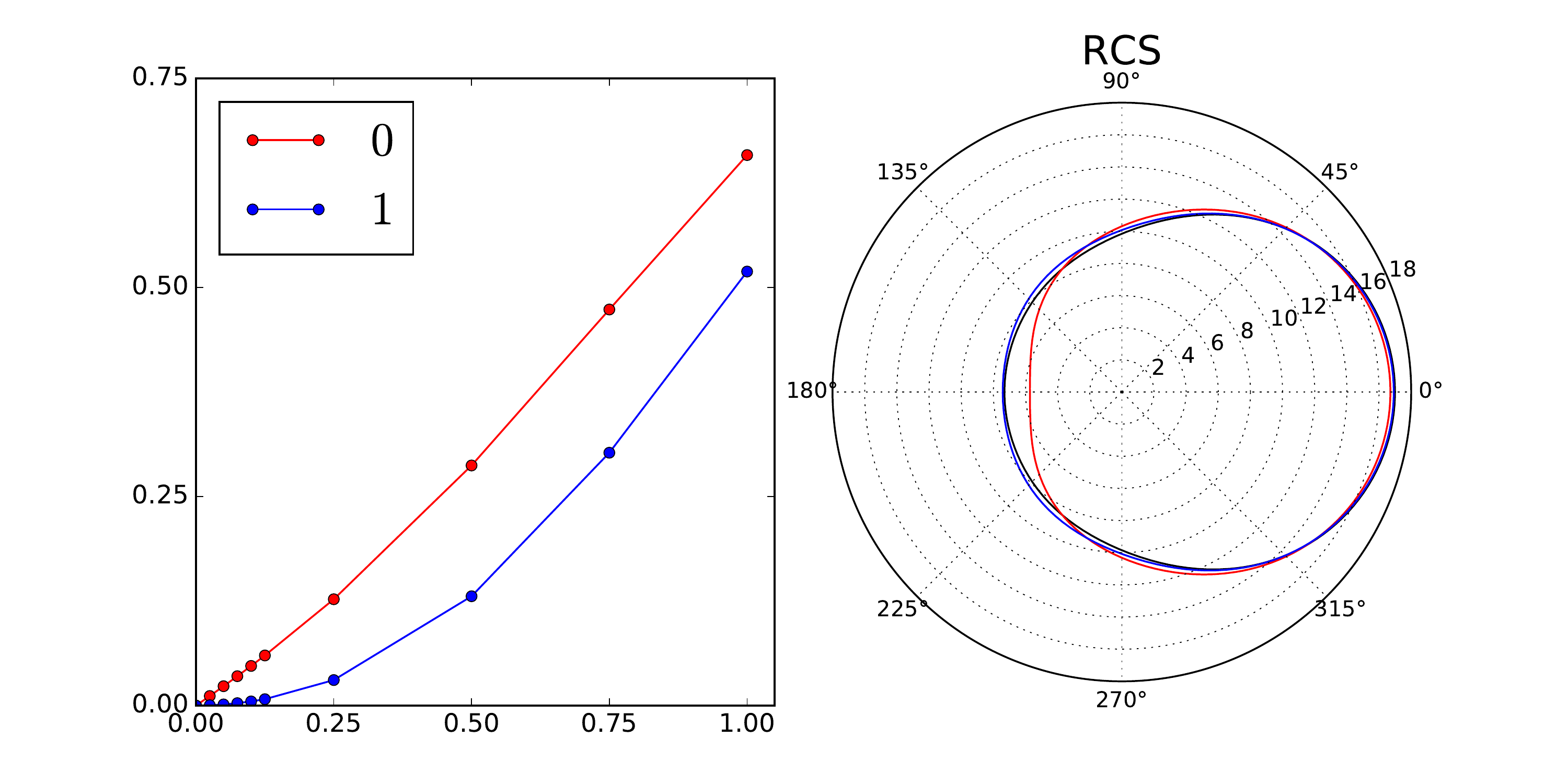} &  \includegraphics[width=1\linewidth]{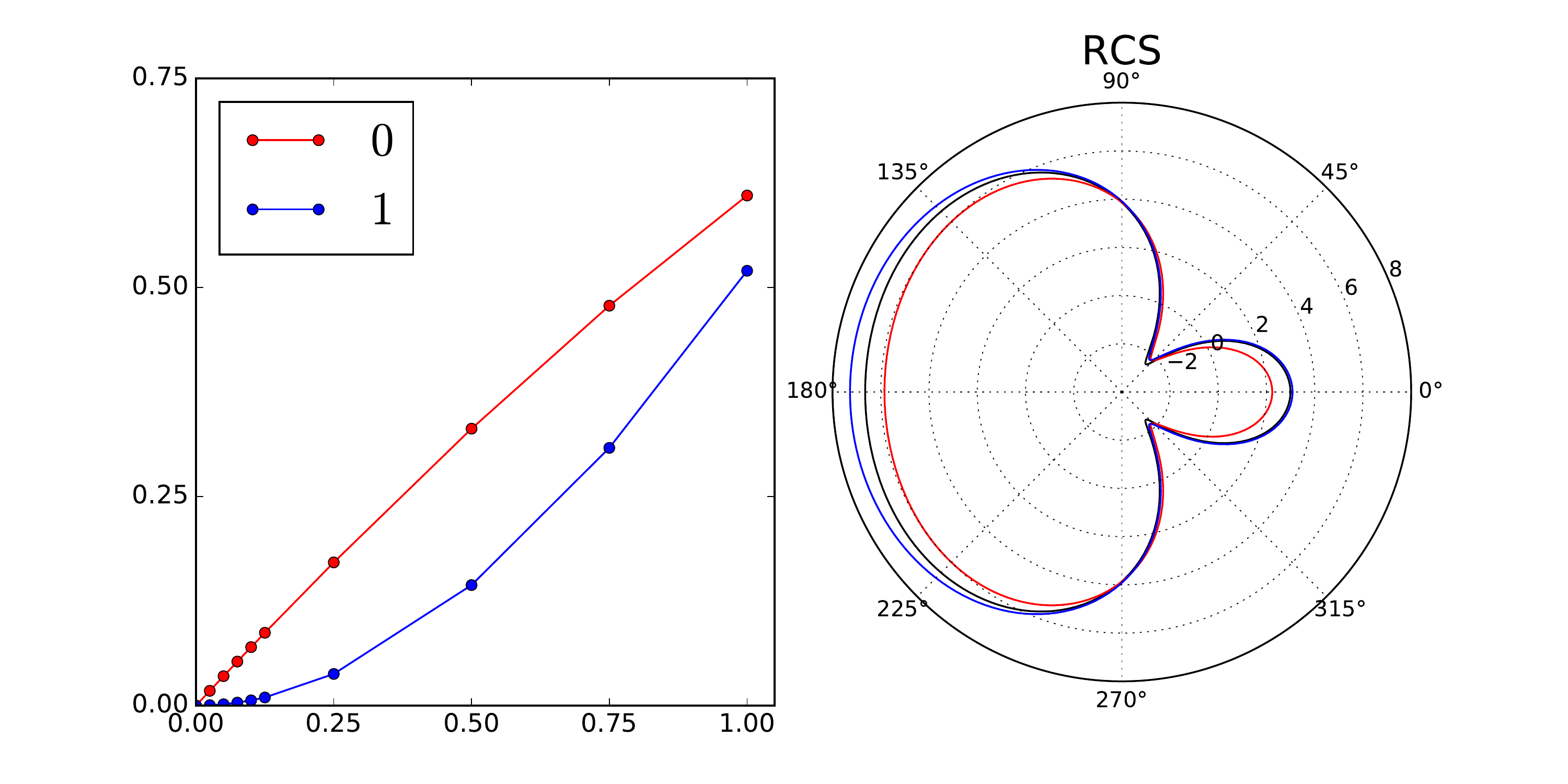}\\ \hline
 $\kappa=8$&  \includegraphics[width=1\linewidth]{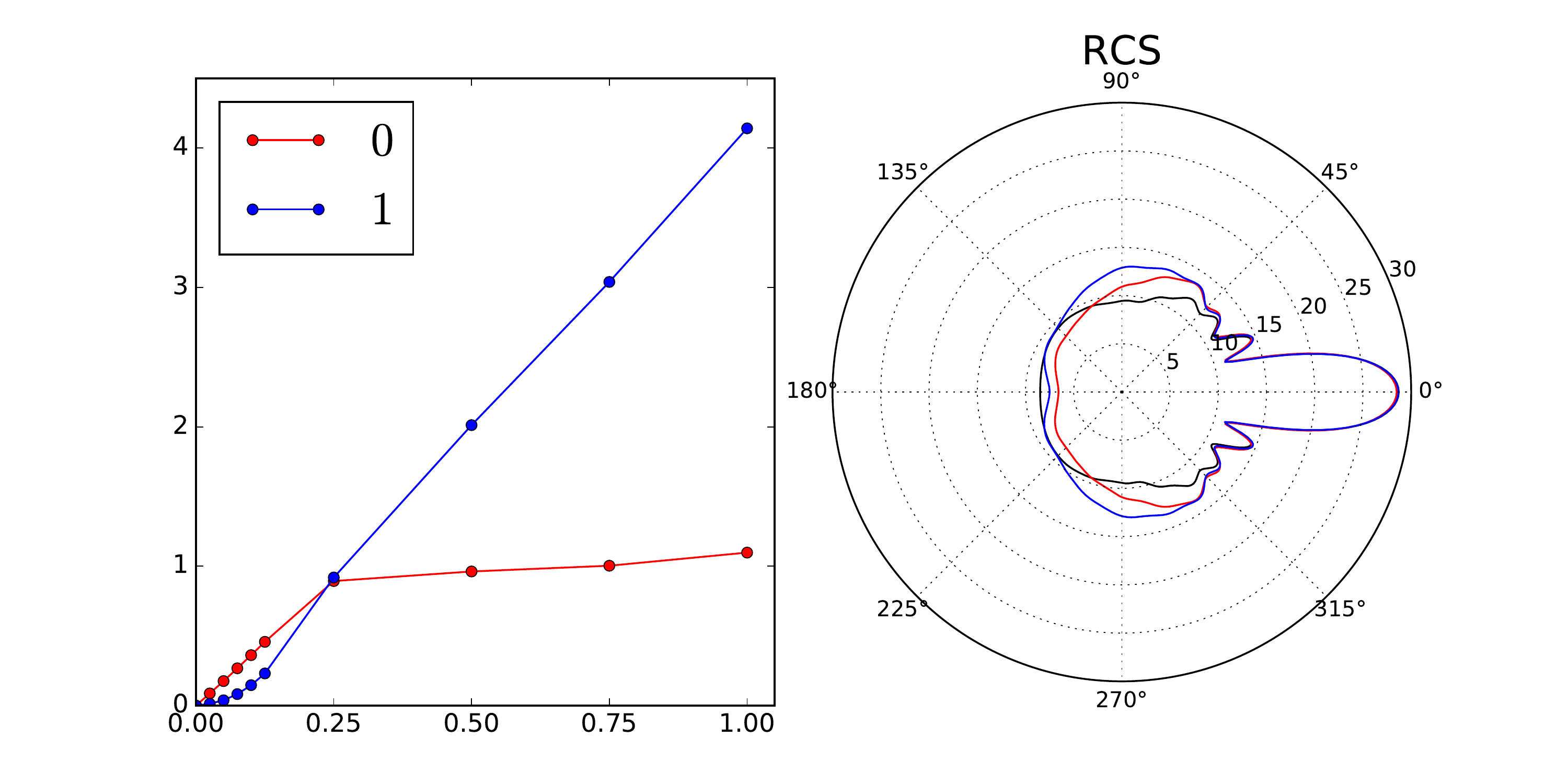} &  \includegraphics[width=1\linewidth]{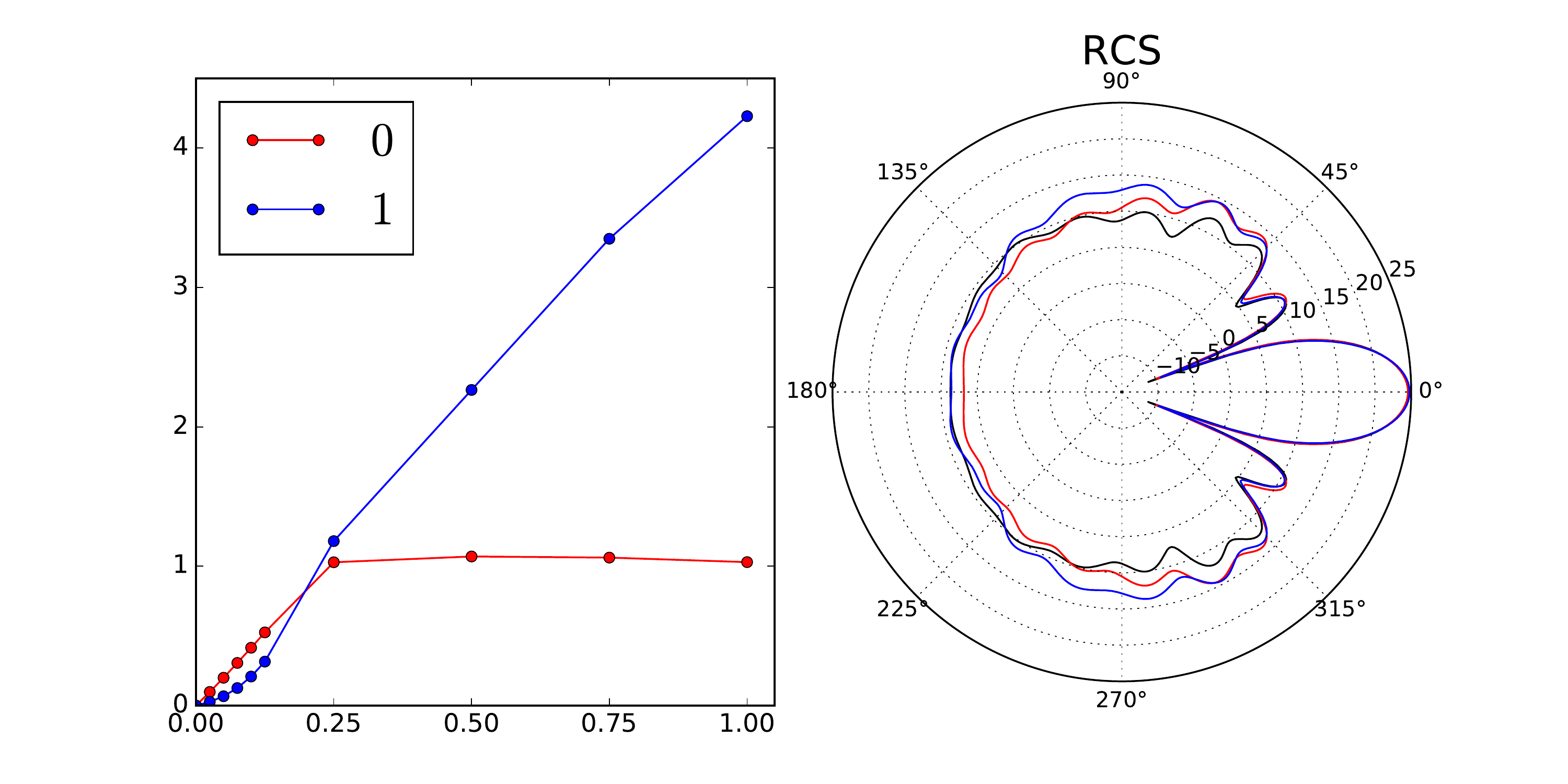}\\\hline
 \end{tabular}
\caption{ZOA (red) vs.~FOA (blue): Relative $L^2$-error on $\IS^1$ function to $t$ (left) and RCS patterns (in dB) for $(\kappa,t)=(1,0.25)$ and $(\kappa,t)=(8,0.1)$ (right).}
\label{tab:ResErr}
\end{center} 
\end{table}  

\begin{table}[t]
\renewcommand\arraystretch{1.7}
\begin{center}
\footnotesize
\resizebox{9cm}{!} {
\begin{tabular}{
    >{\centering\arraybackslash}m{0.8cm}
    |>{\centering\arraybackslash}m{5.4cm}
    |>{\centering\arraybackslash}m{5.4cm}
    }
&  Sound-soft problem & Sound-hard problem \\ \hline
 $\kappa=1$&  \includegraphics[width=.8\linewidth]{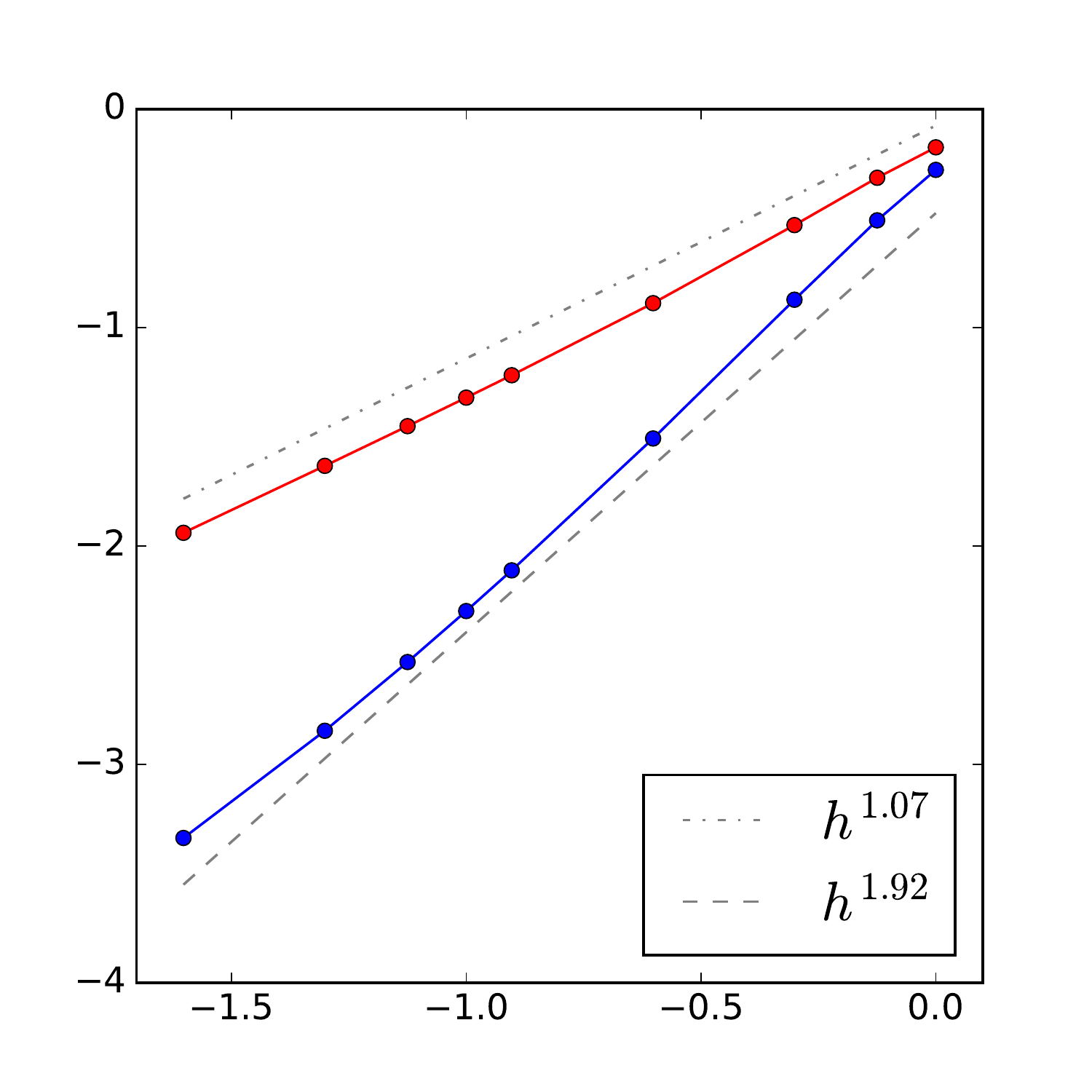} &  \includegraphics[width=.8\linewidth]{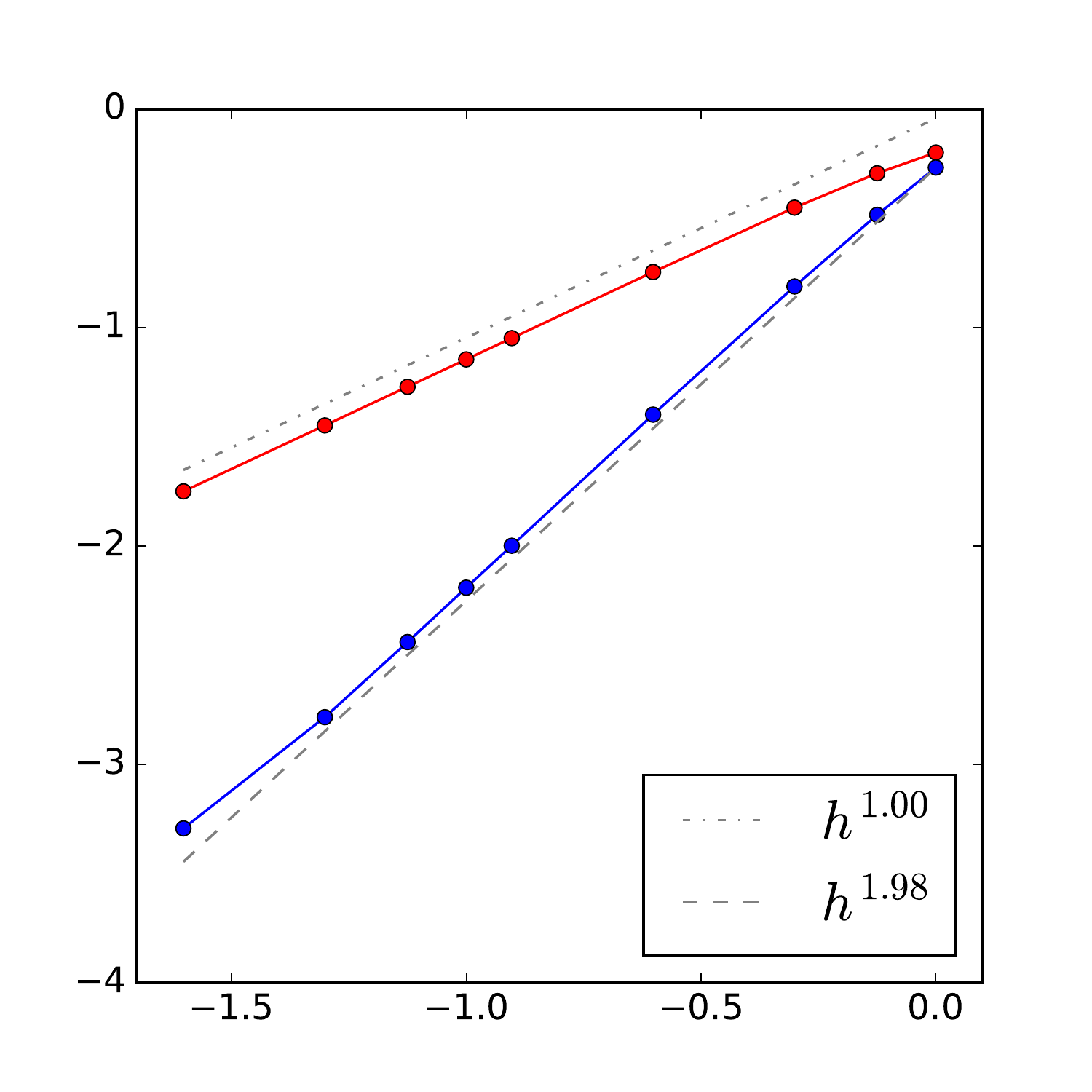}\\ \hline
 $\kappa=8$&  \includegraphics[width=.8\linewidth]{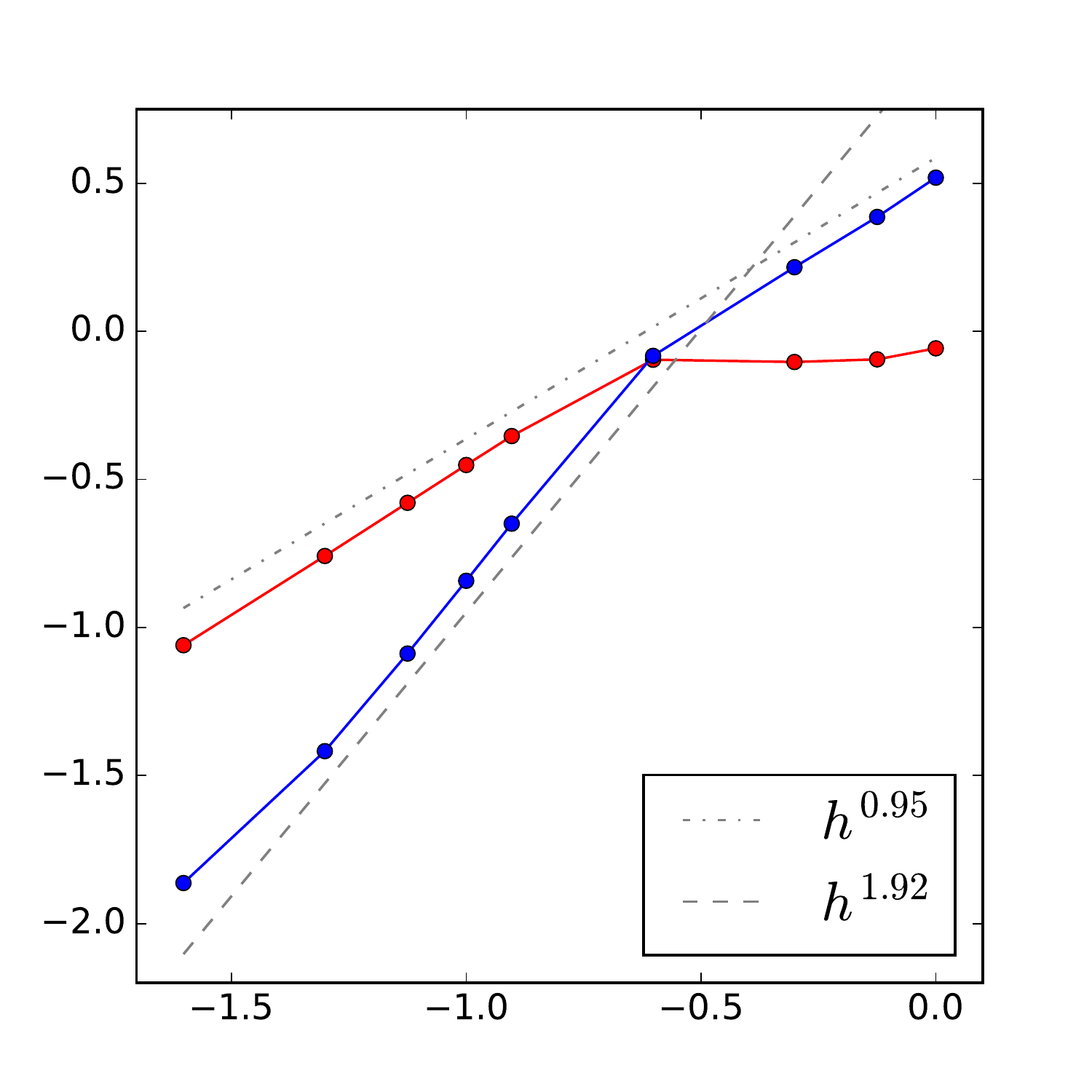} &  \includegraphics[width=.8\linewidth]{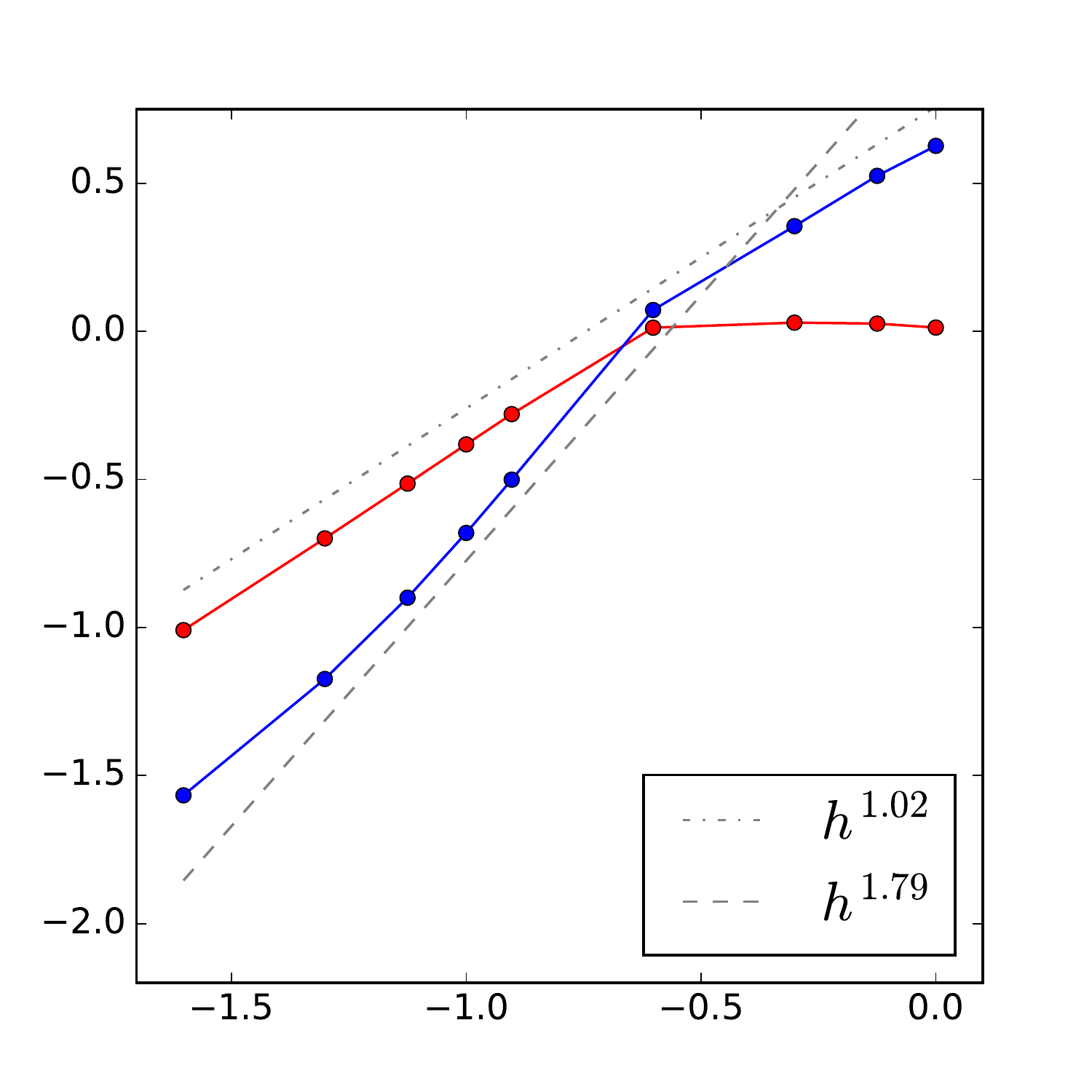}\\\hline
 \end{tabular}}
\caption{ZOA (red) vs.~FOA (blue): Relative $L^2$-error on $\IS^1$ function to $t$ in log-log scale.}
\label{tab:ResErrlog}
\end{center} 
\end{table}   
Henceforth, we aim at studying the wavenumber dependence of the approximates. We now fix $t=0.1$ and solve the problem for $\kappa \in \{1,\cdots,10\}$, with a precision of $20$ elements per wavelength for each $\kappa$. In \Cref{figs:kappa_dependence}, for $\beta=0,1$, we display the relative $L^2$-error of the approximates on $\IS^1$ function to $\kappa$. We notice a linear dependence of the error for $\U$ with respect to $\kappa$ and a dependence of order $\mO(\kappa^{3/2})$ for the FOA for $\beta=0$ and $\beta=1$, respectively. The curves show a stable asymptotic behavior function to the wavenumber. The latter hints at using $\kappa^{3/2}t=\mO(1)$ to keep an accuracy for the FOA bounded independently of the wavenumber. Notice that this estimate is more restrictive that the intuitive bound $ \kappa t=\mO(1)$ proposed in \cite{SAJ18}, confirming the need for a proper wavenumber analysis for the FOA (\Cref{subs:wavenumber}).
\begin{figure}[t]
 \centering
\includegraphics[width=0.6\linewidth]{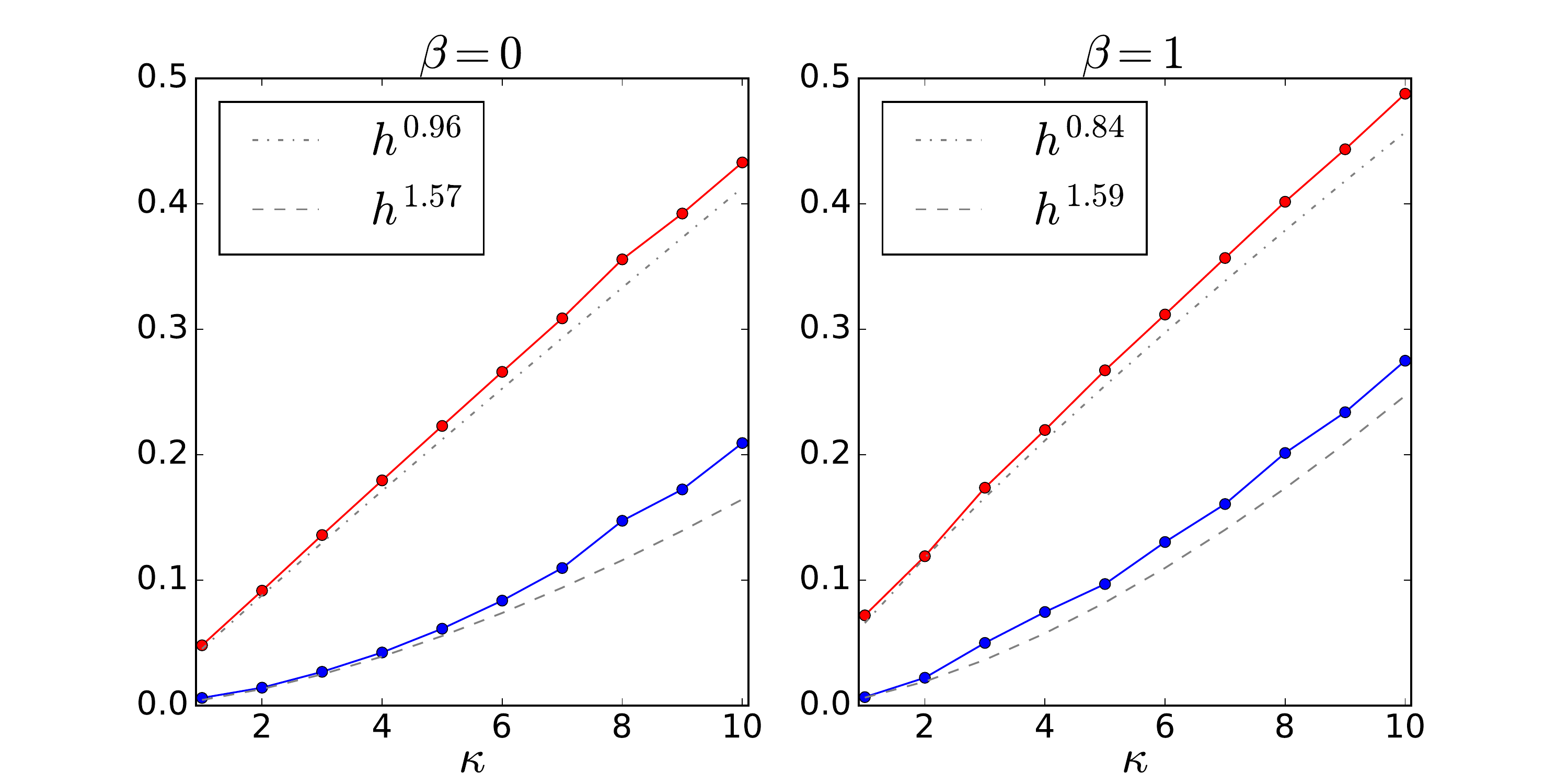}
\caption{ZOA (red) vs.~FOA (blue): Relative $L^2$-error on $\IS^1$ function to $\kappa$ for $\beta=0$ (left) and $\beta=1$ (right) and polynomial fit.}
\vspace{-.3cm}
\label{figs:kappa_dependence}
\end{figure}
\subsection{Unit sphere: convergence analysis}
\label{subsec:TP}
Consider the unit sphere $D:=\{\bx \in \IR^3: \|\bx\|_2 \leq 1\}$ and focus on the convergence rates for the mean field and second order statistical moment. In order to inspect the behavior of both Dirichlet and Neumann traces separately for the second moment, we dispose of a known solution, set $\bxi^\textup{inc}=(\gamma_0\U^\textup{inc},\gamma_1\U^\textup{inc})$, $\mu_0=\mu_1=1$, and consider the following BIEs with $\bxi := \bxi^0 = (\xi_0,\xi_1)$:
\begin{align*}
(\widehat{\opA}_{\kappa_0,\mu_0} + \widehat{\opA}_{\kappa_1,\mu_1}) \bxi &= \bxi^\textup{inc},&\textup{ on } \Gamma,\\
(\widehat{\opA}_{\kappa_0,\mu_0} + \widehat{\opA}_{\kappa_1,\mu_1})^{(2)} \Sigma &= \bxi^\textup{inc}\otimes\bxi^\textup{inc},&\textup{ on } \Gamma^{(2)}.
\end{align*}
Using the Mie series, we know exactly $\bxi$ as well as $\Sigma = \bxi \otimes \bxi$. Due to the domain regularity and the piecewise linear discretization on an affine mesh, we expect convergence rates of \Cref{tab:ConvergenceRates} to be verified. In particular here, the asymptotic error is limited by the discretization error for Dirichlet counterpart of traces, i.e.
\begin{align*}
\|\Sigma - \hat{\Sigma}_L\|_{(H^{1/2}(\Gamma) \times H^{-1/2}(\Gamma))^{(2)}}&=C|\log h|^{1/2} h^{3/2} \|\xi_0^{(2)}\|_{H^{1/2}(\Gamma)^{(2)}} + o(h^{3/2}).
\end{align*}
For the sake of conciseness, we decide to focus on the error for the Dirichlet and Neumann counterparts of the solution and not on cross terms in $H^{1/2}(\Gamma)\otimes H^{-1/2}(\Gamma)$ and $H^{-1/2}(\Gamma)\otimes H^{1/2}(\Gamma)$ as they provide similar results to the problems (P$_\beta$), $\beta=0,1,2$.

We set $\mu_0=\mu_1=1$, and solve the transmission problem for two couples of wavenumbers, referred to as the low frequency (LF): $(\kappa_0,\kappa_1)=(0.4,1)$ and high frequency (HF): $(\kappa_0,\kappa_1)=(8,2)$ cases. We generate a sequence of meshes obtained by subdividing each triangle into two new ones, and by projecting the new vertices onto $\IS^2$. We obtain a cardinality for $V_l$ of $N_l = \mO(2^l)$ and set $L=9$. First, we study the convergence rates for the first moment and full tensor solutions in \Cref{tab:ResSphere}. We remark a $\mO(h^{3/2})$ and $\mO(h^2)$ convergence rate for the Dirichlet and Neumann traces respectively. Also, we notice an oscillatory behavior of the error for the Dirichlet trace for the HF case.

\begin{table}[t]
\renewcommand\arraystretch{1.7}
\begin{center}
\footnotesize
\resizebox{9cm}{!} {
\begin{tabular}{
    >{\centering\arraybackslash}m{2cm}
    |>{\centering\arraybackslash}m{4.5cm}
    |>{\centering\arraybackslash}m{4.5cm}
    }
\vspace{0.1cm}
&  LF: ($\kappa_0,\kappa_1)=(0.4,1)$ & HF: ($\kappa_0,\kappa_1)=(2,8)$ \\ \hline
 $k=1,\xi_L$ ~~~~\includegraphics[width=1.1\linewidth]{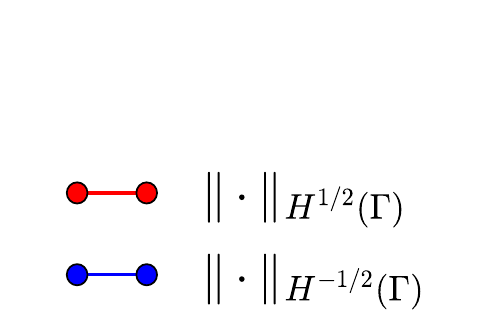} &\includegraphics[width=1\linewidth]{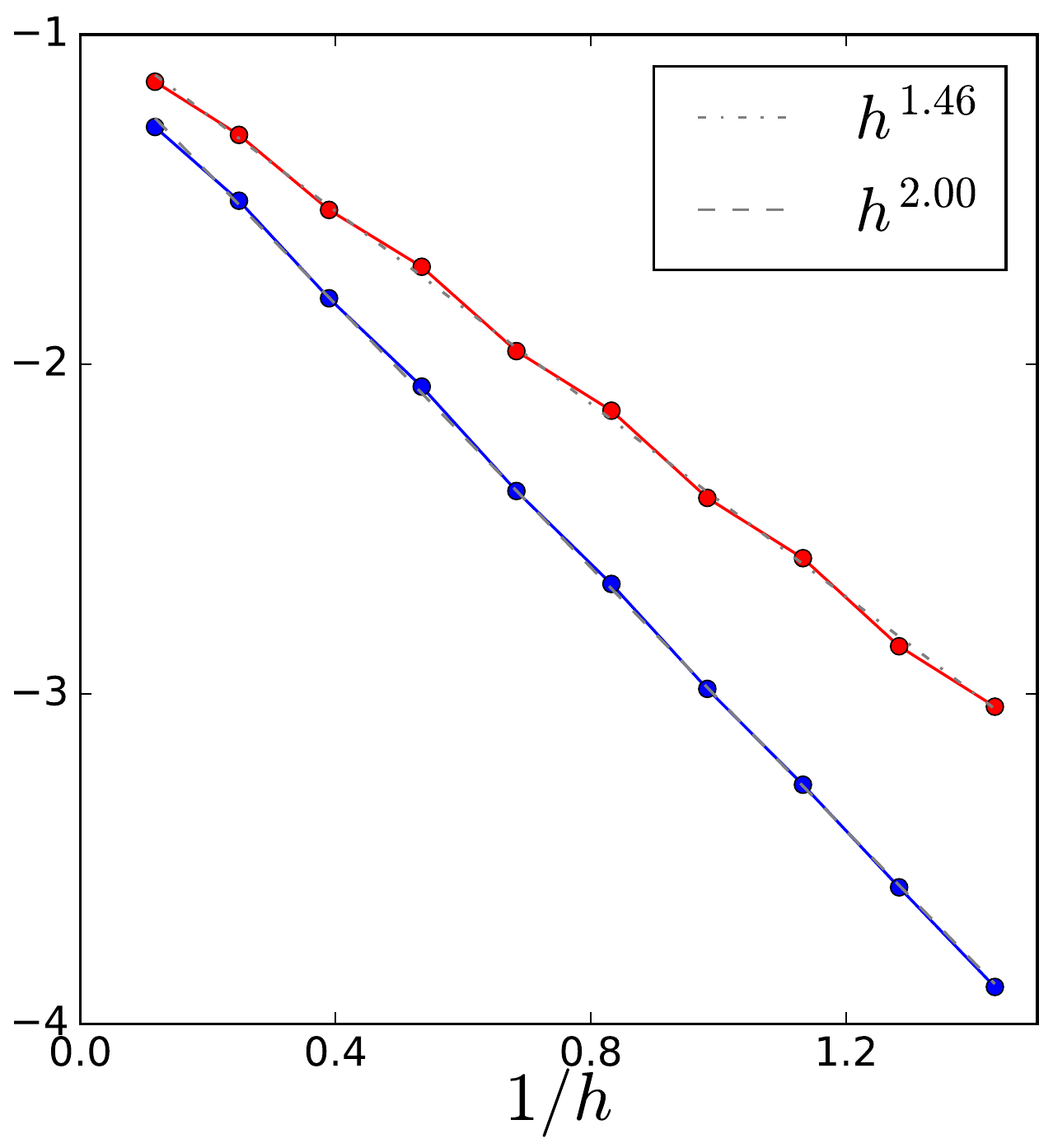} &  \includegraphics[width=1\linewidth]{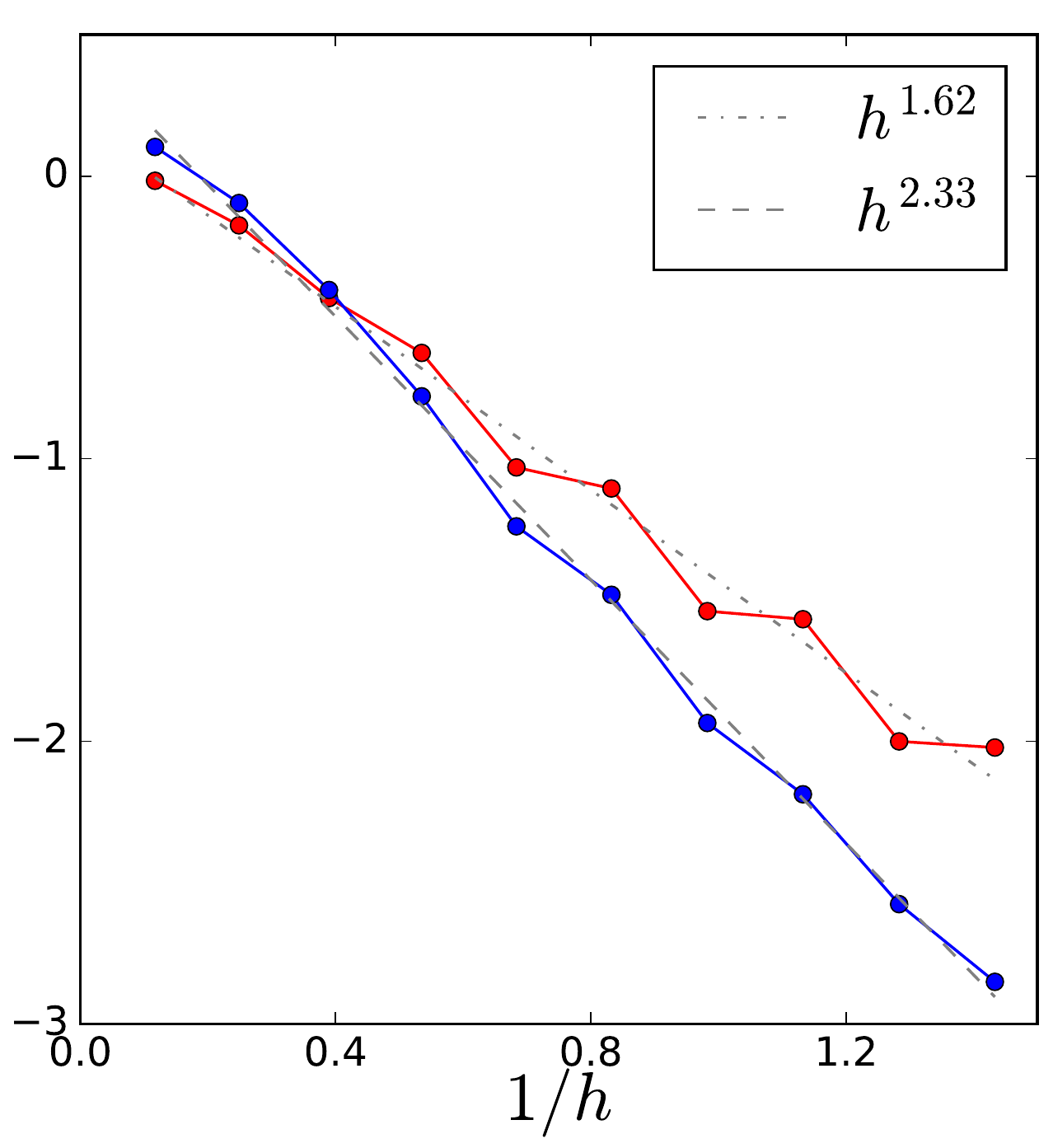}\\ \hline
 $k=2,\Sigma_L$ ~~~~\includegraphics[width=1.1\linewidth]{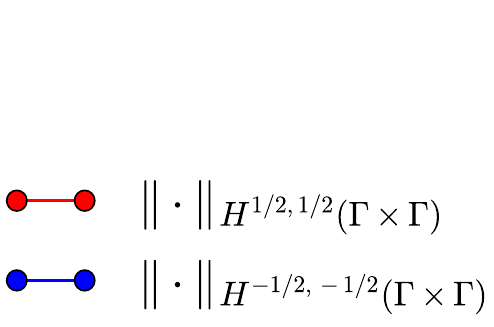} &\includegraphics[width=1\linewidth]{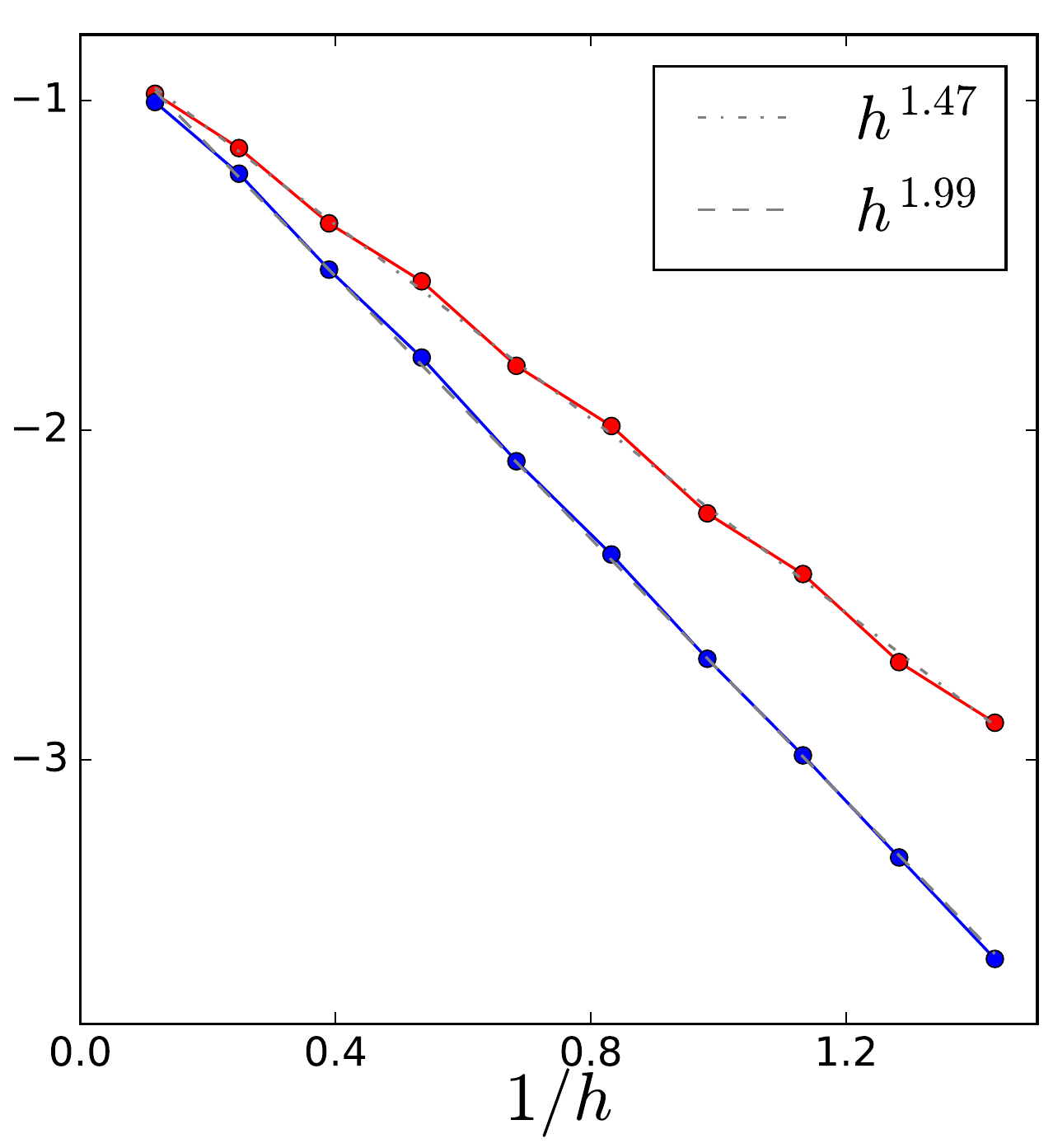} &  \includegraphics[width=1\linewidth]{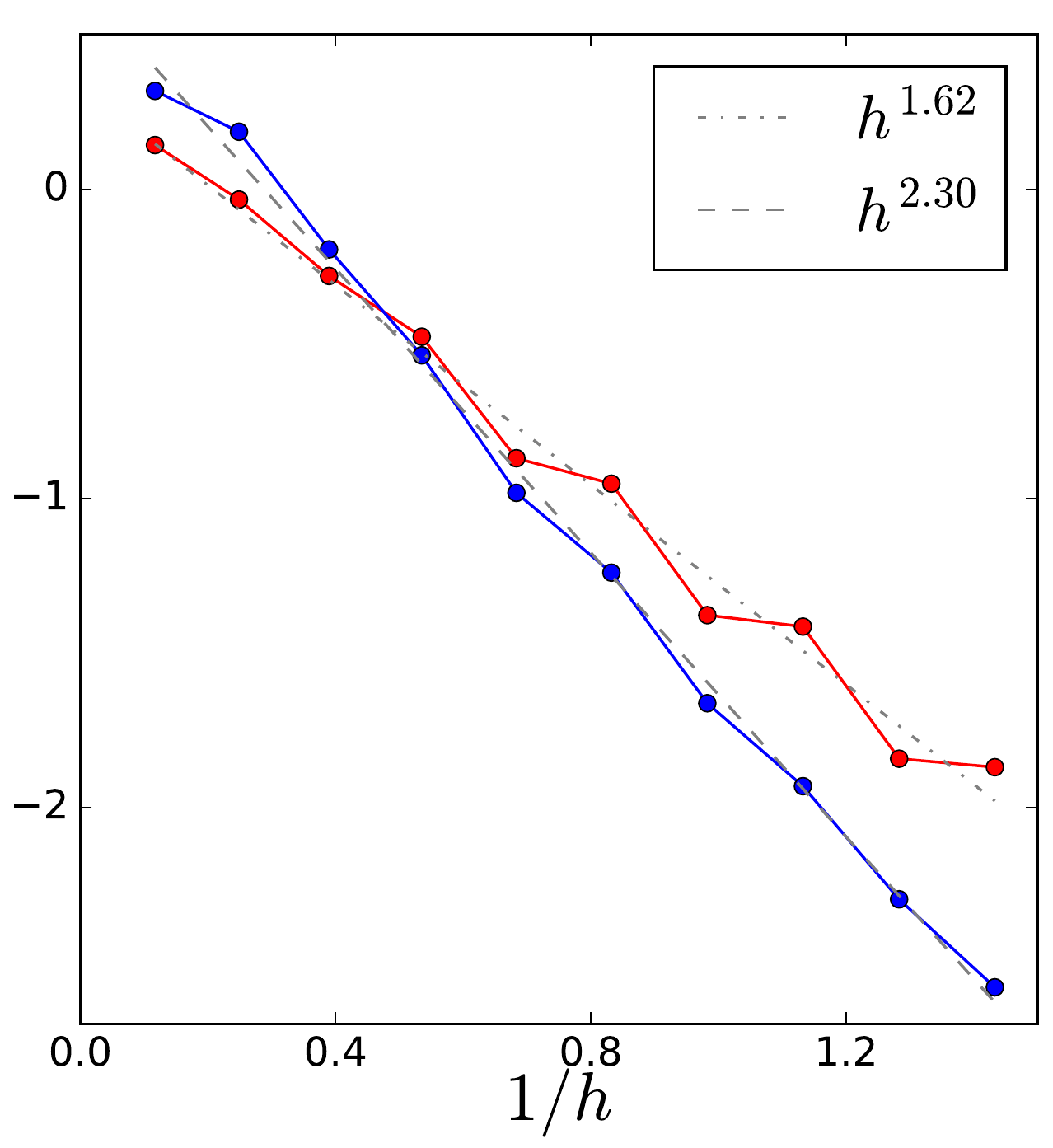}\\ \hline
 \end{tabular}}
\caption{Relative errors in energy norm of the Dirichlet and Neumann Traces on $\IS^2$ and $(\IS^2)^{(2)}$ for the LF and HF cases. Relative energy norm error for Dirichlet (red) and Neumann (blue) trace components with respect to the inverse mesh density $1/h$. }
\label{tab:ResSphere}
\end{center} 
\end{table}  

Next, we focus on the sparse tensor approximation and the minimal resolution level.  For values of $L_0 \in \{0,1,2,3,4\}$, in \Cref{tab:ResCTDof} we represent the relative energy norm error of $\hat{\Sigma}_L (L_0)$ versus the full solution $\Sigma_L$ function to $1/h$. We restrict the case $L_0=4$ to the HF case, as lower refinement levels give satisfactory results. As expected, we remark that for a sufficient minimal resolution level, the solution in the sparse tensor space converges with the same rate as the full solution $\Sigma_L$ (in black).

We also represent the precision $r$, which represents the number of elements per wavelength in the $x$-axis. For the HF case: (i) the sparse solution inherits of limited convergence rate for small values of $1/h$; and, (ii) appears to oscillate less than in the full space.
\begin{table}[t]
\renewcommand\arraystretch{1.7}
\begin{center}
\footnotesize
\resizebox{10cm}{!} {
\begin{tabular}{
    >{\centering\arraybackslash}m{2cm}
    |>{\centering\arraybackslash}m{4.5cm}
    |>{\centering\arraybackslash}m{4.5cm}
    }
\vspace{0.1cm}
&   Dirichlet trace  & Neumann trace \\ \hline
 LF:$\hat{\Sigma}_L(L_0)$  &  \includegraphics[width=1\linewidth]{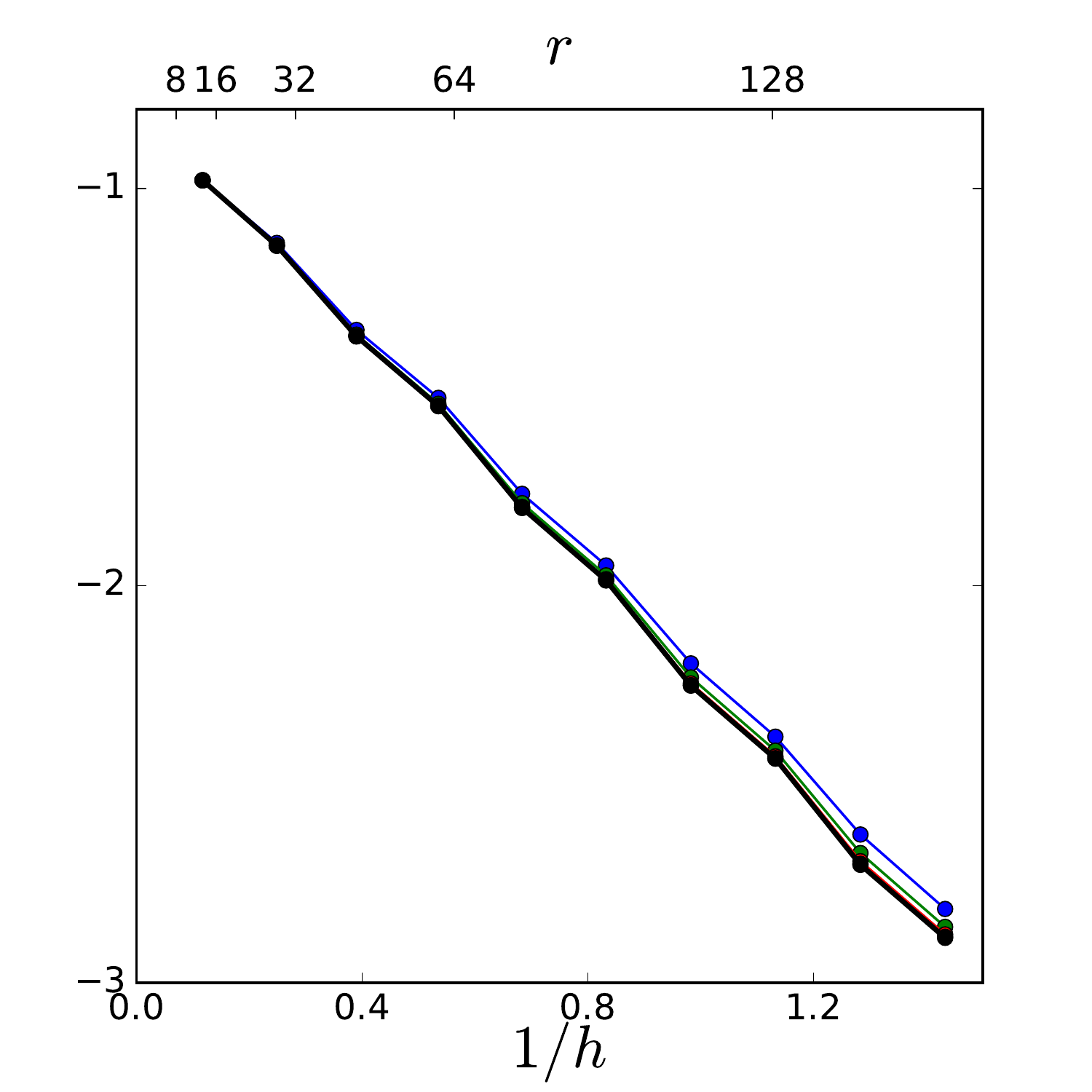} &  \includegraphics[width=1\linewidth]{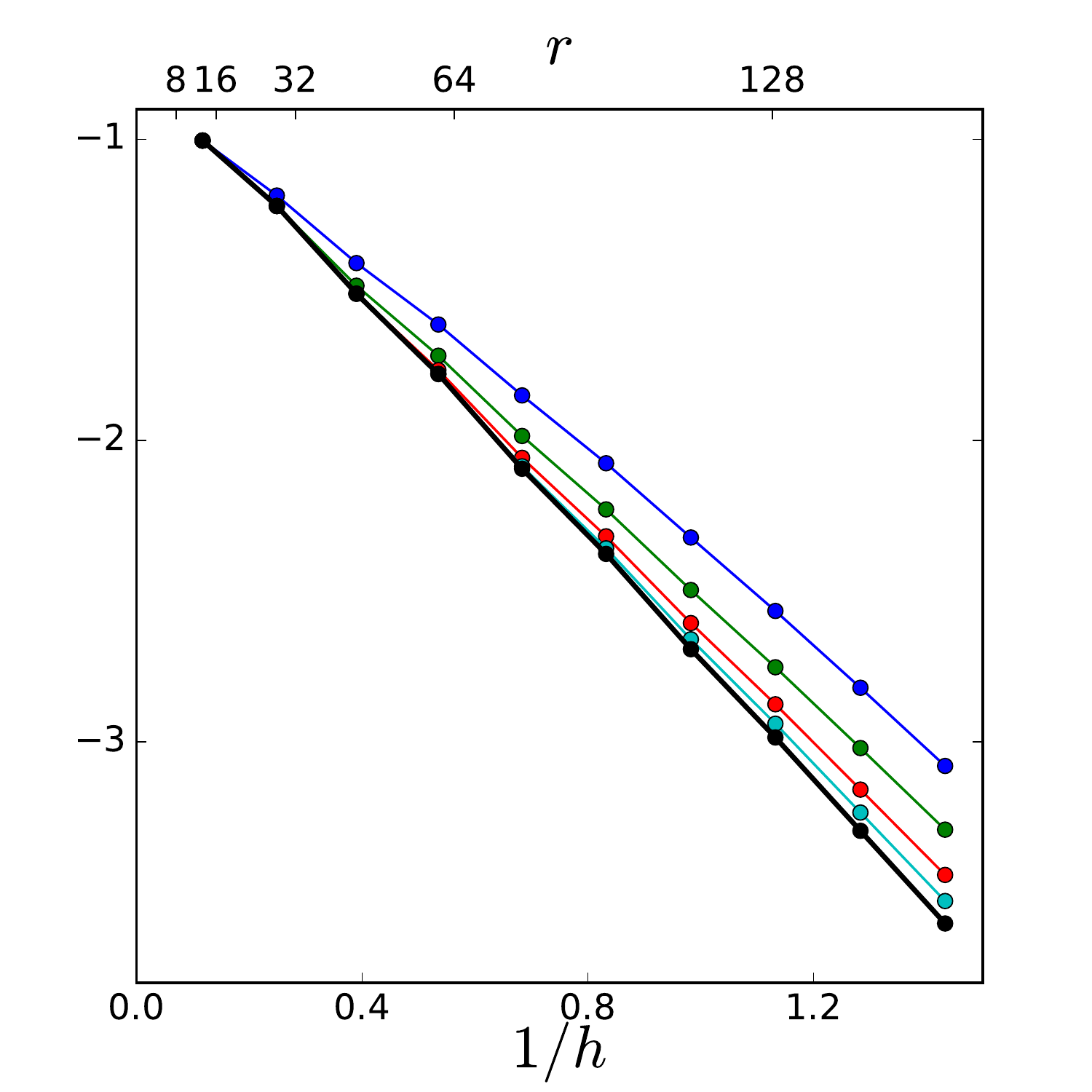}\\ \hline
 HF:$\hat{\Sigma}_L(L_0)$ &\includegraphics[width=1\linewidth]{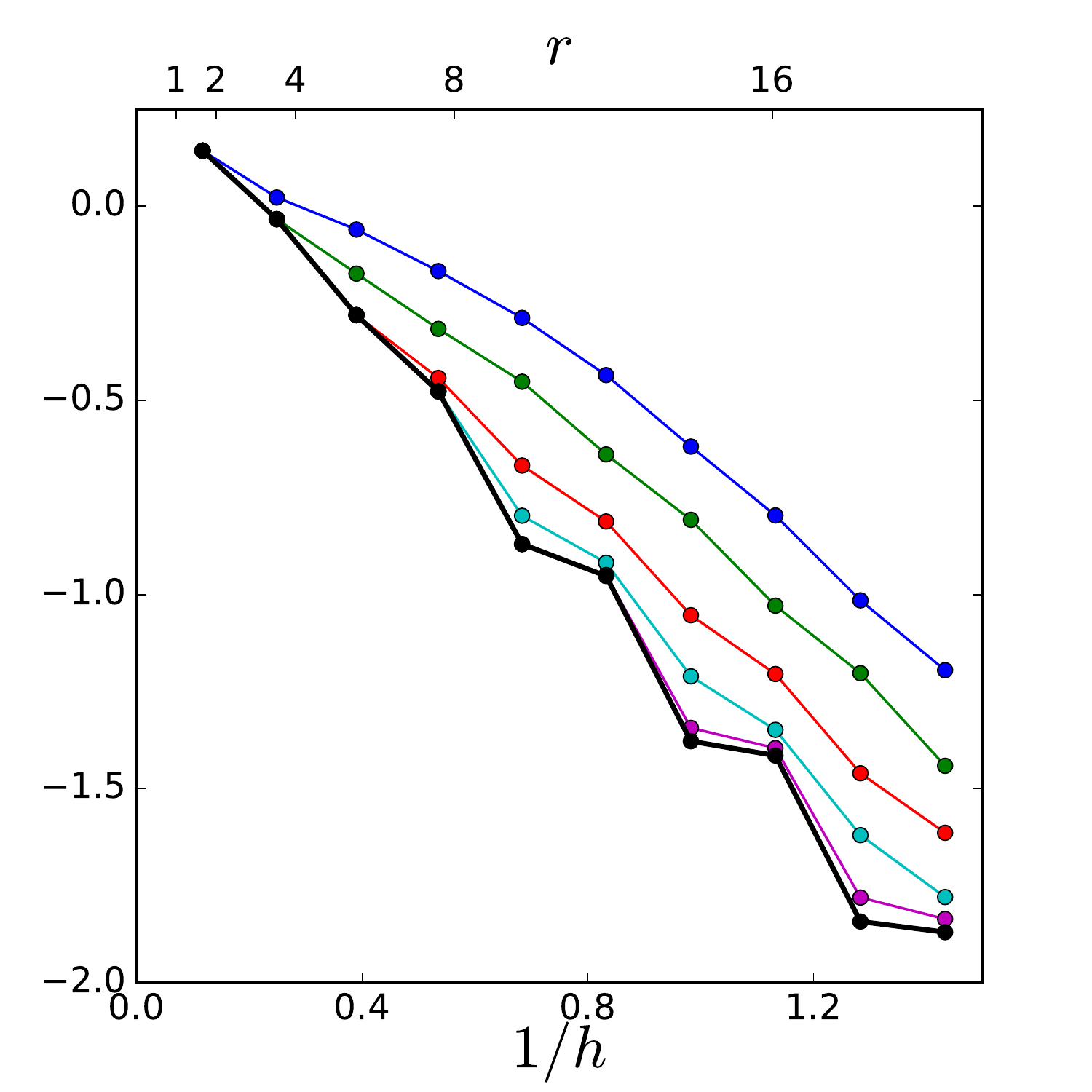} &  \includegraphics[width=1\linewidth]{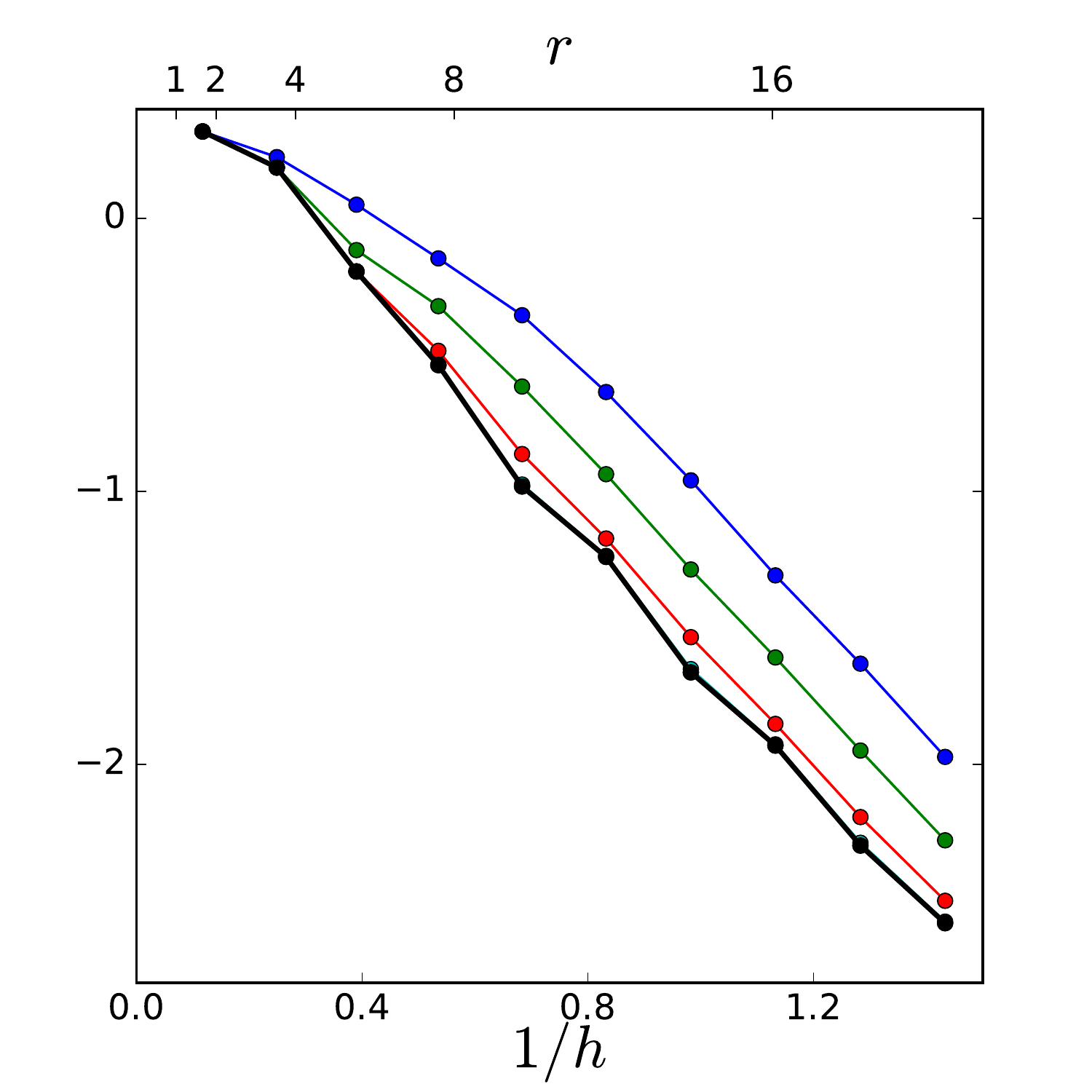}\\ \hline
 \multicolumn{3}{c}{\includegraphics[width=0.9\linewidth]{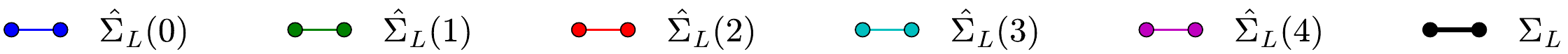}}
 \end{tabular}}
 \vspace{-.2cm}
\caption{Relative errors in energy norm function to $h$ of the Dirichlet and Neumann Traces on ${(\IS^2)}^{(2)}$ for the LF and HF cases.}
\label{tab:ResCT}
\end{center} 
\end{table}  

In order to better assess the quality of the sparse approximate function to $L_0$, we present in \Cref{tab:ResCTDof} the same energy norm errors as in \Cref{tab:ResCT} function to dofs the number of dofs used to get the approximation. The optimal resolution level $\hat{L}_0$ depends on the type of trace and the frequency. 
\begin{table}[t]
\renewcommand\arraystretch{1.7}
\begin{center}
\footnotesize
\resizebox{10cm}{!} {
\begin{tabular}{
    >{\centering\arraybackslash}m{2cm}
    |>{\centering\arraybackslash}m{4.5cm}
    |>{\centering\arraybackslash}m{4.5cm}
    }
\vspace{0.1cm}
&   Dirichlet trace  & Neumann trace \\ \hline
 LF:$\hat{\Sigma}_L(L_0)$  &  \includegraphics[width=1\linewidth]{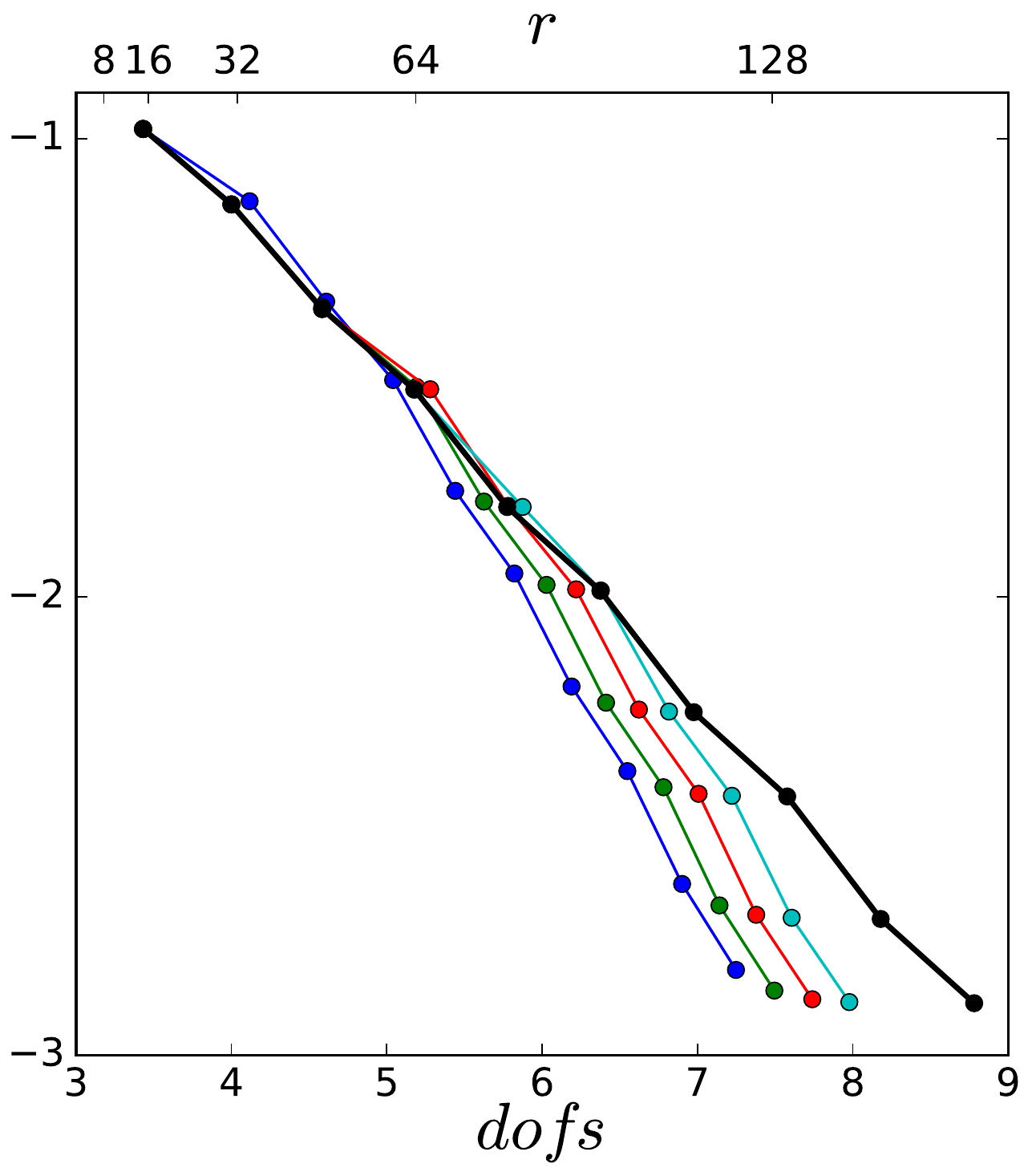} $\hat{L}_0 = \textcolor{blue}{0}$ &  \includegraphics[width=1\linewidth]{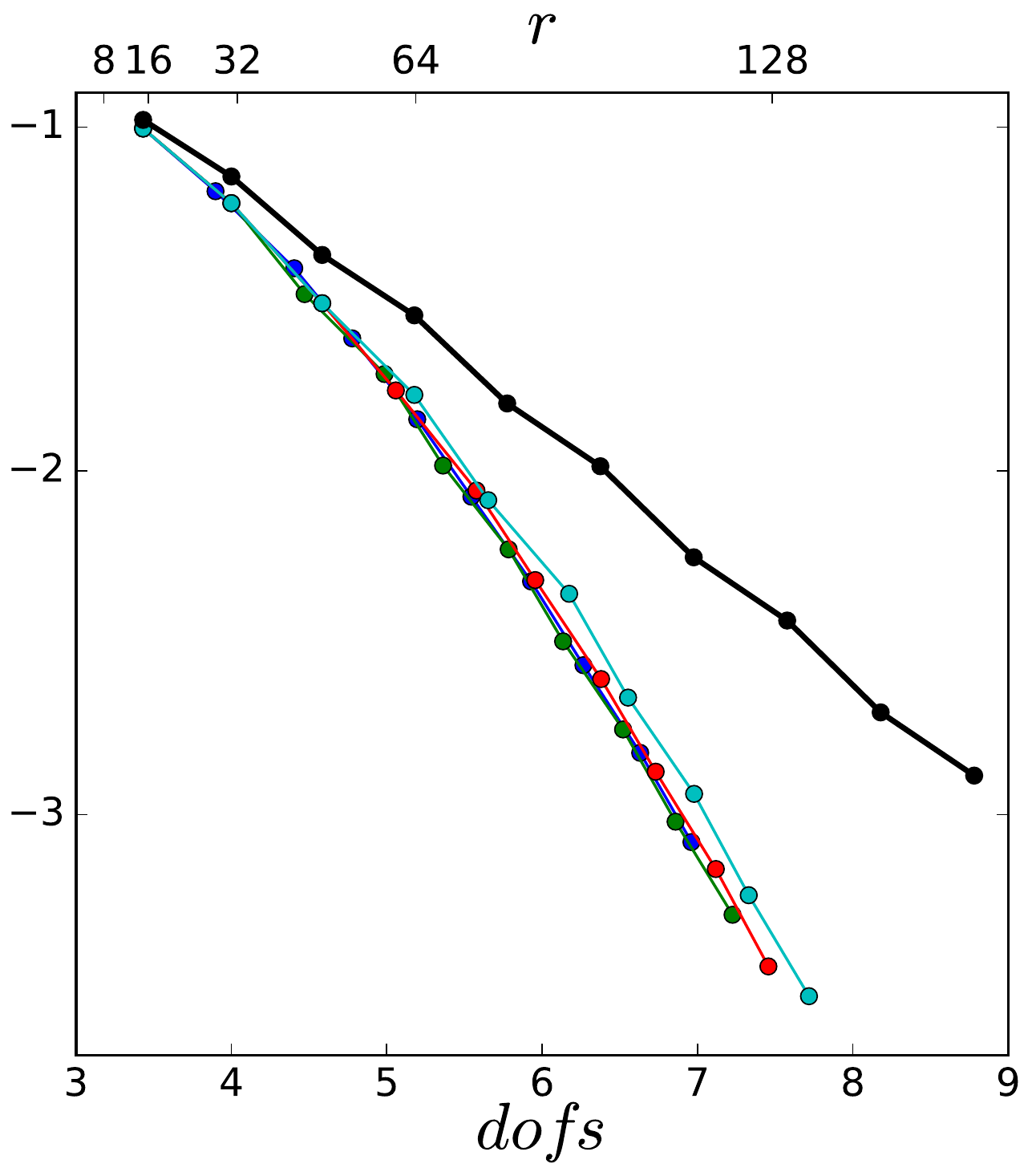} $\hat{L}_0 = \textcolor{blue}{0},\textcolor{Greeen}{1}$\\ \hline
 HF:$\hat{\Sigma}_L(L_0)$ &\includegraphics[width=1\linewidth]{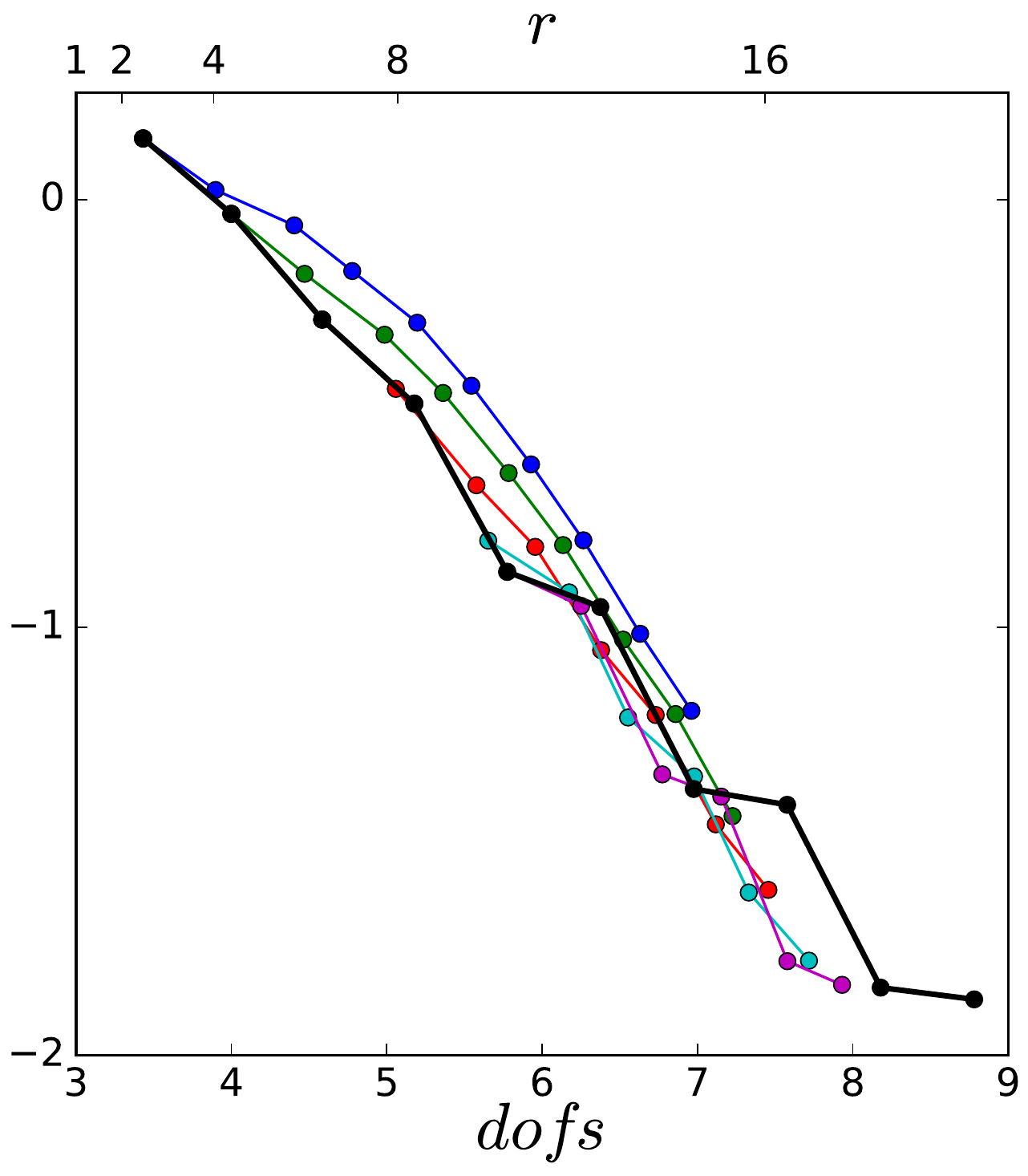}  $\hat{L}_0 = \textcolor{BlueGreeen}{3},\textcolor{Pink}{4}$&  \includegraphics[width=1\linewidth]{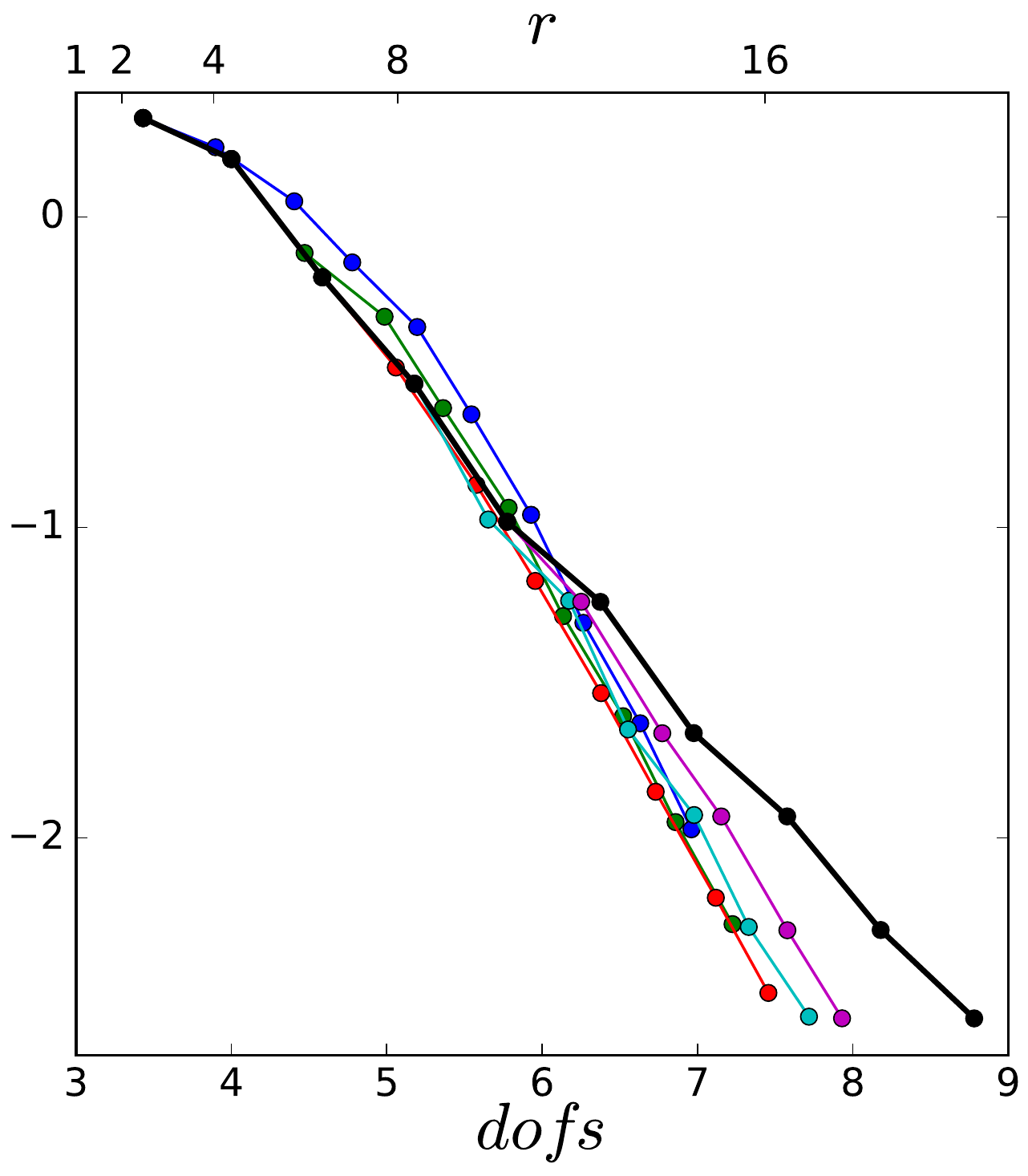} $\hat{L}_0 = \textcolor{Greeen}{1},\textcolor{red}{2}$\\ \hline
 \multicolumn{3}{c}{\includegraphics[width=0.9\linewidth]{figs/legendsparseL0.pdf}}
 \end{tabular}}
 \vspace{-.2cm}
\caption{Relative errors in energy norm function to dofs of the Dirichlet and Neumann Traces on ${(\IS^2)}^{(2)}$ for the LF and HF cases.}
\label{tab:ResCTDof}
\end{center} 
\end{table}  
In the sequel, we focus on the symmetric case (refer to \Cref{subs:cov}). In \Cref{tab:ResCTDofSym} we corroborate that the symmetry of the right-hand side benefits the sparse tensor approximation, as roughly half of the linear systems of the classical CT are needed. 
\subsection{Unit Sphere: Iterative solvers}
\label{subsec:TP_Iterative}
We focus on the practical implementation of the CT. We solve the sub-blocks of the symmetric CT with GMRES and a tolerance of $10^{-8}$. \Cref{fig:res} showcases the number of dofs and GMRES iterations needed to reach the prescribed tolerance of each sub-block for given indices $l_1$ and $l_2$. We highlight the case $L=7$ and $L_0=0$ with bold (resp.~italic) notation for the added (resp.~subtracted) sub-blocks for the symmetric CT (\textit{cf.}~\Cref{tab:CombSym}). Below, as a reference, we show the results for the first moment. 
The number of dofs on the diagonal (i.e.~the bold and italic ones) are of size $N_0 \times N_L$ (resp.~$N_0 \times N_{L-1}$). Thus, the resolution of subsystems of equivalent size when implementing the CT. 

We also remark the effectiveness of Calder\'on preconditioning, as we notice that the number of iterations remains of $8$ independently of $l_1$ and $l_2$. Also, the number of iterations passes from $3$ for first moment to $8$, likely due to $\kappa_2 (\bA \otimes \bB) = \kappa_2(\bA)\kappa_2(\bB)$.
\begin{table}[t]
\renewcommand\arraystretch{1.7}
\begin{center}
\footnotesize
\resizebox{10cm}{!} {
\begin{tabular}{
    >{\centering\arraybackslash}m{2cm}
    |>{\centering\arraybackslash}m{4.5cm}
    |>{\centering\arraybackslash}m{4.5cm}
    }
\vspace{0.1cm}
&   Dirichlet trace  & Neumann trace \\ \hline
 LF: $\hat{\Sigma}_L(L_0)$  &  \includegraphics[width=1\linewidth]{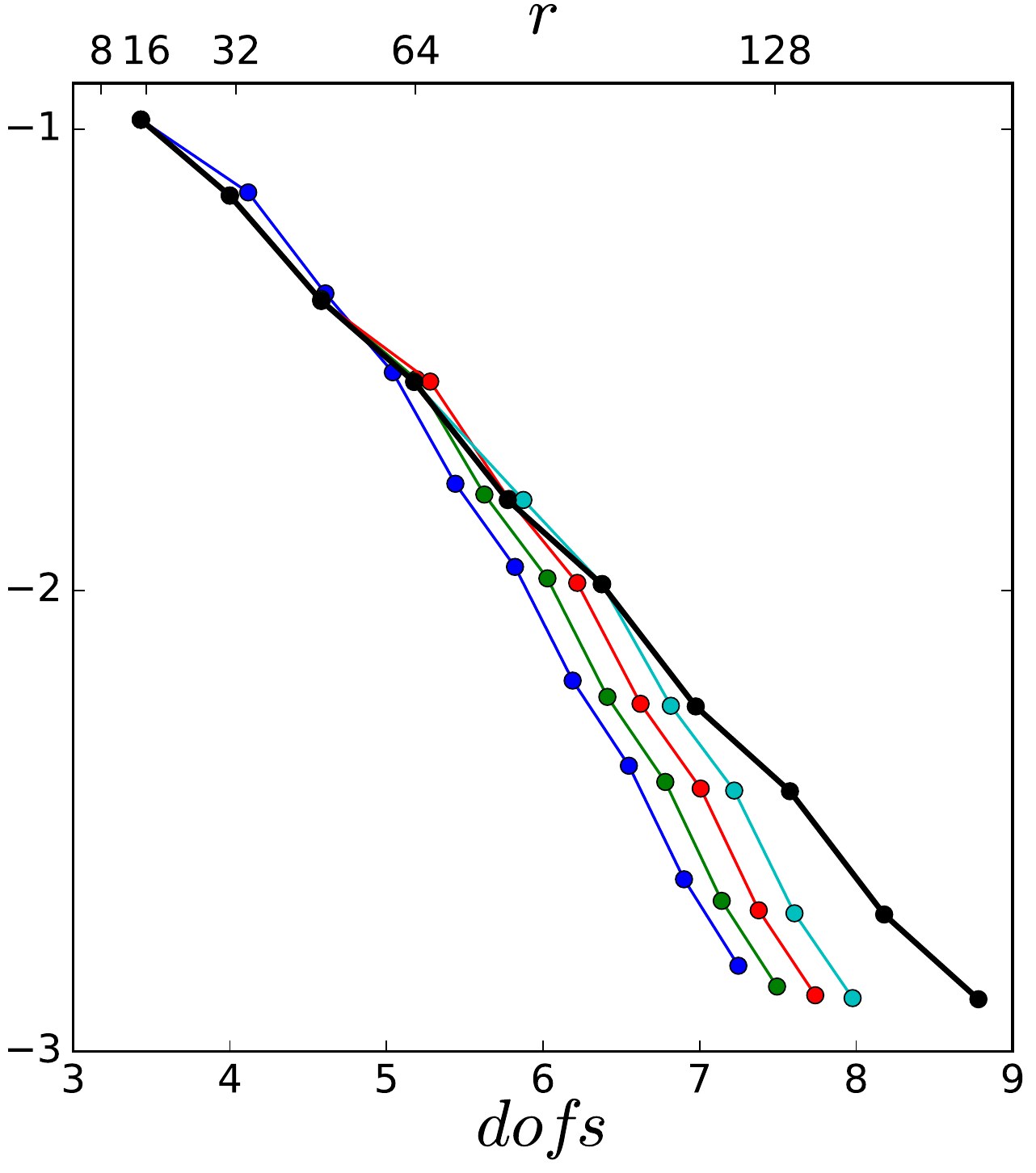} &  \includegraphics[width=1\linewidth]{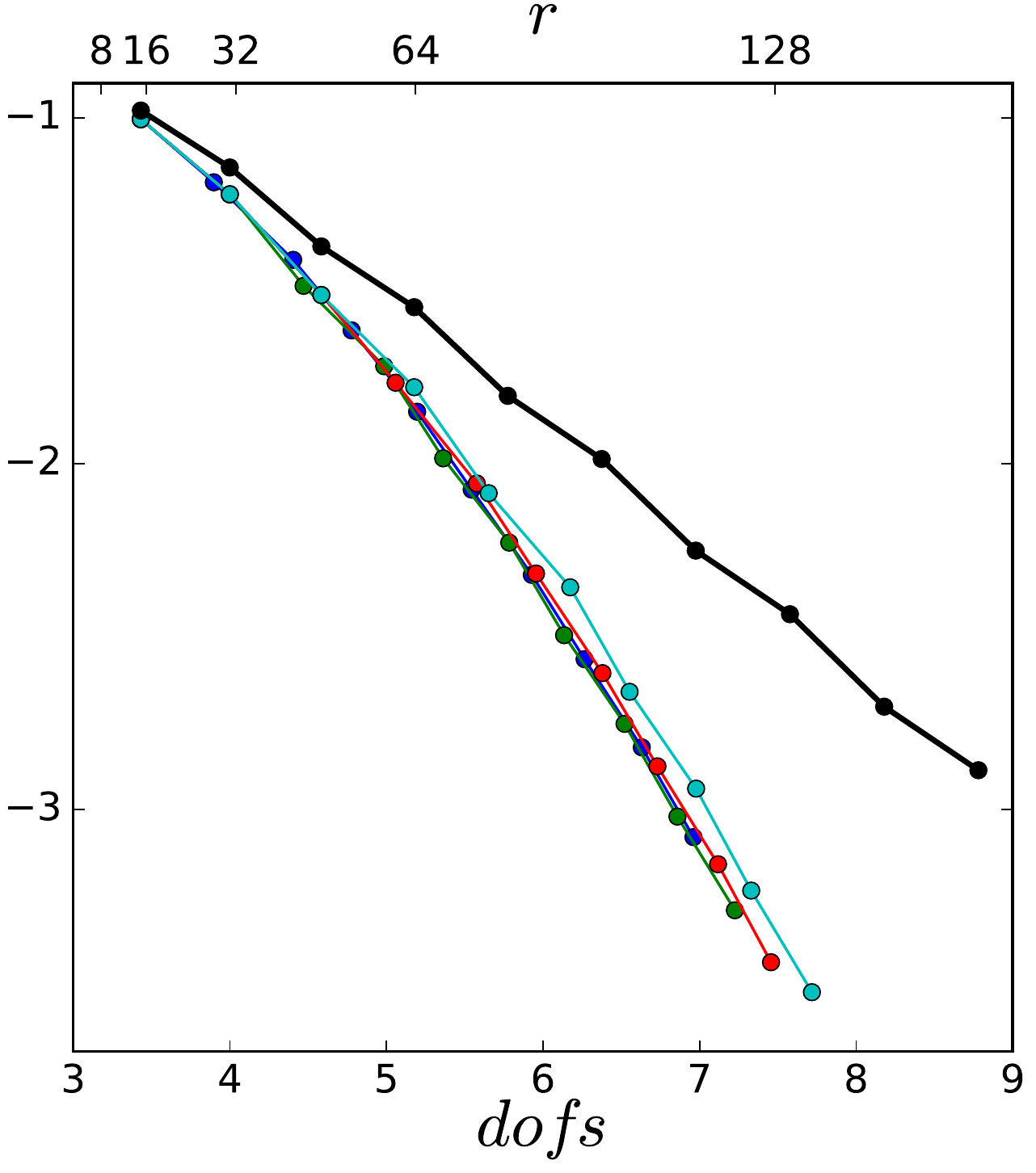}\\ \hline
 HF: $\hat{\Sigma}_L(L_0)$ &\includegraphics[width=1\linewidth]{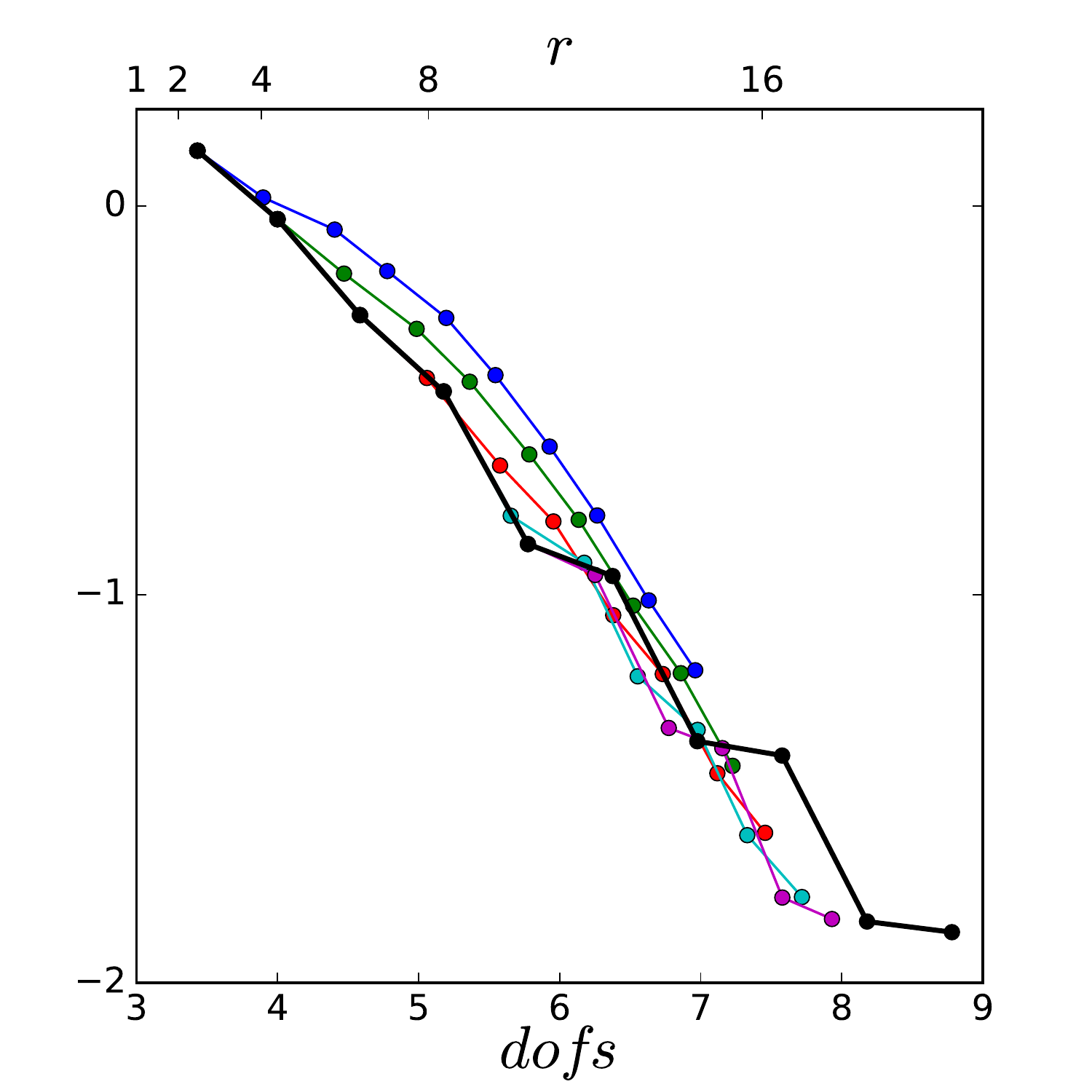} &  \includegraphics[width=1\linewidth]{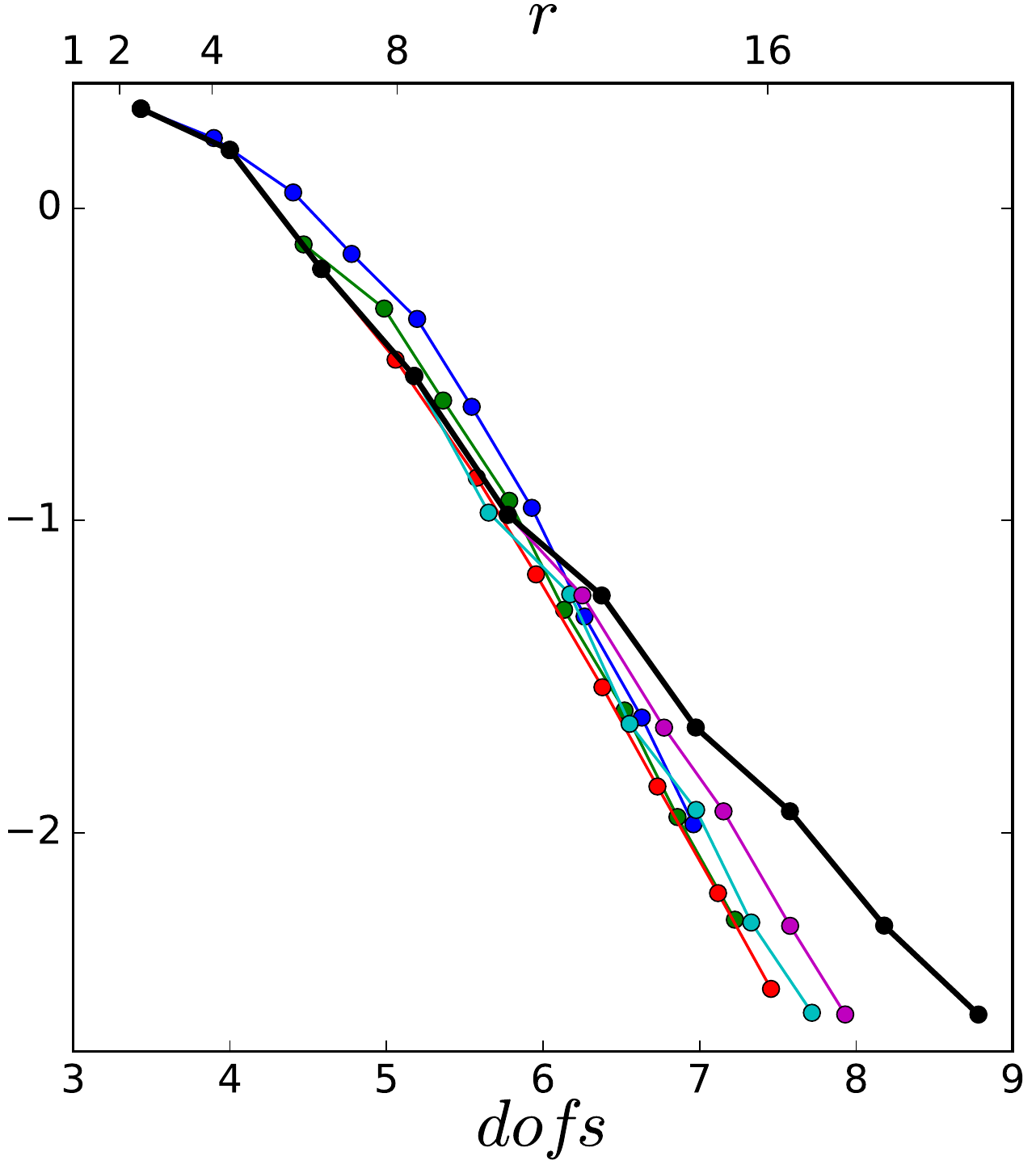}\\ \hline
 \multicolumn{3}{c}{\includegraphics[width=0.9\linewidth]{figs/legendsparseL0.pdf}}
 \end{tabular}}
\vspace{-.2cm}
\caption{Symmetric case: relative errors in energy norm function to dofs of the Dirichlet and Neumann Traces on ${(\IS^2)}^{(2)}$ for the LF and HF cases.}
\label{tab:ResCTDofSym}
\end{center} 
\end{table}  
Besides, we show the solver times in seconds in \Cref{tab:Results_time}. Accordingly, we consider the HF case: we plot the relative residual $l^2$-error of GMRES in \Cref{figs:GmresHF} (in log-log scale). First, in black (resp.~gray), we represent the iterations for $k=1$ and $L=0$ (resp.~$l\in \{1,\cdots,7\}$). We remark that: (i) the iteration count increases compared to the LF case; (ii) the relative error is reasonably resilient with the meshwidth, as the curves are close to each other; and, (iii) fast convergence of the residual towards zero. Still, mesh independence is key in reducing the sensibility to the meshwidth but does not necessarily leads to faster convergence of GMRES, as the condition number remains bounded but can be large, as highlighted for the second moment. Indeed, for several values of $(l_1,l_2)$ we add error convergence curves of GMRES for $k=2$, and renew the previous remark, with a noticeable deterioration of the convergence results. 

Based on the above, we further investigate the properties of the resulting linear systems and the convergence behavior. In \Cref{tab:ResSuper}, we portray again the GMRES residual error, in a semi-log scale (first row). Also, we present the convergence factor at each iteration (second row). The first row shows that all curves present at least a linear decrease, ensuring convergence of GMRES. Moreover, the convergence for the first moment is too fast to observe a super-linear phase. The second moment curves present poor convergence rates close to $1$, with a very slow decrease (see the bottom-right figure), giving a moderate super-linear behavior, still noticeable for $\Sigma_{1,2}$ and $\Sigma_{0,3}$ (see the top-right figure). To finish, we introduce $\bM$ the mass and $\bA$ the impedance matrices. In \Cref{tab:Eigvals}, we plot the eigenvalues distributions of the resulting linear systems (in strong form, such as done in \cite{kleanthous2018calderon}). We remark that the spectra present some clustering at one and have similar patterns. Also, we see that the tensor matrix for $k=2$ has a more scattered cluster, and much more outliers. The latter emerges from the property of the tensor operator, and was expectable. To finish, despite the presence of non-compact terms at continuous level, we observe discrete clustering properties due to \Cref{theorem:cluster}.
\begin{figure}[t]
\begin{minipage}{\linewidth}
\begin{minipage}{0.6\linewidth}
\begin{table}[H]{}
\renewcommand\arraystretch{1.4}
\begin{center}
\footnotesize
\resizebox{7.9cm}{!} {
\begin{tabular}{
|>{\centering\arraybackslash}m{0.8cm}
    |>{\centering\arraybackslash}m{1cm}
    |>{\centering\arraybackslash}m{1cm}
    |>{\centering\arraybackslash}m{1cm}
    |>{\centering\arraybackslash}m{1cm}
    |>{\centering\arraybackslash}m{0.6cm}
    |>{\centering\arraybackslash}m{0.6cm}
    |>{\centering\arraybackslash}m{0.6cm}
    |>{\centering\arraybackslash}m{0.7cm}|
} \hline  
$l_2\backslash l_1$  & 0 & 1& 2&  3&   4&  5&6 & 7\\ \hline  
7                    & \textbf{319,696} &  & &  &   &  & & \\  \hline
   6                 &\textit{159,952} & \textbf{307,600} & &&& && \\  \hline
   5                 & 80,080 &\textit{154,000} & \textbf{301,840}& && &  &\\  \hline
   4               & 40,144 & 77,200 &\textit{151312} &\textbf{299,536}  &    &   &  & \\  \hline
   3               & 20,176 &  38,800  & 76,048 & \textit{150,544}&      &   &  & \\  \hline
   2               & 10,192 &  19,600  & 38,416 & &     &   &  & \\  \hline
   1               & 5,200 &  10,000 &  & &     &   &  & \\  \hline
   0               & 2,704 &   &  & &     &   &  & \\  \hline\hline
   $k=1$& 52 & 100 & 196 &  388 &  772  &1,540 & 3,076 & \textbf{6,148}\\  \hline 
\end{tabular}}
\end{center} 
\label{tab:Dofs}
\end{table}  
\end{minipage}
\begin{minipage}{0.28\linewidth}
\begin{table}[H]
\renewcommand\arraystretch{1.4}
\begin{center}
\footnotesize
\resizebox{3.95cm}{!} {
\begin{tabular}{
|>{\centering\arraybackslash}m{0.8cm}
    |>{\centering\arraybackslash}m{0.1cm}
    |>{\centering\arraybackslash}m{0.1cm}
    |>{\centering\arraybackslash}m{0.1cm}
    |>{\centering\arraybackslash}m{0.1cm}
    |>{\centering\arraybackslash}m{0.1cm}
    |>{\centering\arraybackslash}m{0.1cm}
    |>{\centering\arraybackslash}m{0.1cm}
    |>{\centering\arraybackslash}m{0.1cm}|
} \hline   
$l_2\backslash l_1$ & 0 & 1& 2&  3&   4&  5&6 & 7\\ \hline  
7                 & \textbf{8}  &  &  &   &    &   & & \\  \hline
   6                 & \textit{8}  & \textbf{8} &  &   &    &   & & \\  \hline
   5                 & 8 & \textit{8} & \textbf{8} &   &    &   &  &\\  \hline
   4               &  8 & 8 & \textit{8} & \textbf{8}  &    &   &  & \\  \hline
   3               & 8 & 8  & 8 & \textit{8} &      &   &  & \\  \hline
   2               & 8 & 8  & 8 & &     &   &  & \\  \hline
   1               & 8 &  8 &  & &     &   &  & \\  \hline
   0               & 8 &   &  & &     &   &  & \\  \hline\hline
   $k=1$& 3 & 3 & 3 &  3 &  3  & 3 & 3& \textbf{3} \\  \hline 
\end{tabular}}
\end{center} 
\end{table}  
\end{minipage}
\end{minipage}
\caption{Numbers of dofs for each subsystem (left) and GMRES iterations to reach prescribed tolerance (right).}
\label{fig:res}
\end{figure}
\begin{figure}[t]
\vspace{-0.4cm}
\includegraphics[width=.9\linewidth]{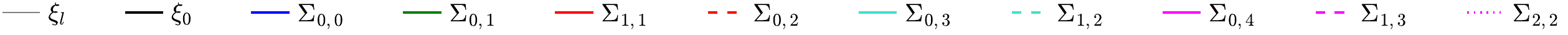}
\begin{minipage}{0.5\linewidth}
\vspace{-0.4cm}
\begin{table}[H]
\renewcommand\arraystretch{1.4}

\footnotesize
\resizebox{6.5cm}{!} {
\begin{tabular}{|c|c|c|c|c|c|c|c|c|} \hline   
$ l_2\backslash l_1$ & 0    &  1   &     2  &  3  & 4&  5&6 & 7\\ \hline
   7                 & \textbf{1,588} &      &        &     & &   &  & \\  \hline
   6                 &  \textit{613}    & \textbf{1,320} &         &     & &   &  & \\  \hline
   5                 &  459    &  \textit{688}  &  \textbf{1,167}  &     & &   & &  \\  \hline
   4               &  209    &  436  & \textit{601}&   \textbf{1,107}& &   & & \\  \hline
   3               &  98.8    &  157  & 392     &    \textit{ 583}& &   & & \\  \hline
   2                 &  49.4    &   77.3  &174     &     & &   & & \\  \hline
   1               &  25.9    &   44.3  &        &     & &   & & \\  \hline
   0                  & 8.83    &     &        &     & &   &  & \\  \hline\hline
   $k=1$ & 0.246  & 0.328  & 0.767 & 1.25  & 1.83 & 5.51  &10.7 & \textbf{17.7} \\   \hline
\end{tabular}}
\vspace{0.1cm}
\caption{}
\label{tab:Results_time}
\end{table}
\end{minipage}
\begin{minipage}{0.5\linewidth}
\vspace{-0.4cm}
\begin{figure}[H]
\includegraphics[width=1\linewidth]{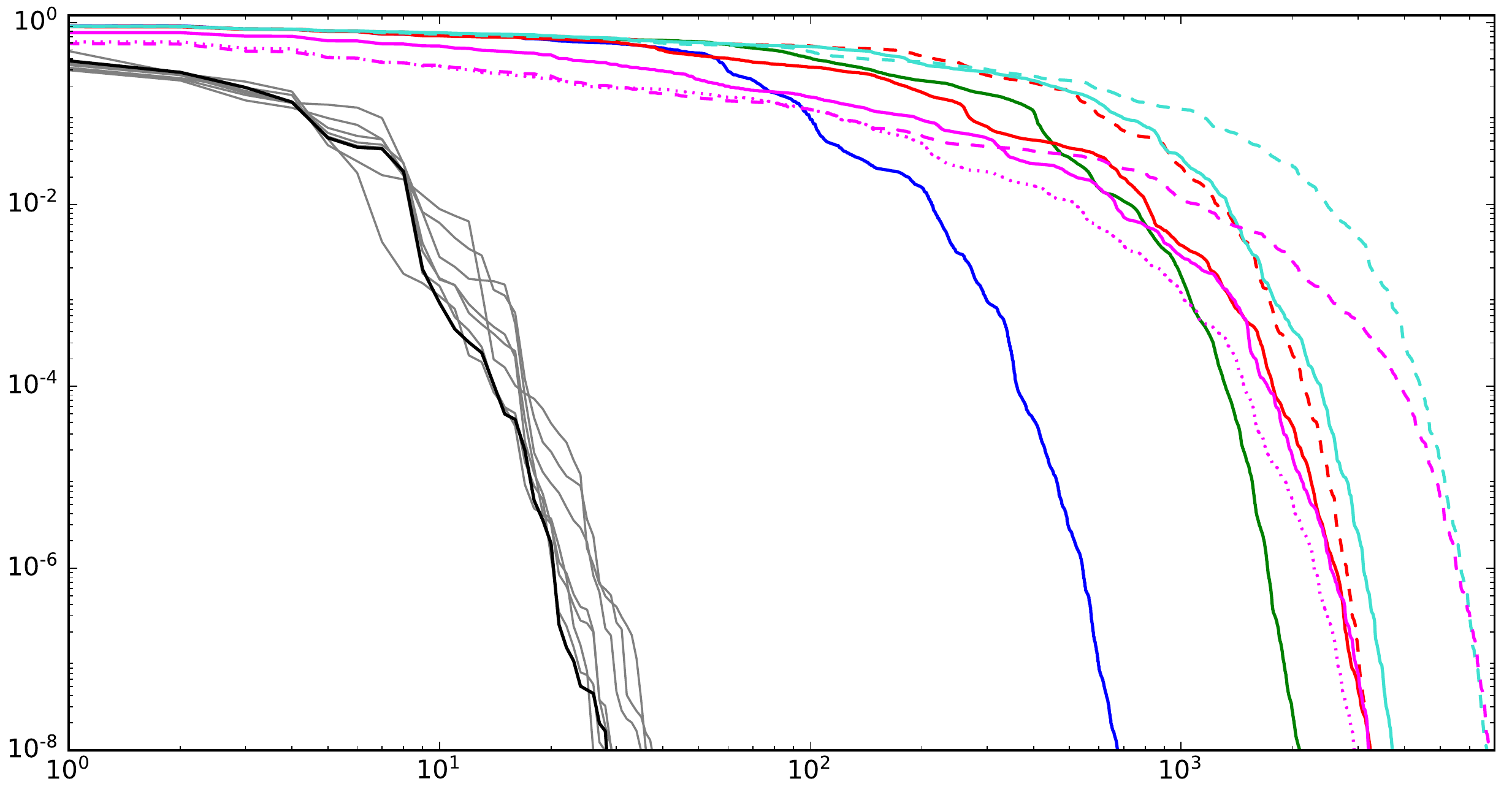}
\caption{}
\label{figs:GmresHF}
\end{figure}
\end{minipage}
\vspace{-.6cm}
\caption*{\Cref{tab:Results_time} (left) : solver times (in seconds) for the LF case. \Cref{figs:GmresHF} (right): relative $l^2$-error of GMRES in log-log scale for the HF case.}
\end{figure}
\begin{table}[t]
\renewcommand\arraystretch{1.7}
\begin{center}
\footnotesize
\resizebox{9cm}{!} {
\begin{tabular}{
    >{\centering\arraybackslash}m{1cm}
    |>{\centering\arraybackslash}m{5.3cm}
    |>{\centering\arraybackslash}m{5.3cm}
    }
\vspace{0.1cm}
&$k=1$ & $k=2$ \\ \hline\hline 
\multicolumn{3}{c}{Residual error of GMRES $\|\br_m\|_2$ function to iteration count $m$} \\\hline
$\|\br_m\|_2$ &\includegraphics[width=\linewidth]{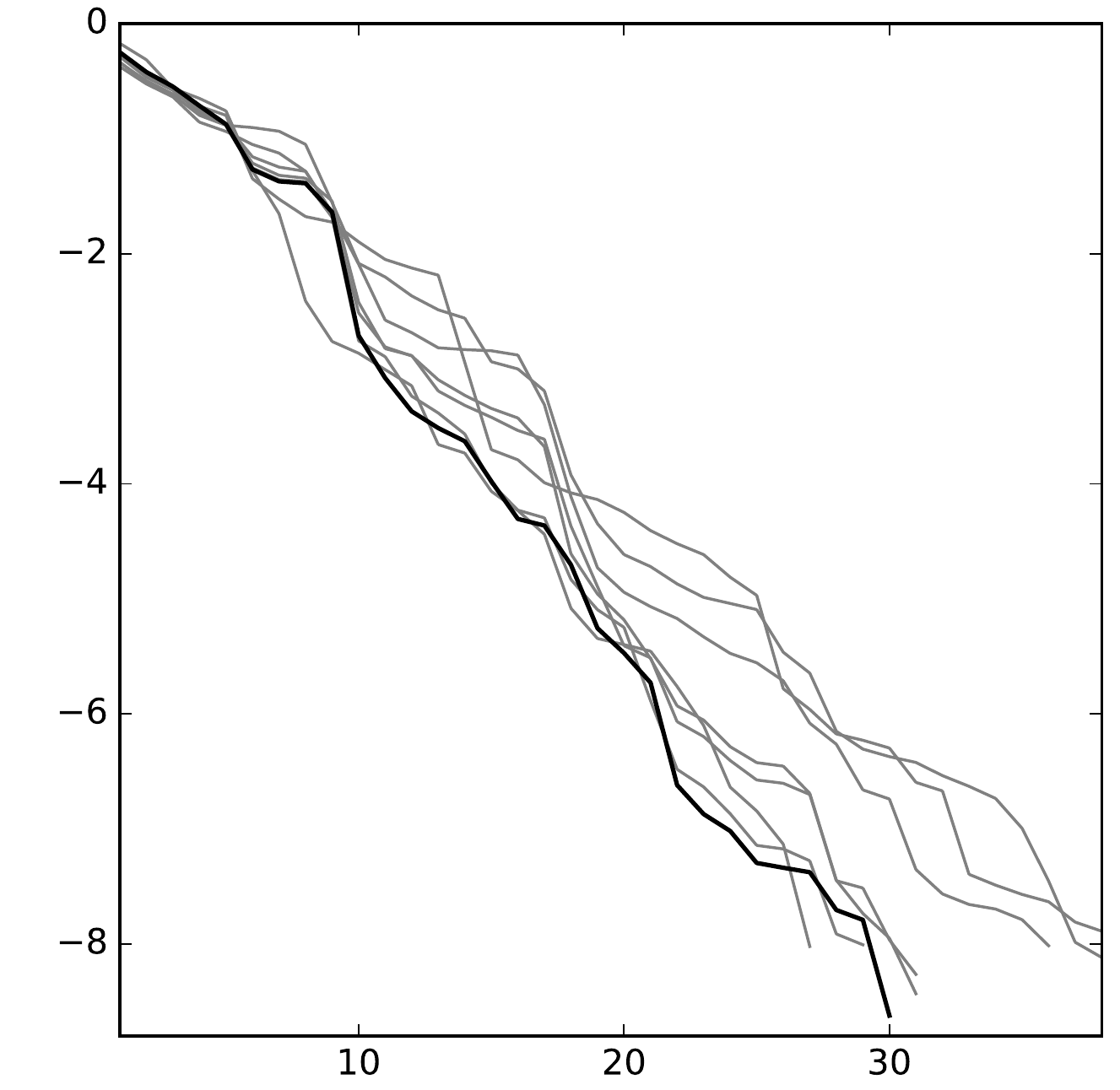}  &  \includegraphics[width=\linewidth]{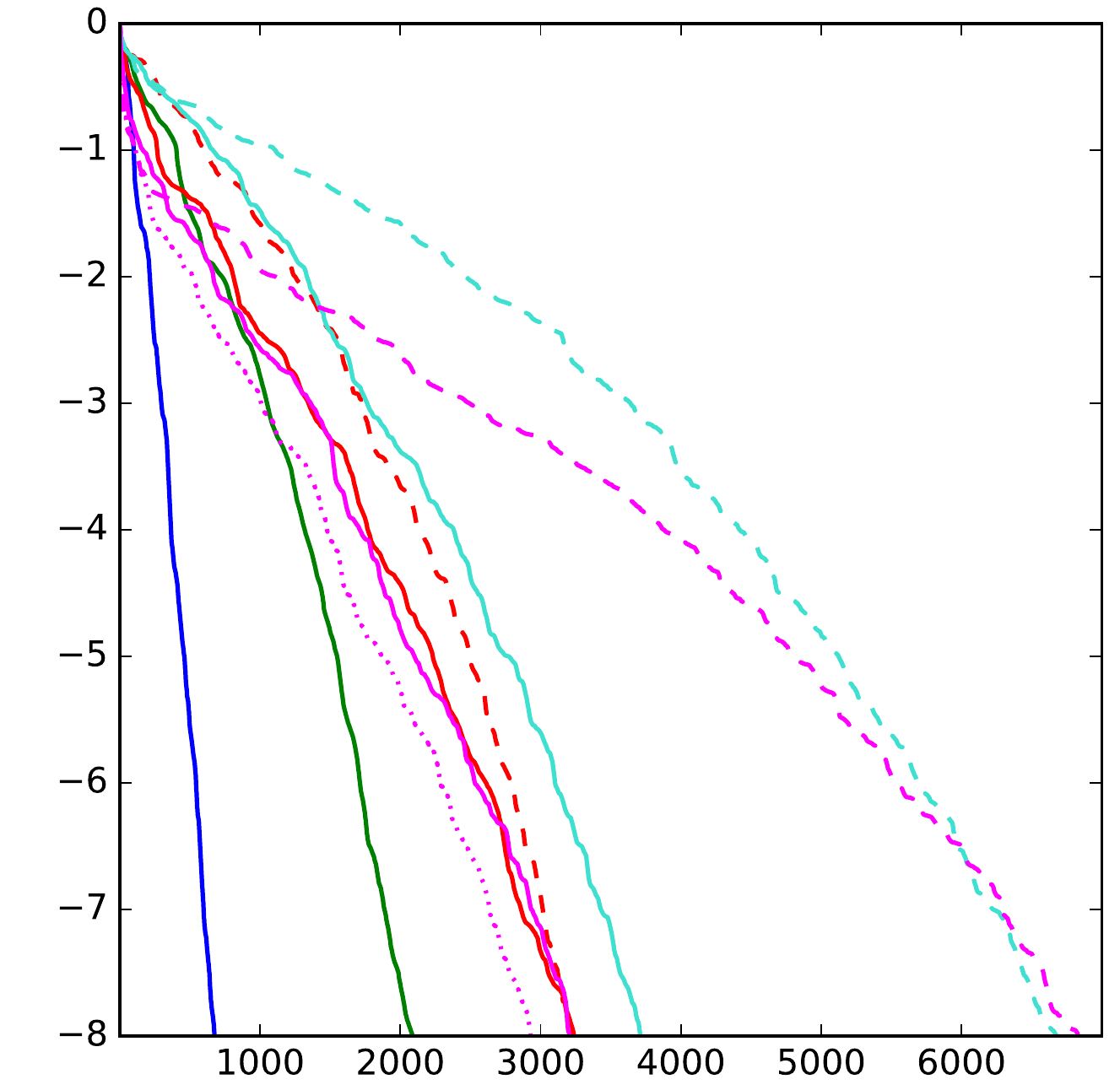}\\\hline
\multicolumn{3}{c}{Convergence estimate $Q_m$ function to iteration count $m$} \\\hline
 $Q_m$&\includegraphics[width=\linewidth]{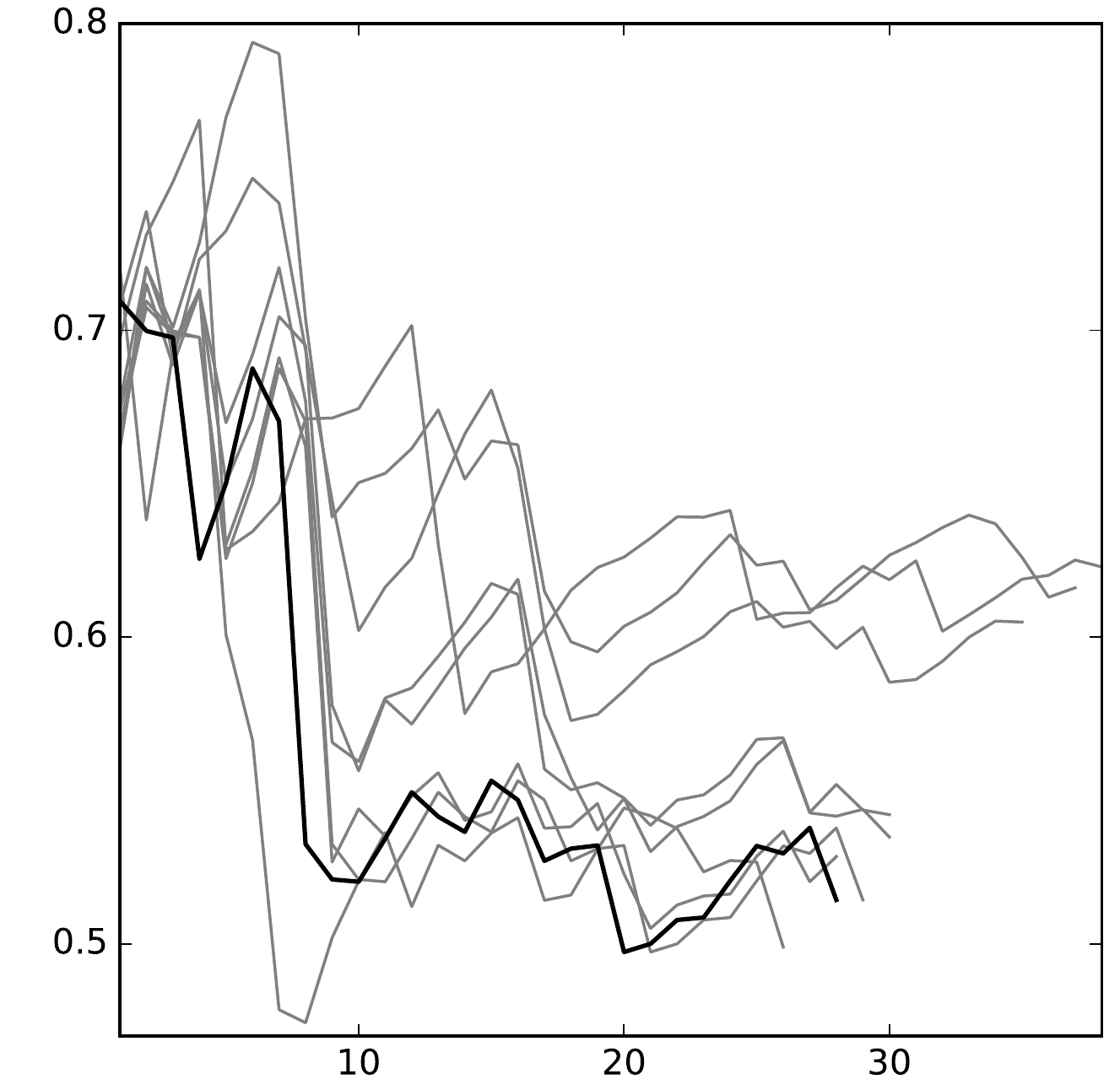}  &  \includegraphics[width=\linewidth]{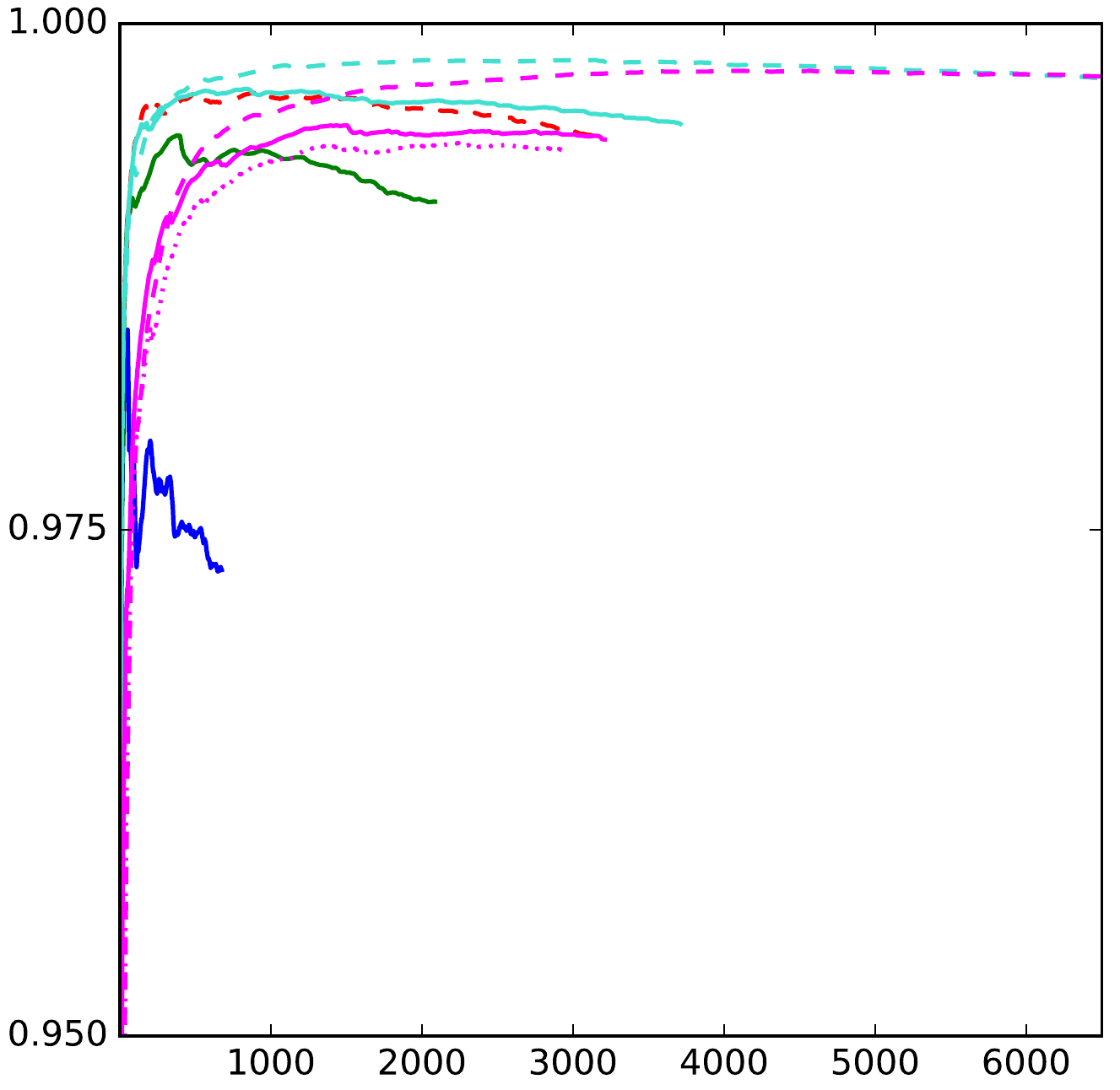}\\ \hline
 \multicolumn{3}{c}{\includegraphics[width=0.9\linewidth]{figs/legendk12.pdf}} 
 \end{tabular}}
 \vspace{-0.2cm}
\caption{HF case: Complete survey of the GMRES convergence}
\label{tab:ResSuper}
\end{center} 
\end{table}  
\begin{table}[t]
\renewcommand\arraystretch{1.7}
\begin{center}
\footnotesize
\begin{tabular}{
    >{\centering\arraybackslash}m{0.8cm}
    |>{\centering\arraybackslash}m{3.5cm}
    |>{\centering\arraybackslash}m{3.5cm}
    |>{\centering\arraybackslash}m{3.5cm}
    }
\vspace{0.1cm}
Case&$L=2$ & $L=3$& $L=4$ \\ \hline\hline
$k=1$  & \includegraphics[width=1.1\linewidth]{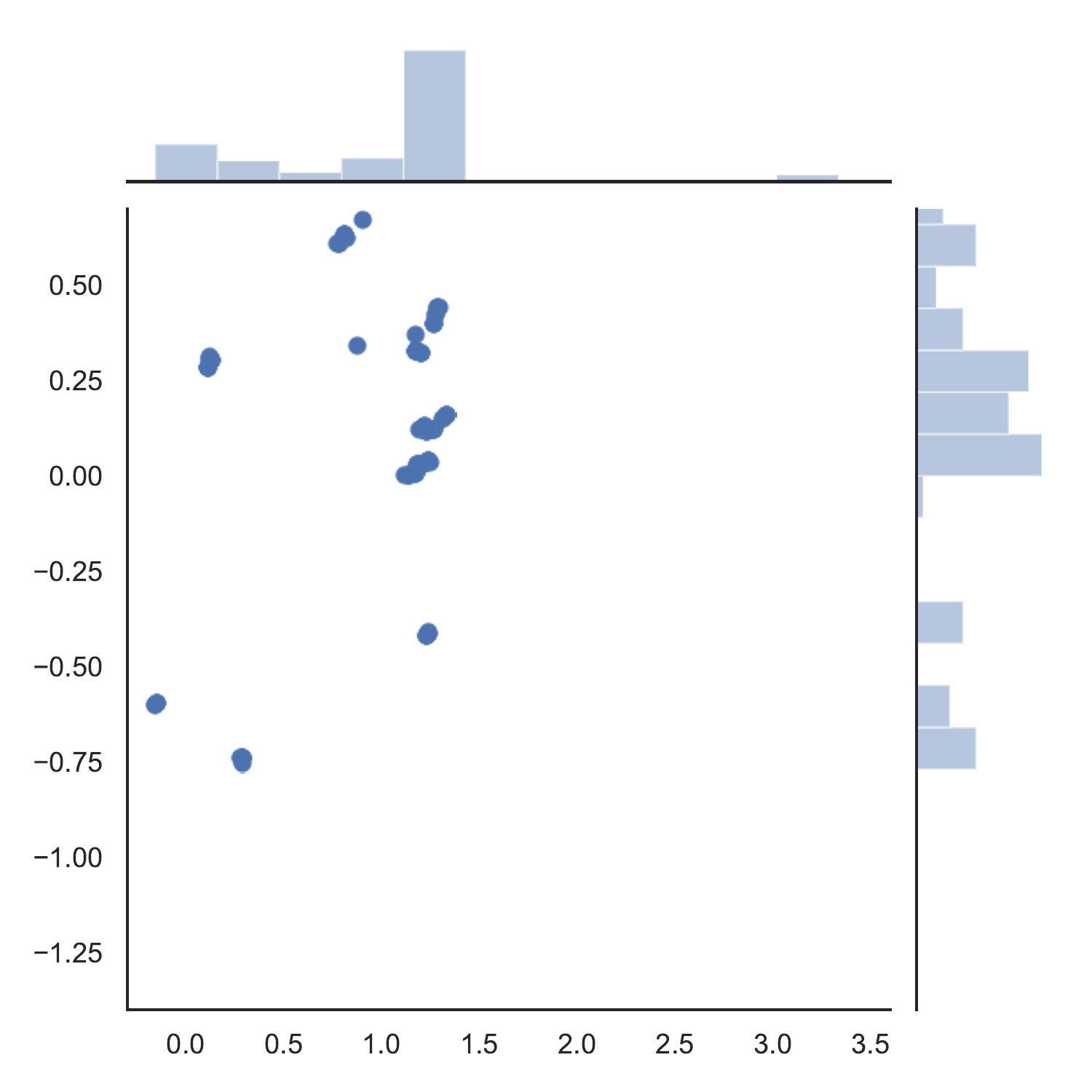} &\includegraphics[width=1.1\linewidth]{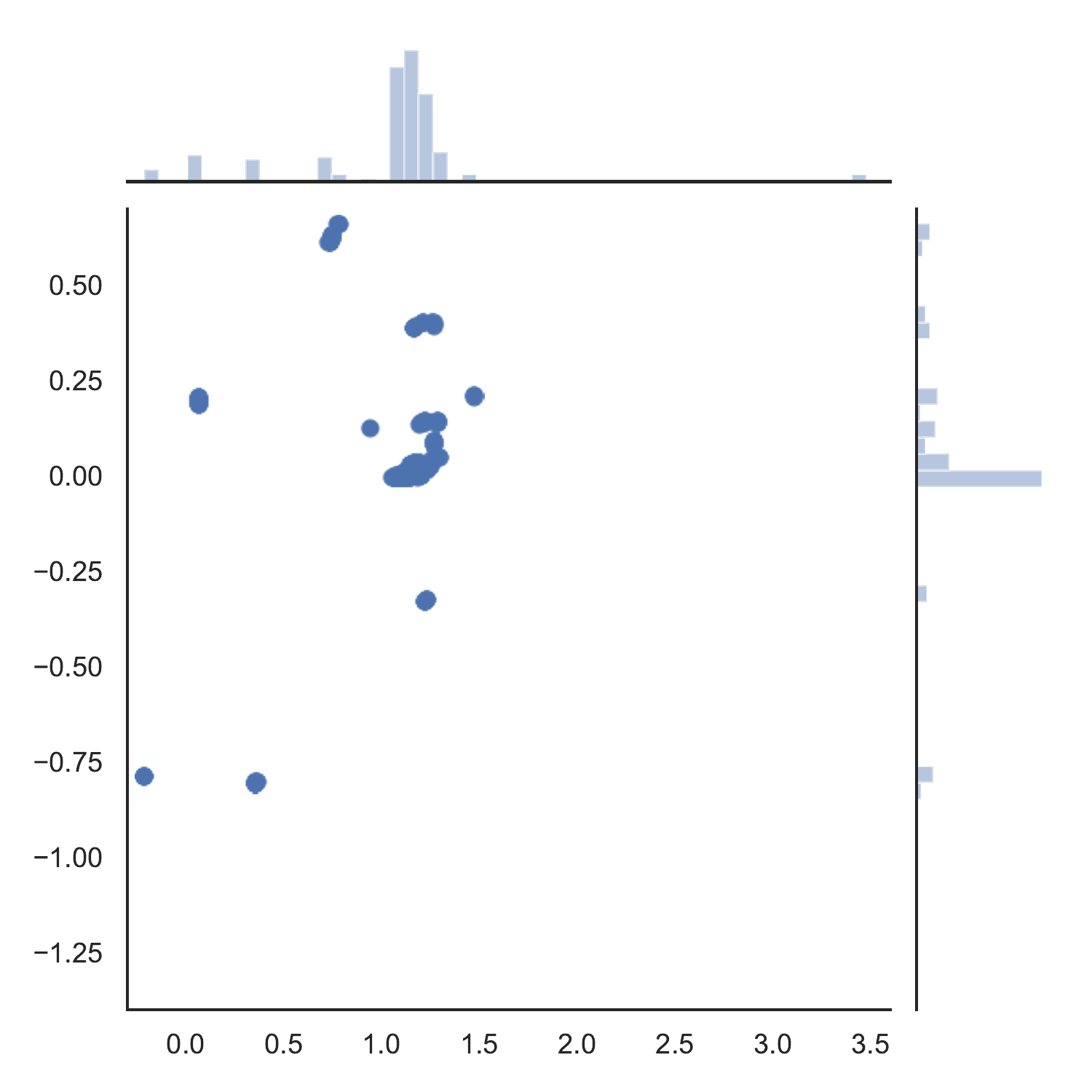} &\includegraphics[width=1.1\linewidth]{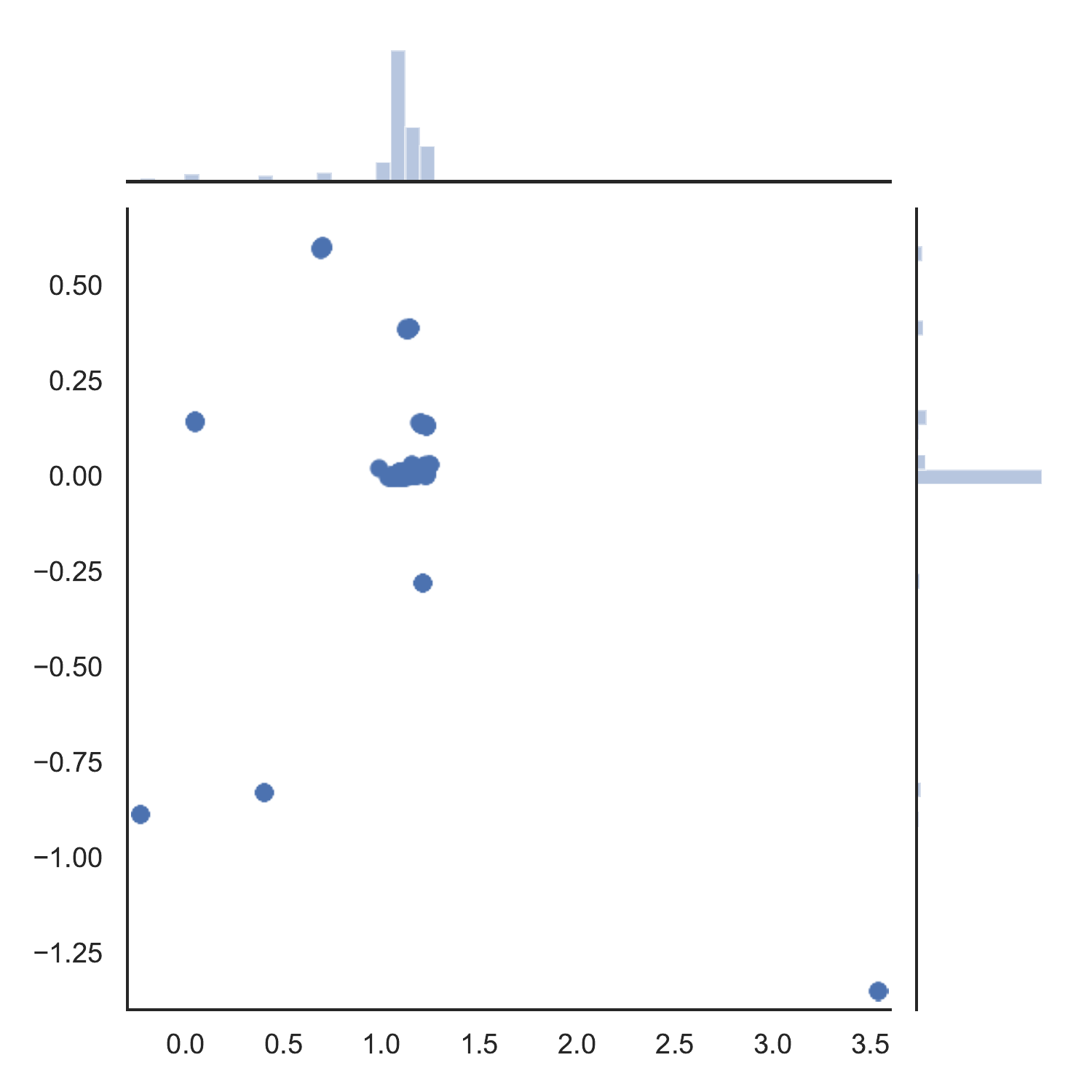} \\ \hline
$k=2$ &\includegraphics[width=1.1\linewidth]{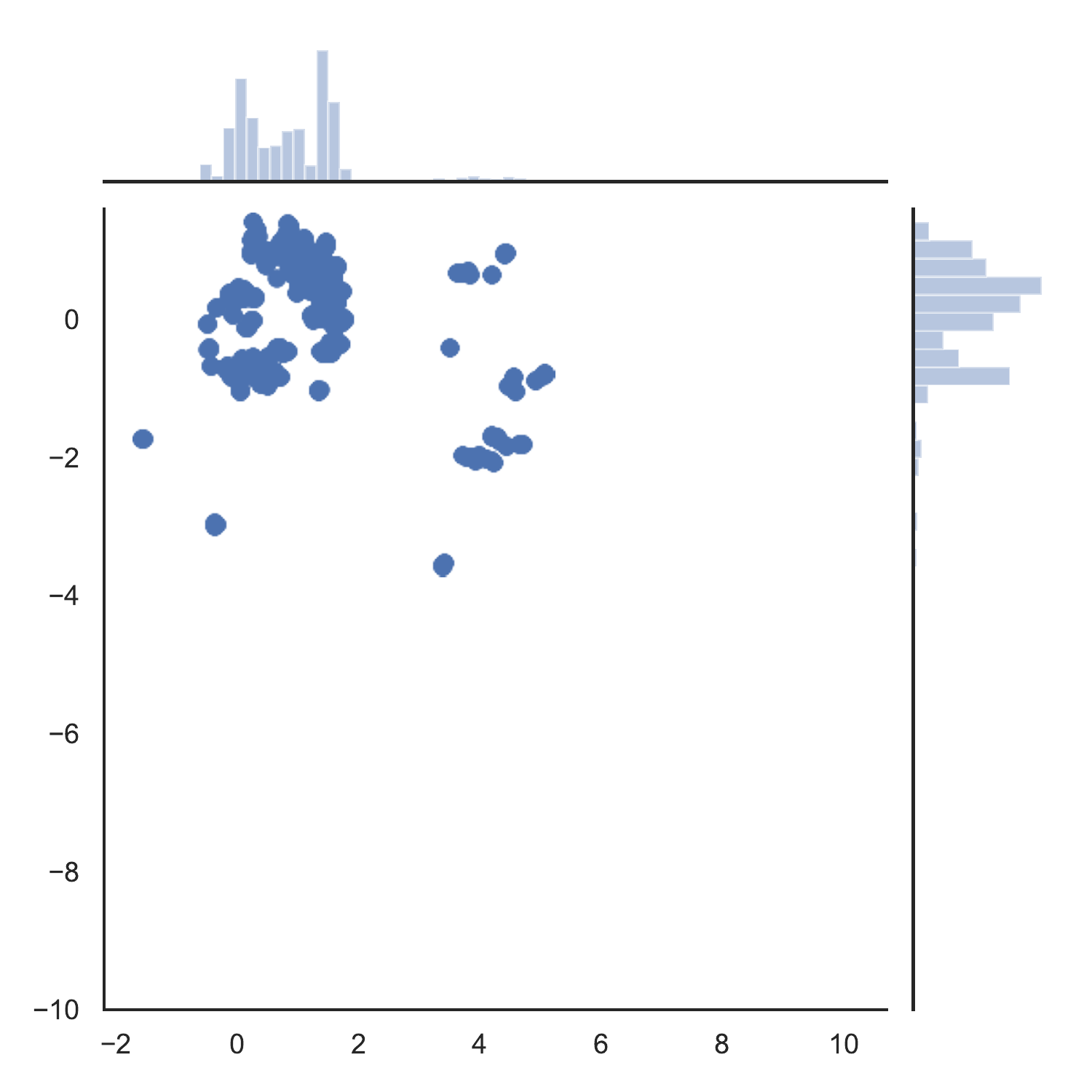}&\includegraphics[width=1.1\linewidth]{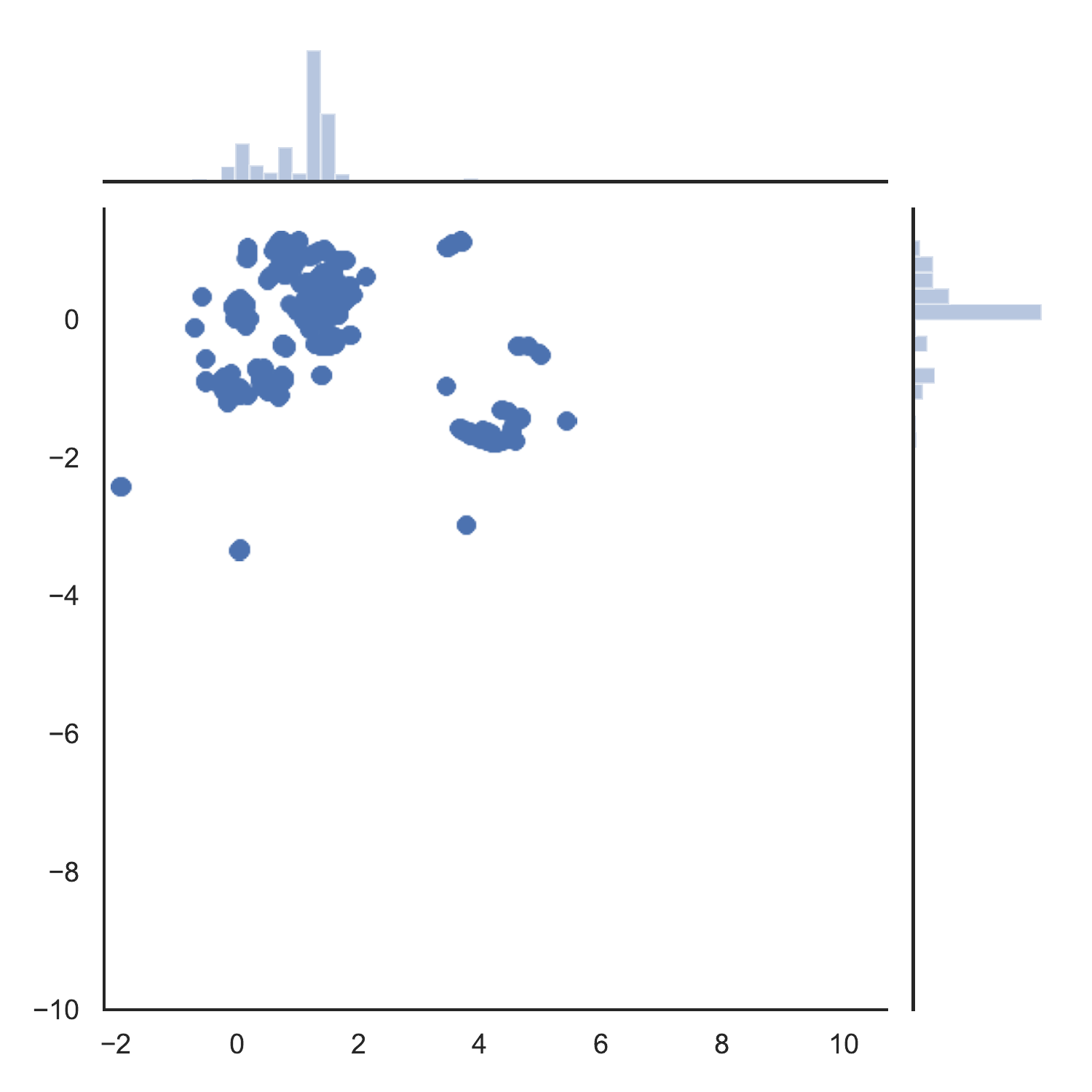}&\includegraphics[width=1.1\linewidth]{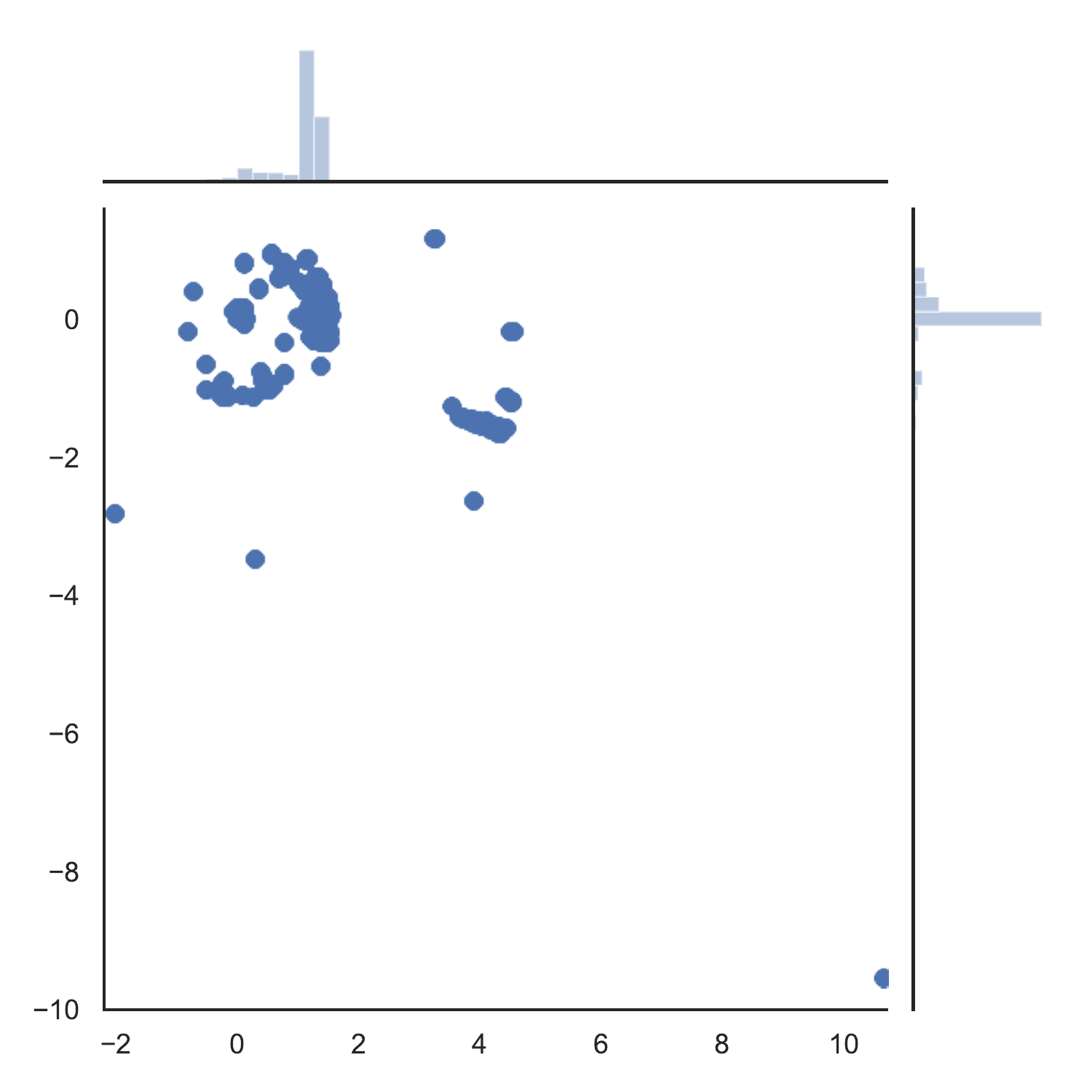}  \\ \hline
 \end{tabular}
\caption{HF case: Eigenvalues distribution dependence on $L$ for the resulting preconditioned matrix $(\bM^{-1} \bA \bM^{-1} \bA)^{(k)}$, $k=1,2$.}
\label{tab:Eigvals}
\end{center} 
\end{table}

To conclude, the HF case shows that the limiting step of the CT is decisively the solver step, and justifies even more the use of efficient preconditioning techniques. In addition, the tensor operator structure of the CT and its speedup with hierarchical matrices allow for limited memory requirements for the impedance matrices and matrix-matrix product function to the number of dofs of the subsystems. These results motivate the use of hierarchical matrices to describe both unknown and right-hand side, in order to reduce matrix-matrix product execution times (\textit{cf.}~\cite{dolz2017covariance}).
\newpage
\subsection{Real case: Non-smooth domain}
\label{subsec:Non_Smooth}
To finish, we compare FOSB with Monte Carlo simulations for a complex case: the sound-soft scattering by a unit Fichera Cube with $\kappa = 5$. We perturb the boundary face located at the $z=0.5$-plane --represented in red in \Cref{fig:Volume Plot} later on-- and use $\IP^0$ elements. We set $L_0=0$ and $L=2$ and generate a sequence of nested meshes associated with discrete spaces $X_0$, $X_1$, $X_2$ of cardinality $N_0=1140$, $N_1 = 2804$ and $N_2=6370$. The zeroth level corresponds to a precision of $10$ elements per wavelength and related meshes are found in \Cref{fig:Meshes}.
\begin{figure}[t]
\begin{minipage}{0.32\linewidth}
\centering
\includegraphics[width=.7\linewidth]{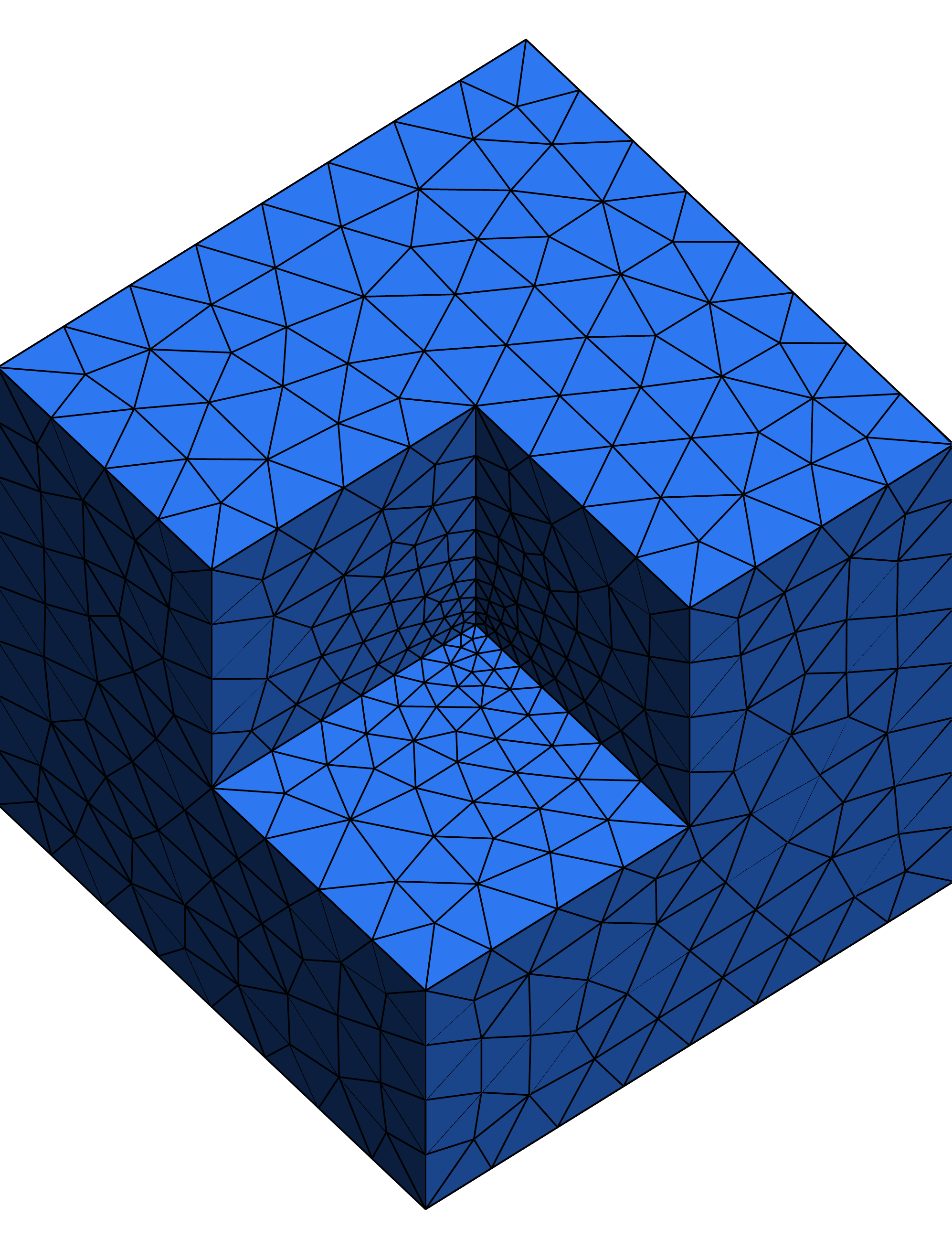}
\end{minipage}
\begin{minipage}{0.32\linewidth}
\centering
\includegraphics[width=.7\linewidth]{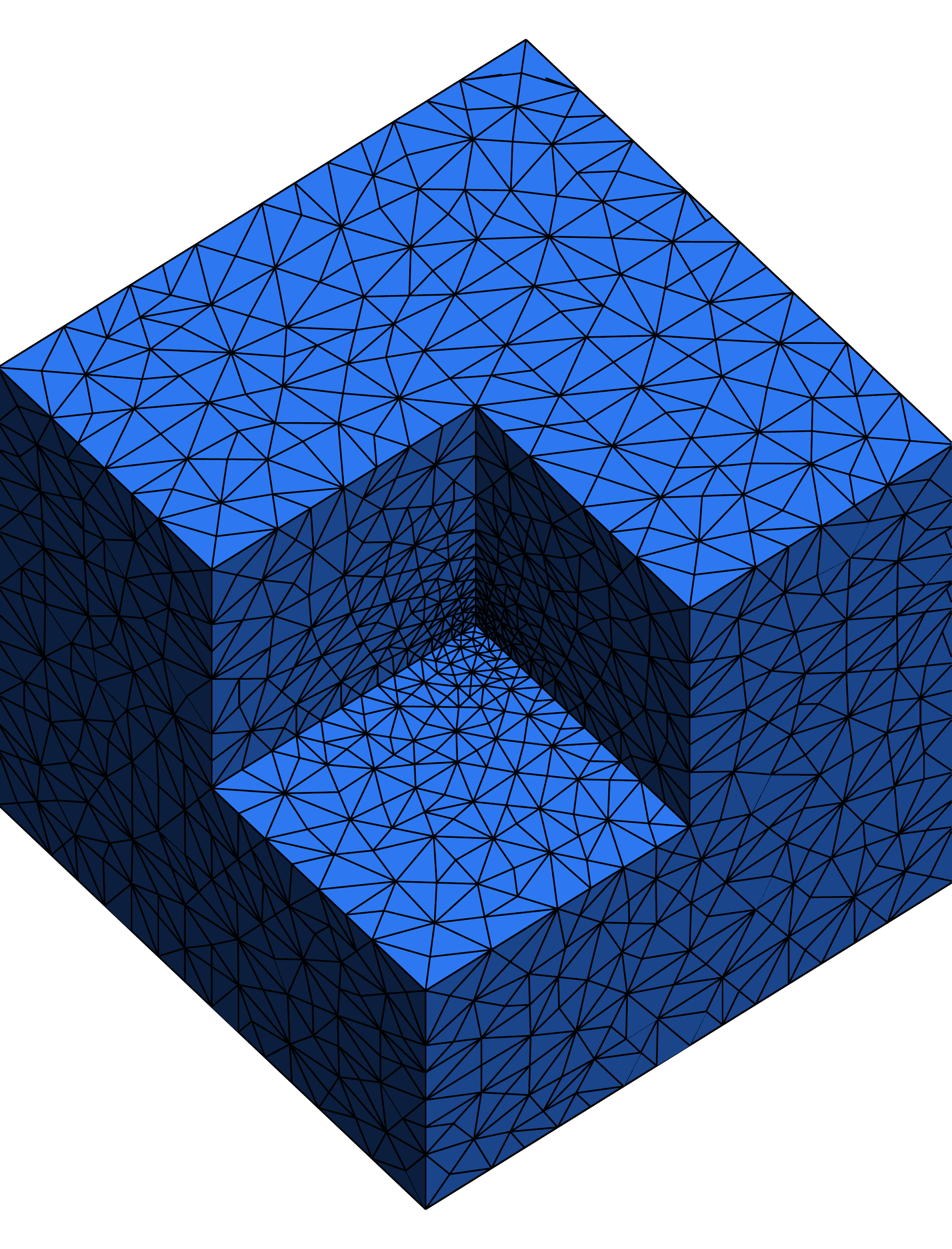}
\end{minipage}
\begin{minipage}{0.32\linewidth}
\centering
\includegraphics[width=.7\linewidth]{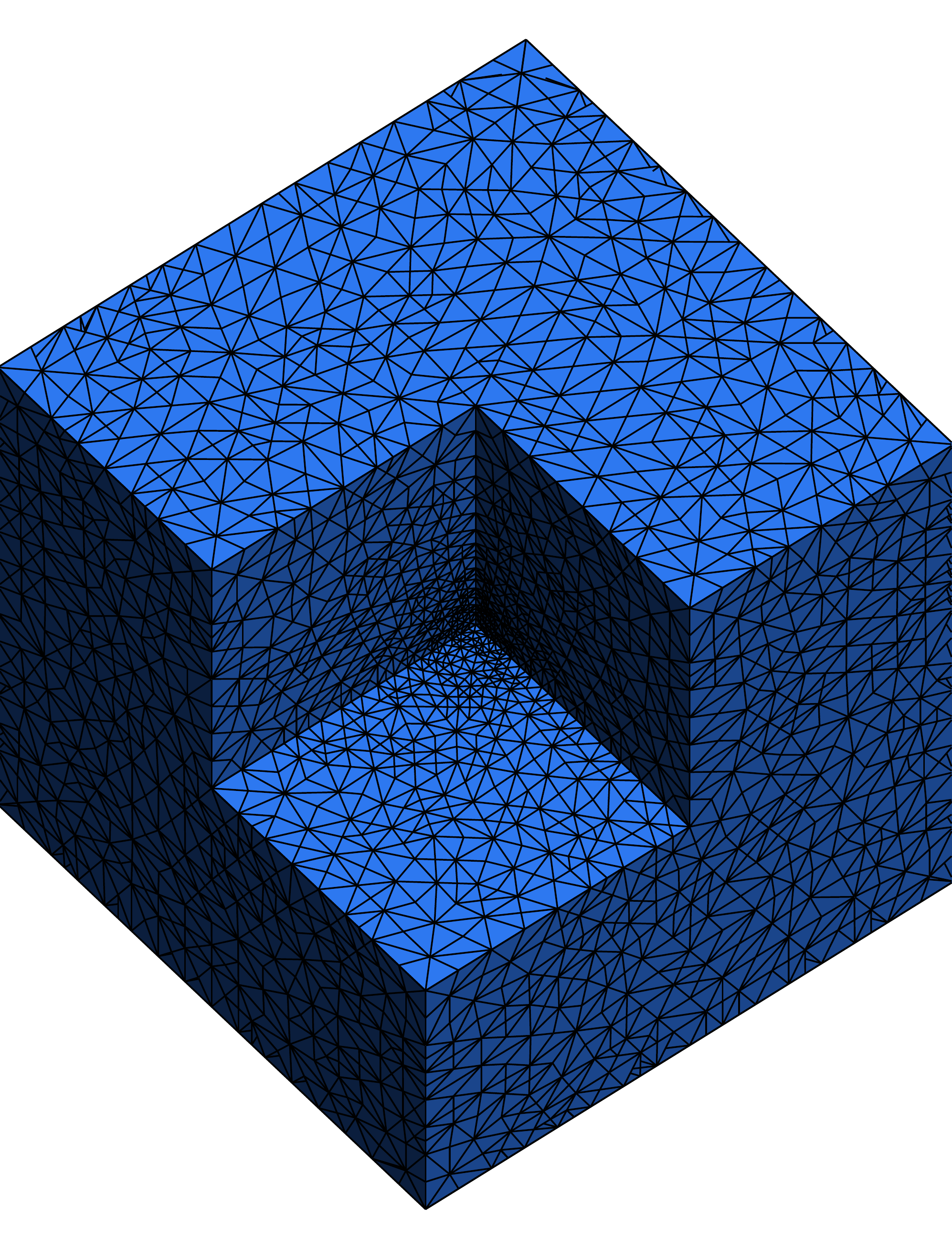}
\end{minipage}
\caption{Sequence of nested meshes used to perform the FOSB.}
\label{fig:Meshes}
\end{figure}

\begin{figure}[t]
\begin{minipage}{0.32\linewidth}
\centering
\includegraphics[width=.8\linewidth]{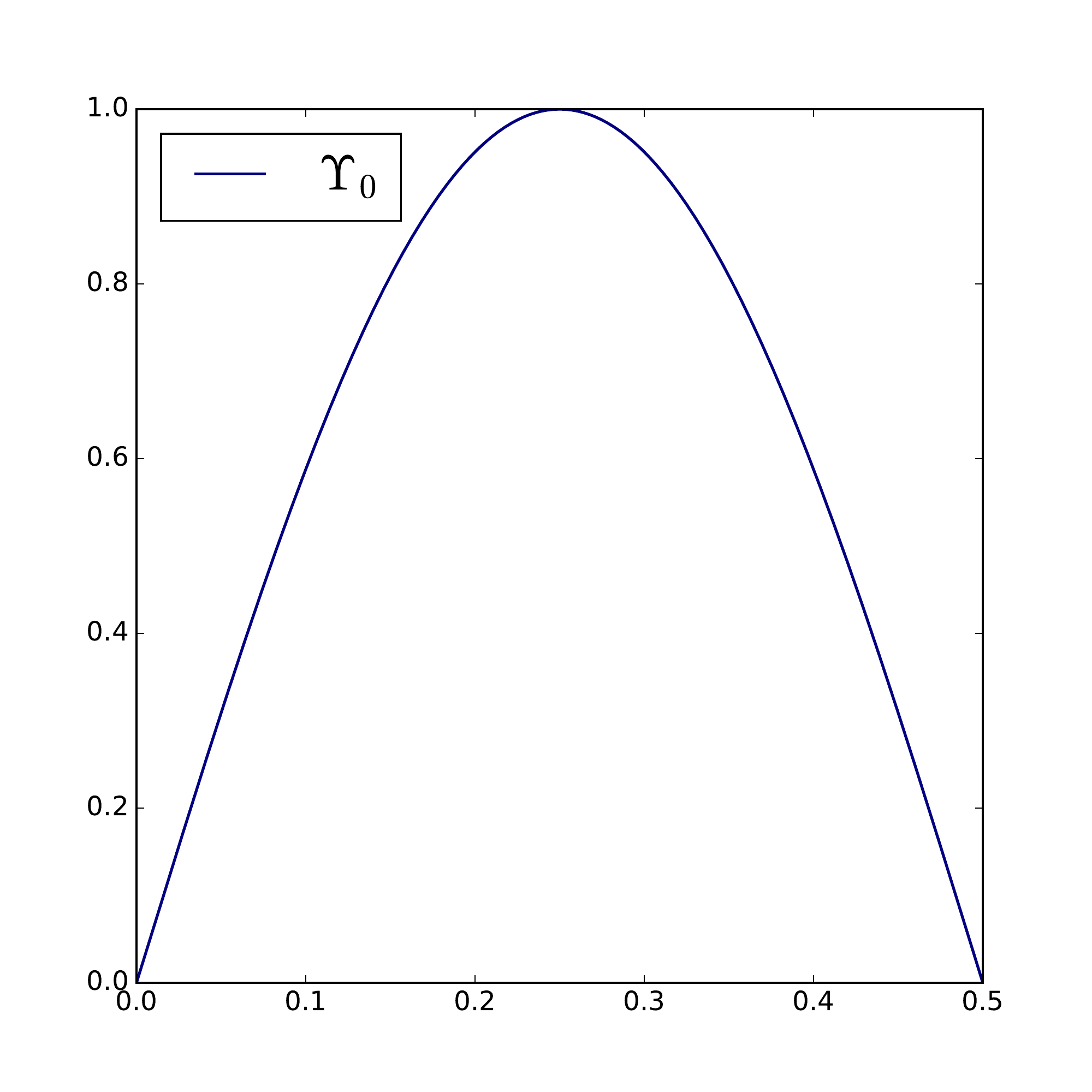}
\end{minipage}
\begin{minipage}{0.32\linewidth}
\centering
\includegraphics[width=.8\linewidth]{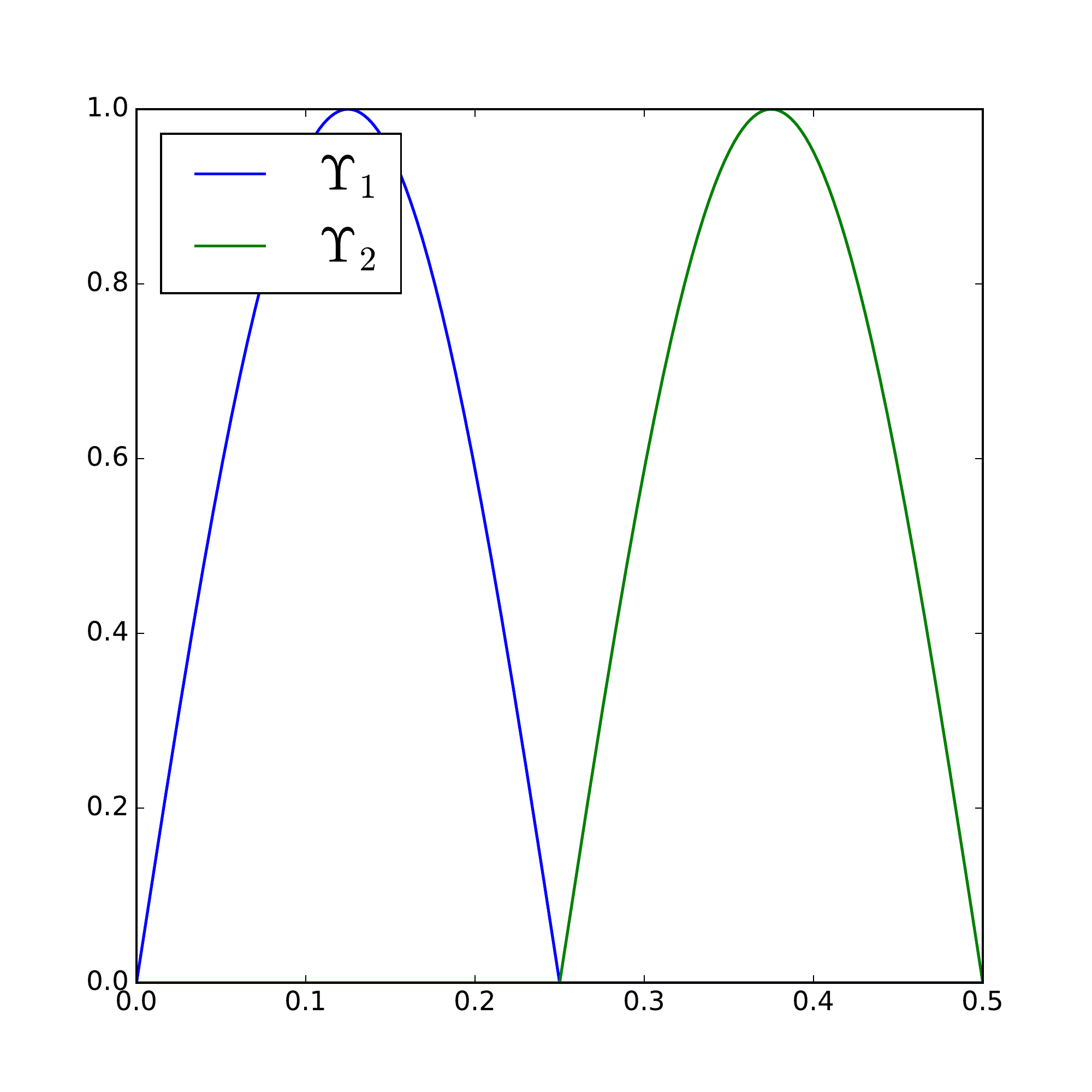}
\end{minipage}
\begin{minipage}{0.32\linewidth}
\centering
\includegraphics[width=.8\linewidth]{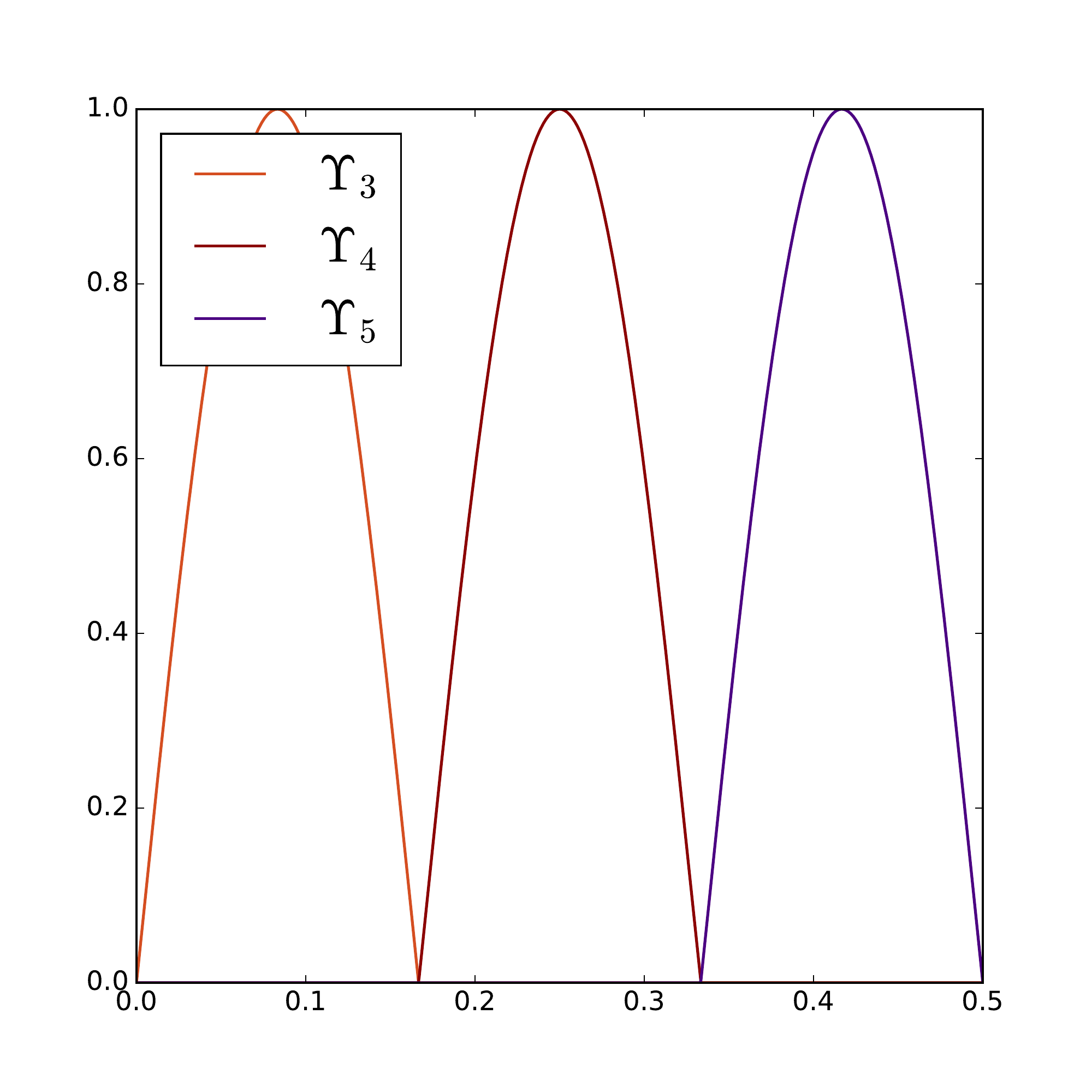}
\end{minipage}
\caption{Splines sinusoidal functions used for random families of perturbed boundaries}
\label{fig:splines}
\end{figure}

Given uniformly distributed random variables $Y_{ij} \in \mU[-1,1], i=0,\cdots,5$, the perturbation field is given as:
\be
\label{eq:peturbation}\bv(\bx,\omega):= \sum_{i=0}^5\sum_{j=0}^5 \Upsilon_i(x)\Upsilon_j(y) Y_{ij} \hat{\bee}_z,~ \bx \in \Gamma, ~ z = 0.5,
\ee
with $\Upsilon_i$ denoting fundamental sine splines of the form $|\sin (q\pi x)|$, $x \in [0,0.5]$, $q\in\{2,4,6\}$ with support of length $0.5 /(q+1)$ represented in \Cref{fig:splines}. Therefore, for $\bx_1$ and $\bx_2$ in $\Gamma$, and $z_1=z_2=0.5$, we have
\be
\mM^{k}[\bv\cdot \bn](\bx_1,\bx_2) = \sum_{i=0}^5\sum_{j=0}^5 \frac{1}{3}\Upsilon_i(x_1)\Upsilon_j(y_1)\Upsilon_i(x_2)\Upsilon_j(y_2).
\ee
The perturbation parameter is set to $t=0.075$. As a reference, we compute the mean field and variance through Monte Carlo method with $M=5000$ simulations (see  \cite[Section 5]{SAJ18}). For each realization, we deform the mesh corresponding to level $L=2$ and obtain $\U(\omega_m)$. Next, we evaluate:
\be
\U^\textup{MC} : = \frac{1}{M}\sum_{m=1}^M \U (\omega_m) \textup{, and } \IV^\textup{MC} := \frac{1}{M} \sum_{m=1}^M (\U(\omega_m)-\U^\textup{MC}) \overline{ (\U(\omega_m)-\U^\textup{MC})}.
\ee
In \Cref{fig:FicheraPerturbed}, we plot mesh elements corresponding to $z=0.5$ for the nominal shape (in red). Therefore, we plot deformed mesh issued from realizations of the perturbation field (in blue). 
\begin{figure}[t]
\begin{minipage}{0.19\linewidth}
\includegraphics[width=1\linewidth]{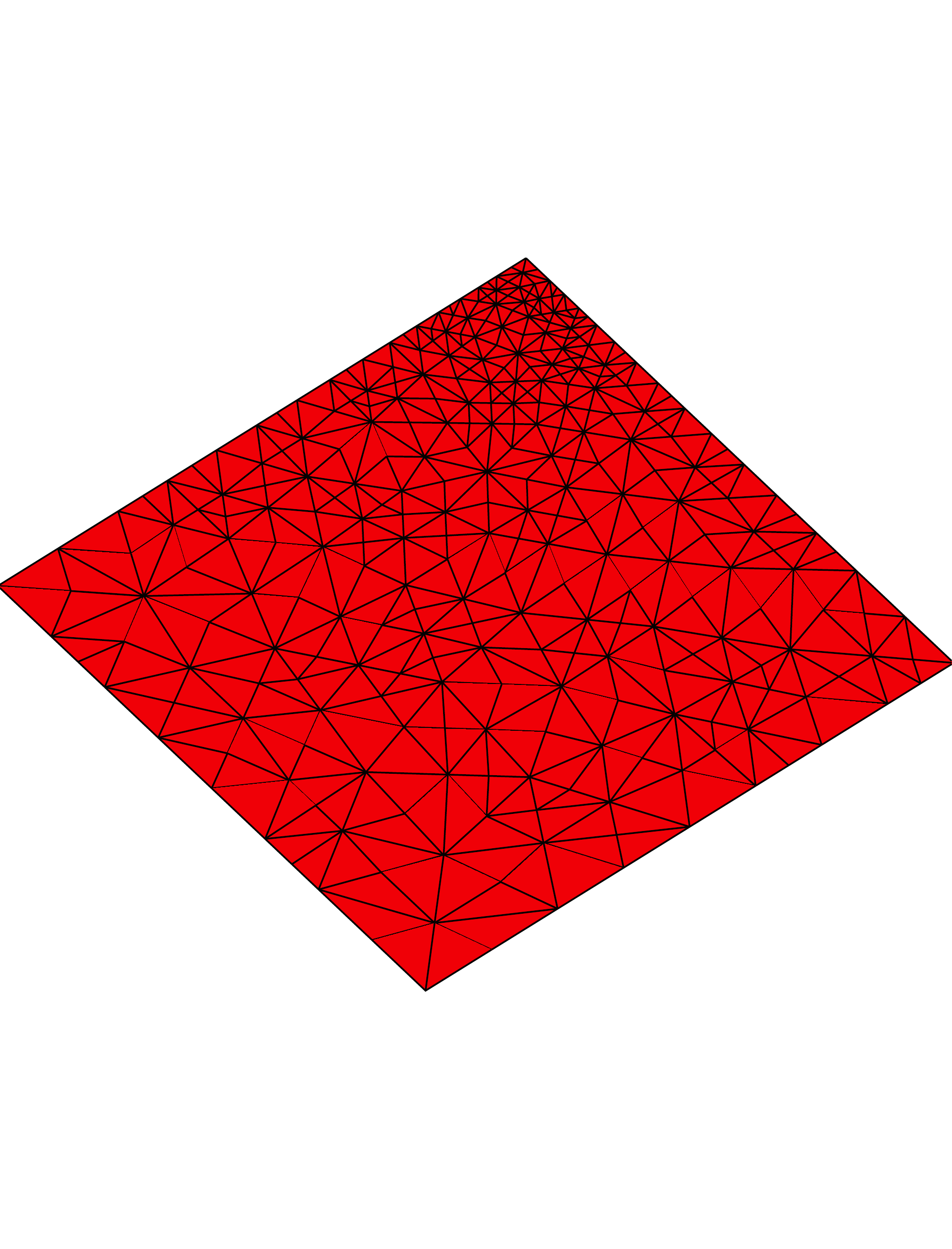}
\end{minipage}
\begin{minipage}{0.19\linewidth}
\includegraphics[width=1\linewidth]{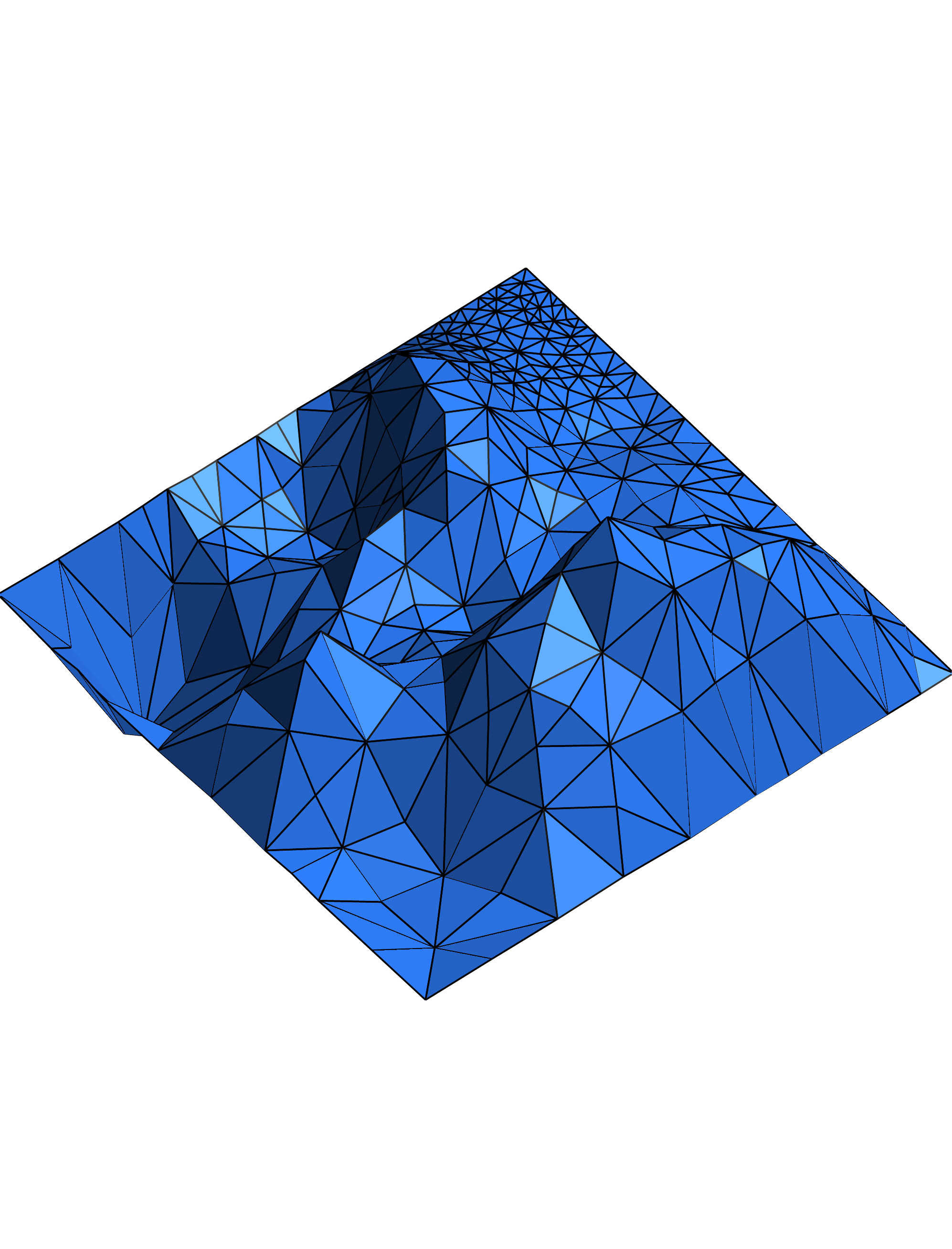}
\end{minipage}
\begin{minipage}{0.19\linewidth}
\vspace{0.4cm}
\includegraphics[width=1\linewidth]{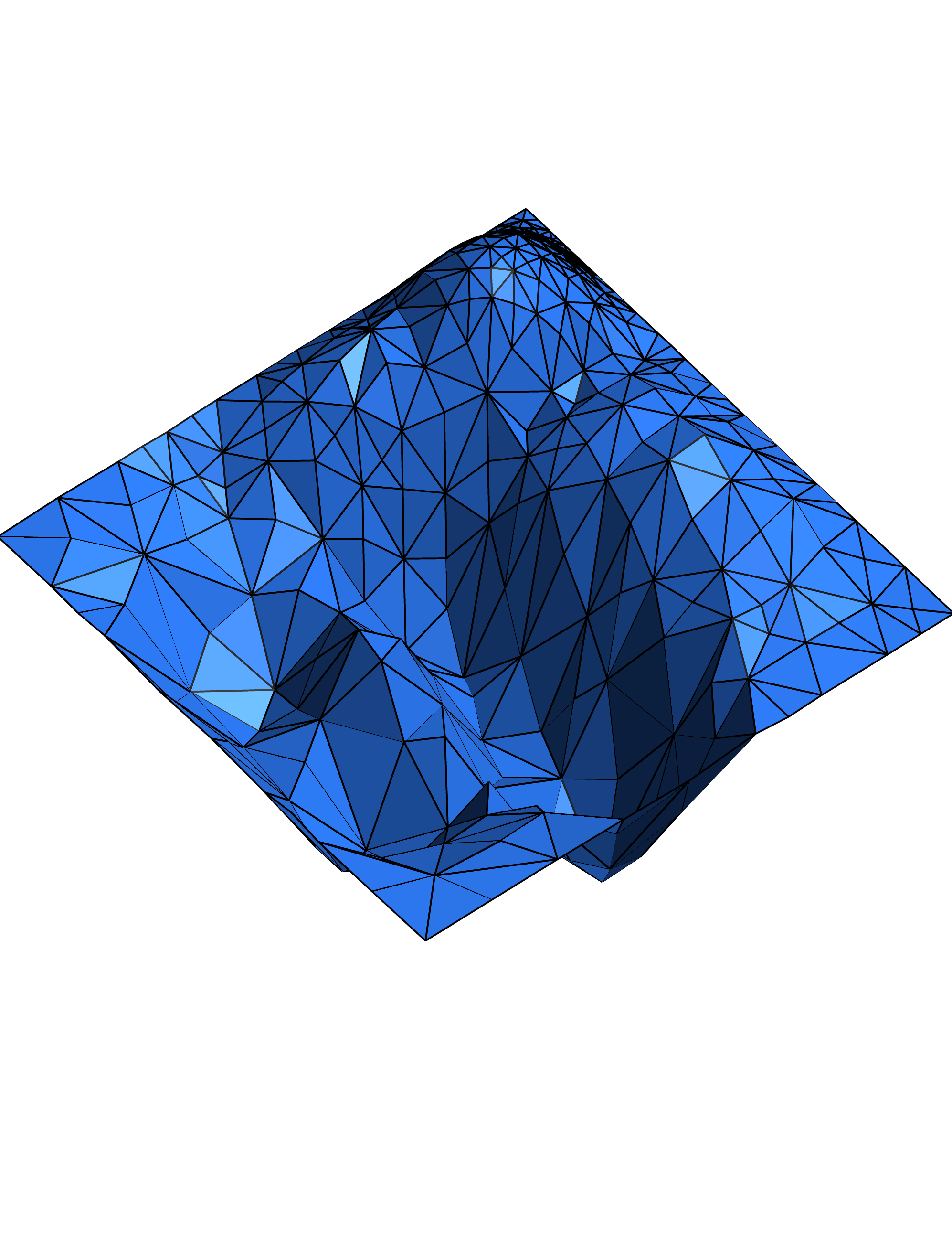}
\end{minipage}
\begin{minipage}{0.19\linewidth}
\includegraphics[width=1\linewidth]{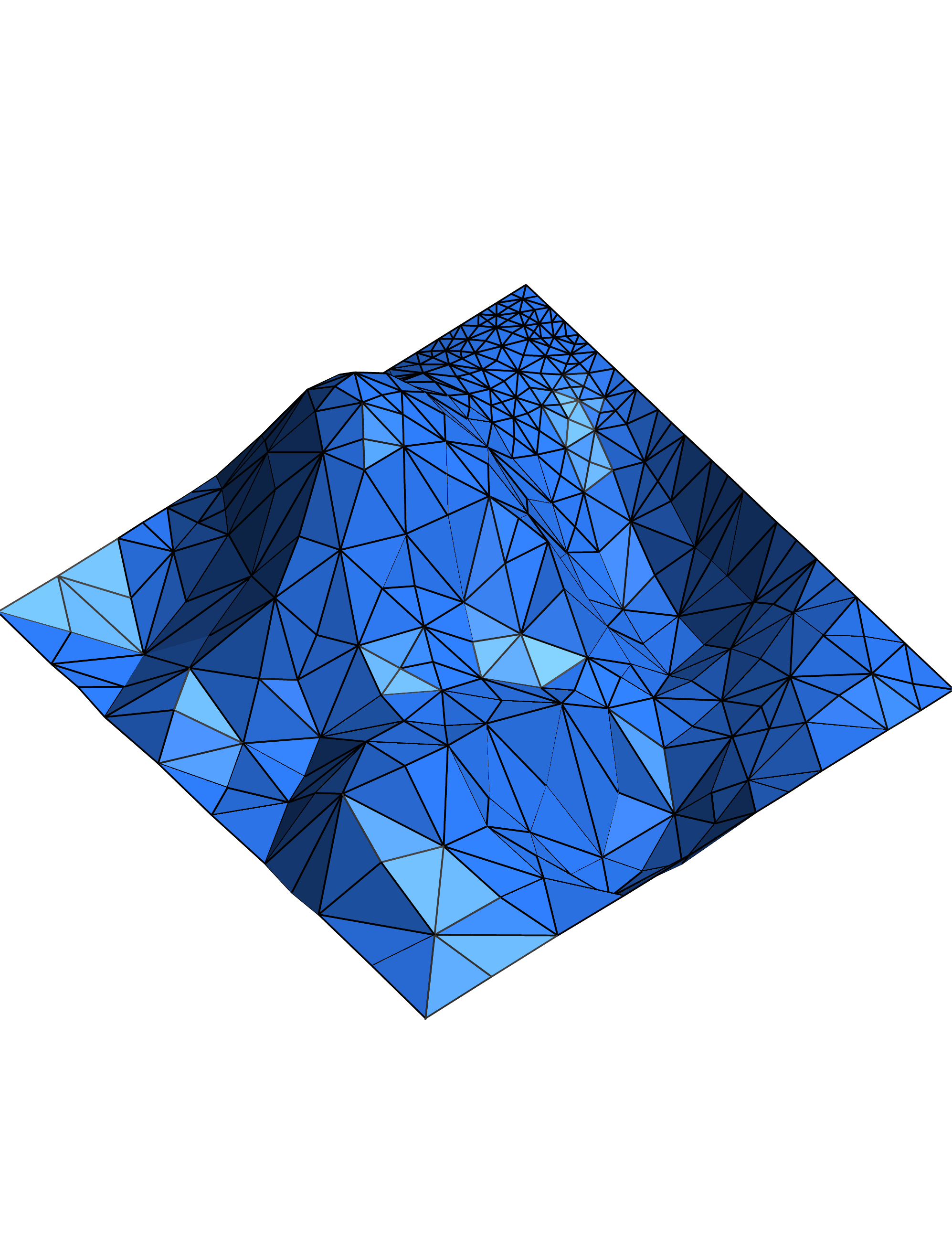}
\end{minipage}
\begin{minipage}{0.19\linewidth}
\includegraphics[width=1\linewidth]{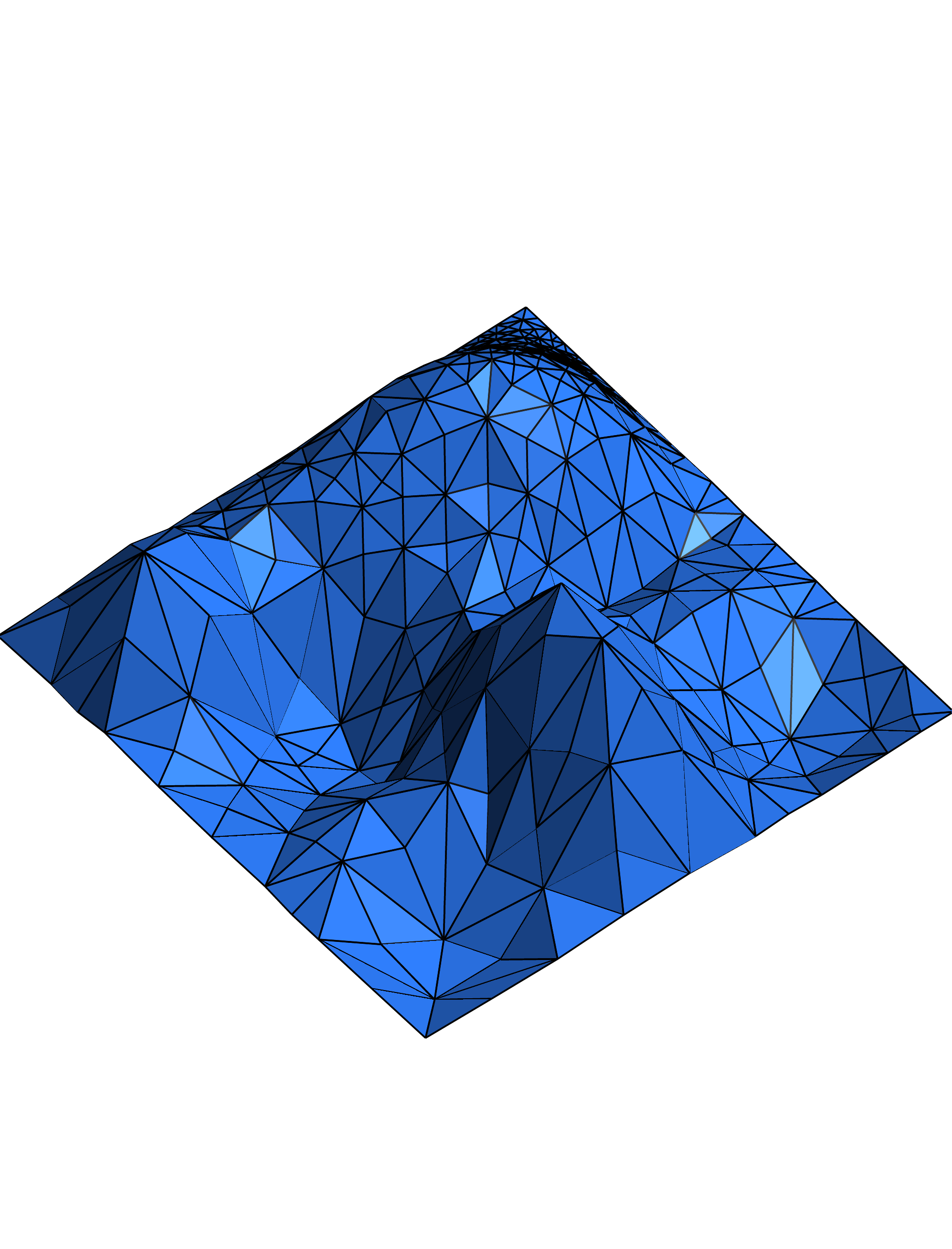}
\end{minipage}
\caption{Nominal mesh (red) and transformed meshes corresponding to realizations of MC simulation (blue).}
\label{fig:FicheraPerturbed}
\end{figure}

For the implementation of a symmetric FOSB, we use an indirect formulation \cite{sauter} for the tensor equation and choose a Near-Field preconditioner \cite{IEEE_EJH} as it outperformed the Multiplicative Calder\'on preconditioner in solution times. In \Cref{fig:Volume Plot}, we plot the total squared density of $\U$ (left) and the standard deviation --square root variance-- $\sqrt{\hat{\IV}}$ (right). We remark that the area close to the perturbed boundary has a higher variance while the zone behind the Fichera Cube has low sensibility to shape variation.

\begin{figure}[t]
\begin{minipage}{0.49\linewidth}
\centering
\includegraphics[width=.8\linewidth]{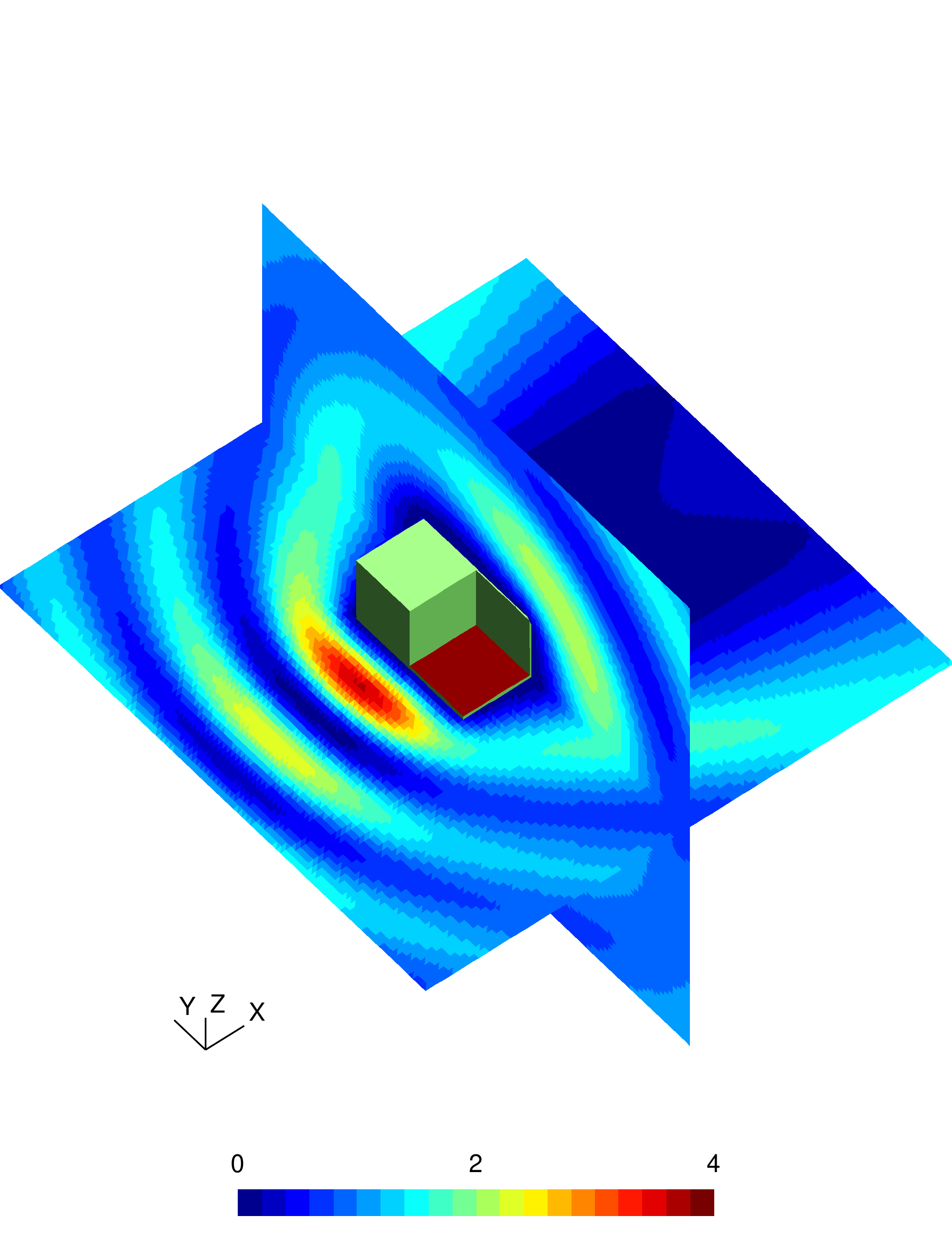}
\end{minipage}
\begin{minipage}{0.49\linewidth}
\centering
\includegraphics[width=.8\linewidth]{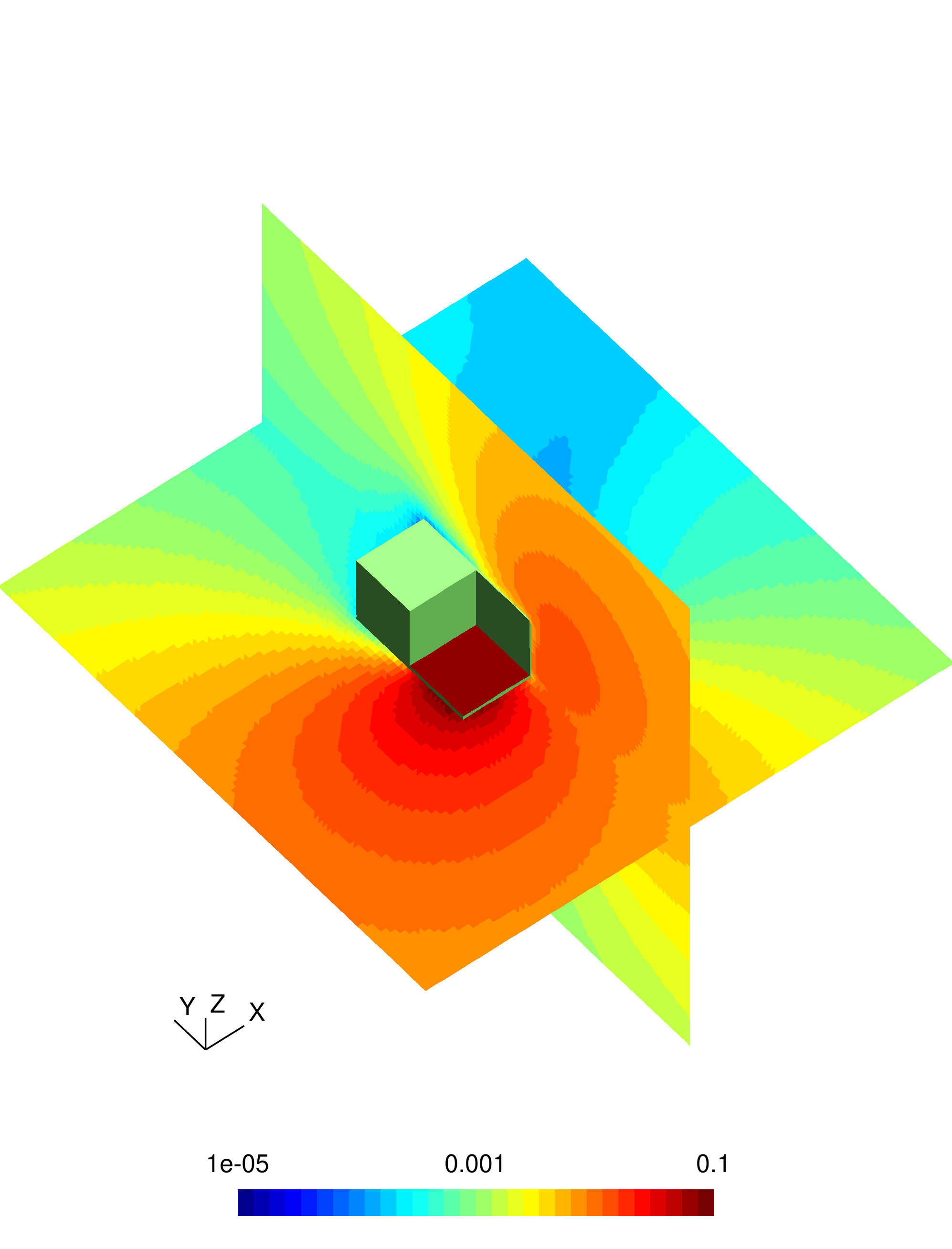}
\end{minipage}
\caption{Volume plot of the squared density for $\U$ (left) and the standard deviation $\sqrt{\hat{\IV}}$ obtained through the FOSB method.}
\label{fig:Volume Plot}
\end{figure}
In \Cref{tab:ResComp_Fichera} we compare the RCSs provided by both methods, namely MC (left column) and FOSB (right column) --we inspire ourselves of the plots in \cite{harbrecht2019rapid}. In the first row, we represent the approximation of the mean field (red) and its sensibility (blue). Second row focuses on RCS for the squared root variance. We remark that both methods show similar patterns. Indeed, the relative $L^2$-error on $\IS^1$ between the FFs differ by a $11.0\%$ (resp.~$18.4\%$) for $\U^\textup{MC}$ and $\U$ (resp.~$\sqrt{\IV^\textup{MC}}$ and $\sqrt{\hat{\IV}}$), evidencing accuracy of the FOSB scheme. The latter is interesting as it shows that the FOA behaves well for domains with corners, albeit lacking theoretical results on it. Finally, the symmetric CT led to a total $dofs=18,420,424$ compared to $N_L^2 = 41,088,100$ for the full tensor space. As a comparison, $M N_L = 38,460,000$ for MC. The total execution time for MC was $13$ hours $26$ min.~compared to $6$ hours $19$ min.~for the FOSB.

\begin{table}[t]
\renewcommand\arraystretch{1.7}
\begin{center}
\footnotesize
\resizebox{9cm}{!} {
\begin{tabular}{
    >{\centering\arraybackslash}m{5.4cm}
    |>{\centering\arraybackslash}m{5.4cm}
    }
\vspace{0.1cm}
  Monte Carlo &FOSB \\ \hline
 \includegraphics[width=\linewidth]{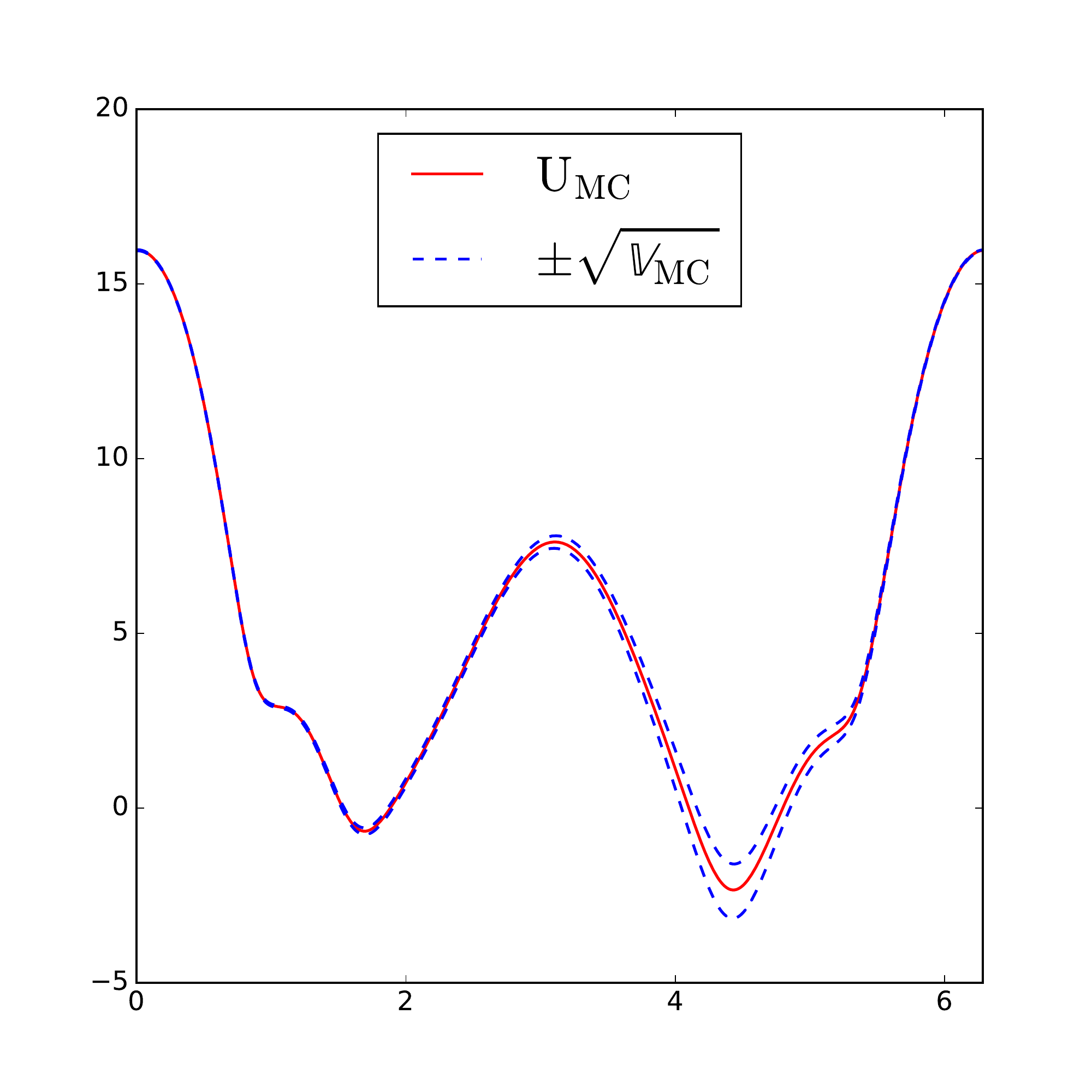} &  \includegraphics[width=\linewidth]{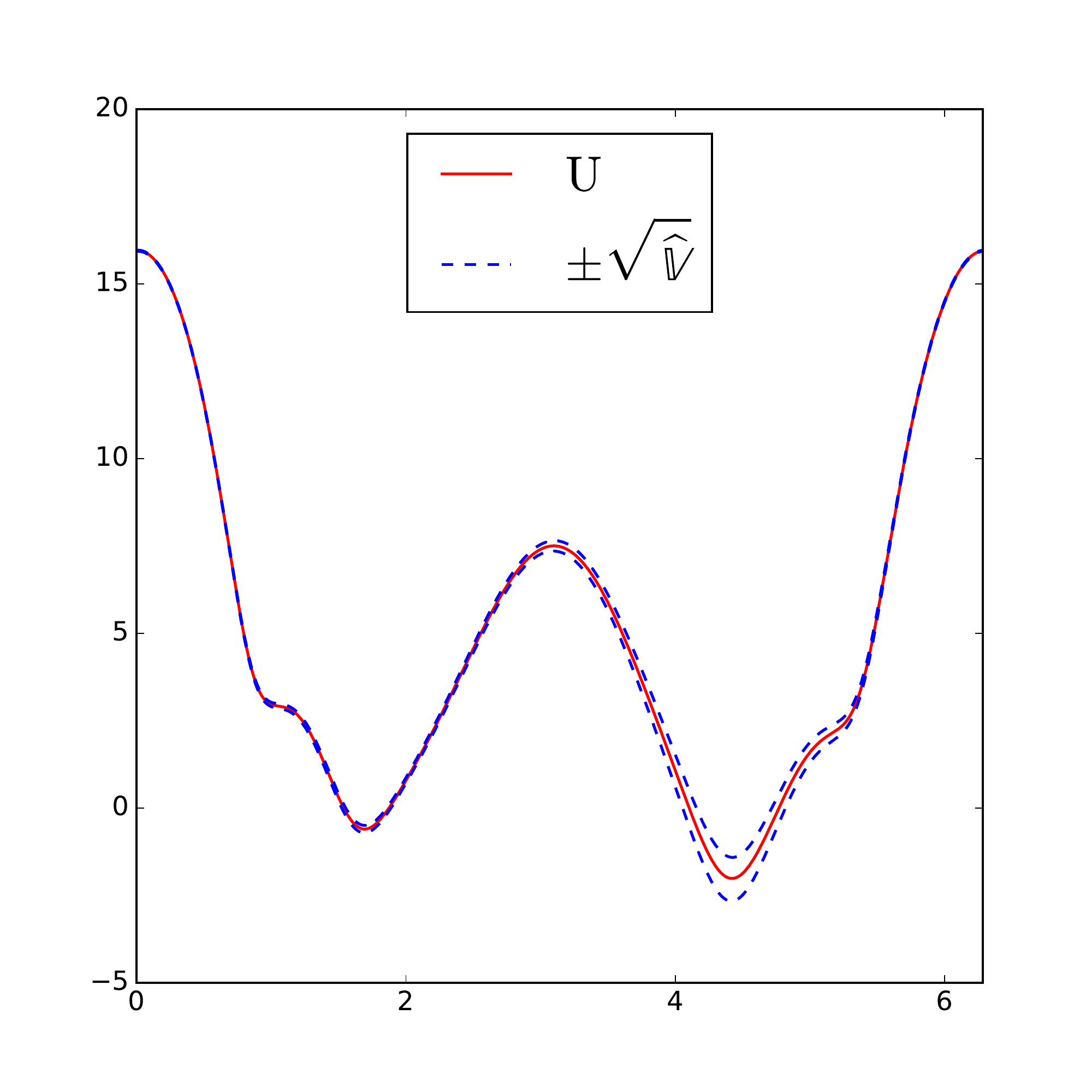}\\ \hline
 \includegraphics[width=\linewidth]{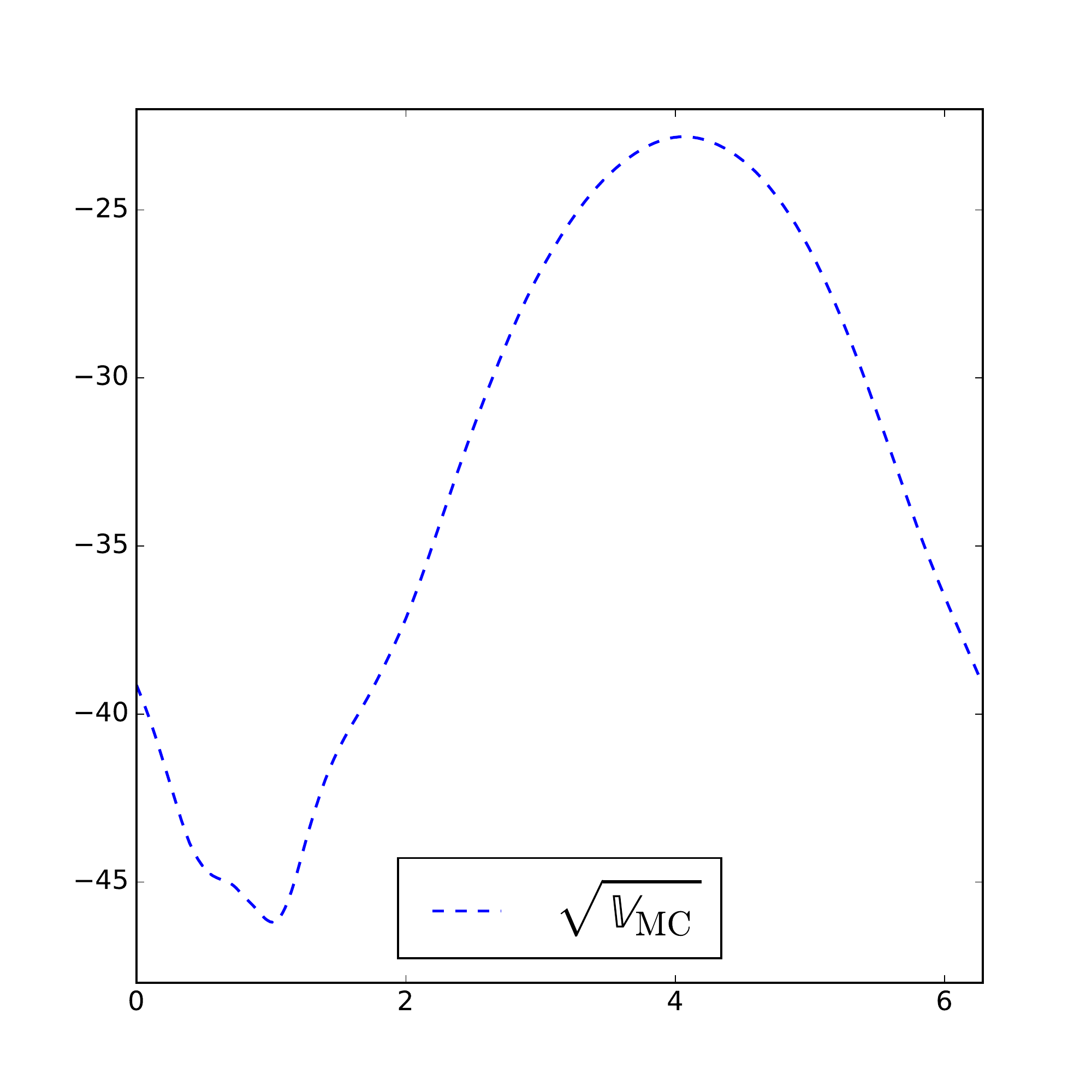} &  \includegraphics[width=\linewidth]{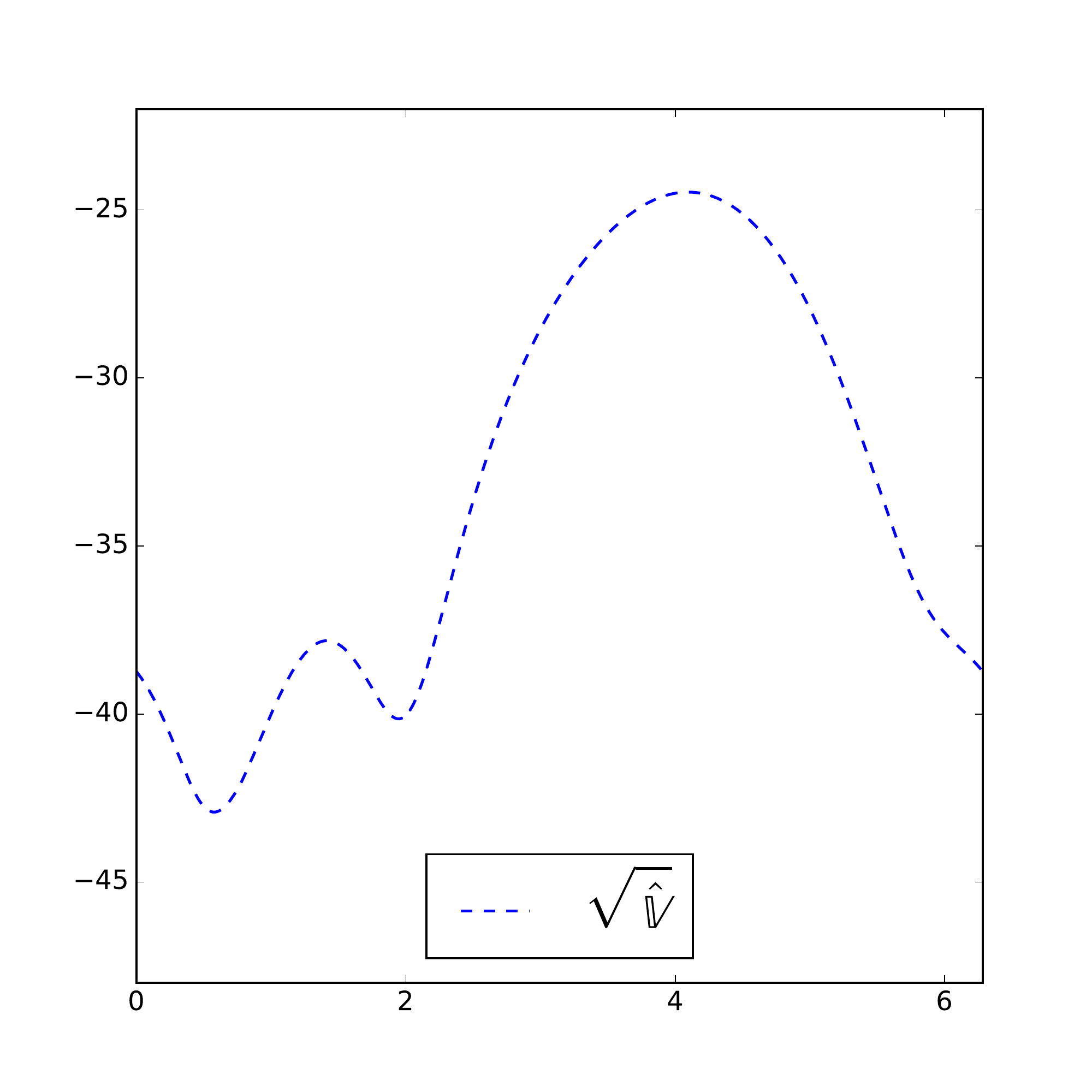}\\ \hline
 \end{tabular}}
\caption{Final comparative results between the MC (left) and FOSB (right) methods. First row shows the approximation for the mean RCS (red) and its standard deviation (blue) while second rows focuses on the standard deviation. RCSs are in represented (dB) versus the angle ($\theta+\pi$) in radians.} 
\label{tab:ResComp_Fichera}
\end{center} 
\end{table}

\section{Conclusion}
\label{sec:conclusion}
In this work, we tackled UQ for random shape Helmholtz scattering problem. Under small perturbation assumptions, we applied the FOSB method and allowed for an accurate approximation of statistical moments with an almost optimal memory and computational requirements. We provided the complete analysis for Helmholtz boundary value problems and added comments concerning the efficient implementation of the schemes. Numerical experiments evidenced the applicability of the technique and showed good scalability and robustness when coupled with fast resolution methods and efficient preconditioners. Observe that theory presented in \Cref{sec:galerkin_CT} and \Cref{subs:prec} and numerical results of \Cref{subsec:TP,subsec:TP_Iterative} are of interest for general Helmholtz-based tensor operators BIEs, as they are developed aside from the FOA framework. Conversely, the FOA and the numerical results of \Cref{subsec:SSSH} apply to low-rank approximation-based schemes to solve the deterministic equation \cite{dambrine2015computing}.

Further research includes asymptotic wavenumber analysis of each specific boundary condition and under additional requirements would lead in some cases to elliptic formulations, simplifying greatly the Galerkin scheme --C\'ea's lemma-- and the sparse tensor approximation. We hope that the analysis carried on in \cite{galkowski2016wavenumber} can be extended to the FOSB method for the exterior sound-soft problem, and would provide $\kappa$-explicit estimates of the constants involved in the scheme along with bounds for the GMRES for both nominal solution and sub-blocks for the CT.

In parallel, this work suggests the use of more efficient tools such as: (i) multilevel matrix-matrix product and compression techniques for the covariance kernel (e.g., hierarchical matrices or low-rank approximations), (ii) efficient iterative solvers for multiple right-hand sides \cite{sun2018variants}, and (iii) fast preconditioning techniques \cite{IEEE_EJH}. Those improvements would allow to compute higher moments in a satisfactory number of operations. Also as a by-product of our study we have rendered available the Bempp-UQ plug-in for further improvement and usage. Our code currently supports the $\IP^1$ projection between grids, the tensor GMRES, the CT for $k=2$ and the FOA for all problems considered here. Current work seeks to implement the FOA for Maxwell equations; speed up the preconditioner matrix assembly for both Helmholtz and Maxwell cases; and incorporate high-order quadrature rules routines for UQ purposes.
\appendix
\section{Boundary reduction}
\label{appendix:A}
\subsection{\texorpdfstring{Case: \textnormal{(P}$_\boldbeta$\textnormal{)}}{Case: (P$_\beta$}}
Set $\fA_i \equiv \fA_{\kappa_i}$ and recall the following identity:
\begin{lemma}
\label{lemma:uinc}
Let $\U$ be the solution of problem (P$_\boldbeta$) with $\bxi^\textup{inc}:=(\gamma_0\U^\textup{inc},\gamma_1\U^\textup{inc})$. Thus, for $\bxi:=(\gamma_0 \U^0 ,\gamma_1 \U^0)$, we have
\be
\left(\frac{1}{2} \opI + \opA_0\right) \bxi = \bxi^\textup{inc}.
\ee
\end{lemma}
\begin{proof}
Consider $\U$ the solution of (P$_\boldbeta$). Since $\U^\textup{inc}$ provides admissible Cauchy data inside the scatterer and $\U^\textup{scat}$ is a radiative Helmholtz equation, the following identities hold:
\be
\left(\frac{1}{2} \opI - \opA_0\right) \bxi^\textup{scat}= \bxi^\textup{scat}\textup{, and }
\left(\frac{1}{2} \opI + \opA_0\right) \bxi^\textup{inc} = \bxi^\textup{inc}.
\ee
Summing both equations, we get,
\begin{align*}
\opA_0(2\bxi^\textup{inc} -\bxi)=\frac{1}{2}  \bxi &\Rightarrow \left( \opA_0 + \frac{1}{2} \opI\right) \bxi =  2 \opA_0 \bxi^\textup{inc} = \bxi^\textup{inc}.
\end{align*}\end{proof}
This last result allows us to determine directly the BIEs of $\U$ from the BCs for $\beta= 0,1,3$. Let us focus on (P$_2$). Second row of \Cref{lemma:uinc} and the BCs rewrite:
\be
\left(\opW_\kappa  - \imath \eta \left(\frac{1}{2}\opI + \opK'_\kappa \right) \right)\gamma_0 \U = \gamma_1 \U^\textup{inc}.
\ee
Next, the integral representation formula (\ref{eq:representation_formula}) gives:
\begin{align*}
\U^\textup{scat}&= - \SL_{\kappa}(\gamma_1 \U^\textup{scat}) + \DL_{\kappa}(\gamma_0 \U^\textup{scat} ) \textup{ in } D^c \textup{, and }\\
0 = &- \SL_{\kappa}(\gamma_1 \U^\textup{inc}) + \DL_\kappa(\gamma_0 \U^\textup{inc} ) \textup{ in } D^c.
\end{align*}
Summing both equations yields
\be
\U= \U^\textup{inc}- \SL_{\kappa}(\gamma_1 \U) + \DL_{\kappa}(\gamma_0 \U ) \textup{ in } D^c,
\ee
giving the final results for the expected field. 
\subsection{\texorpdfstring{Case: \textnormal{(SP}$_\boldbeta$\textnormal{)}}{Case: (SP$_\beta$}}
Similarly, boundary reduction for the shape derivative is deduced from the following identities for $\bxi':=\bxi^0$:
\be
\left(\frac{1}{2} \opI - \opA_0\right) \bxi' = \bxi'\textup{, and }
\left(\frac{1}{2} \opI + \opA_1\right) \bxi'^1 = \bxi'^1 \textup{, if } \beta= 3,
\ee
by considering the boundary conditions. In particular, for $\beta=2$:
\begin{align*}
-\opW_\kappa \gamma_0 \U' + \left(\frac{1}{2}\opI - \opK'_\kappa \right) \gamma_1 \U' &= \gamma_1 \U' \\
\Rightarrow\left(\opW_\kappa  - \imath \eta \left(\frac{1}{2}\opI + \opK'_\kappa \right) \right)\gamma_0 \U' &=  \left(\frac{1}{2}\opI + \opK'_\kappa \right) g_2.
\end{align*}
Finally, we detail the potential reconstruction for the transmission problem. Representation formulas \Cref{eq:representation_formula} yield:
\begin{align*} 
{\U'}^{i}  &= -  \SL_{\kappa_i} \gamma_1 \U'^i + \DL_{\kappa_i} \gamma_0 \U'^i\textup{ in } D^i,~i=0,1\\
& = \fR_\kappa(\bxi) \textup{ in } \mD \textup{  (\Cref{eq:compact_potential})}.
\end{align*}
\bibliographystyle{siamplain}
\bibliography{references}

\end{document}